\documentclass[a4paper]{amsart}
\usepackage{graphicx} 
\usepackage{xcolor} 
\usepackage{graphicx}
\usepackage[toc]{appendix} 
\usepackage{amsmath, amssymb}
\usepackage{amsfonts} 
\usepackage{amsthm} 
\usepackage{tikz-cd}
\usepackage{systeme}
\usepackage[top=3cm, bottom=3cm,inner=1.5cm, outer=1.5cm]{geometry}	
\usepackage{color}   
\usepackage{hyperref}
\hypersetup{
    colorlinks=true, 
    linktoc=all,    
    linkcolor=black,  
    citecolor = black,
}

\newtheorem{theorem}{Theorem}[section] 
\newtheorem{lemma}[theorem]{Lemma}

\newtheorem{prop}[theorem]{Proposition}

\numberwithin{equation}{section}

\newcommand{\abs}[1]{\lvert #1 \rvert}
\newcommand{\sumstack}[1]{{\substack{#1}}}
\newcommand{\nsmod}[1]{\text{ mod $#1$}} 
\newcommand{\fq}[0]{\mathbb{F}_q}
\newcommand{\Disc}[0]{\mathrm{Disc}}
\newcommand{\Cl}[0]{\mathrm{Cl}}

\theoremstyle{definition}

\newtheorem*{remark}{Remark}
\title[Non-vanishing of $L$-functions associated with $D_4$-quartic function fields ordered by conductor]{Non-vanishing of Artin $L$-functions associated with $D_4$-quartic\\ function fields ordered by conductor}
\author{Victor Ahlquist}
\address{Department of Mathematical Sciences, Chalmers University of Technology and the University \newline
	\rule[0ex]{0ex}{0ex}\hspace{8pt} of Gothenburg, SE-412 96 Gothenburg, Sweden}
\email{vicahlqu@chalmers.se} 
\date{}
\begin{document}

\begin{abstract}
 We study the low-lying zeros of certain Artin $L$-functions associated with $D_4$-quartic function fields. Specifically, we prove that when ordered by conductor, at least $77\%$ of these $L$-functions are non-vanishing at the central point. This generalises and extends results over $\mathbb{Q}$ due to Durlanik, proving that an infinite number of these $L$-functions are non-vanishing.

We obtain these results by examining the low-lying zeros of the $L$-functions using the one-level density. Specifically, we apply and extend a method used by Rudnick, who studied Dirichlet $L$-functions associated with quadratic function field extensions, to the $D_4$-case. The main difficulty is studying $L$-functions which are associated to $D_4$-fields whose quadratic subfield is of large discriminant. These $L$-functions are studied by utilising the so-called flipped field of a $D_4$ extension, combining a method introduced by Friedrichsen for counting $D_4$-fields, with explicit ramification theory in such fields provided by Altuğ, Shankar, Varma and Wilson.
\end{abstract}
\maketitle
\section{Introduction}
Recent decades have seen much interest in the vanishing, or non-vanishing, of $L$-functions at the central point. In some cases, the degree of vanishing of an $L$-function at the central point is tied to arithmetic data of a corresponding algebraic or geometric object. For example, in the case of $L$-functions associated with elliptic curves, the Birch and Swinnerton-Dyer conjecture asserts that the degree of vanishing, i.e. the so-called analytic rank, is equal to the rank of the group of rational points on the curve. 

The average analytic rank of elliptic curves was first studied by Brumer \cite{Brumer}, proving an upper bound of $2.3$ for this average. Brumer studied elliptic curves over $\mathbb{Q}$, as well as over the rational function field $\fq(T)$, with $q$ a prime power, however the results over $\mathbb{Q}$ are conditional on GRH. This bound was later improved by Heath-Brown \cite{Heath-Brown} and Young \cite{Young1}. The current best result for the average (algebraic) rank, which is in fact unconditional, is the bound $0.885$, proven by Bhargava and Shankar \cite{BS}.

The case when the degree of vanishing at the central point equals zero has been given much attention in the literature and lower bounds for the number of $L$-functions in a family which do not vanish at the central point are known as non-vanishing results. For the family of all elliptic curve $L$-functions, the first such result is due to Young \cite{Young}, conditional on the GRH. There are numerous results for other families of $L$-functions and we only mention two examples here. For a certain family of modular form $L$-functions, Iwaniec, Luo and Sarnak \cite[Corollaries 1.6, 1.7, 1.8]{ILS} proved non-vanishing, under the GRH, for a positive proportion, when ordered by weight or level. Furthermore, Iwaniec and Sarnak \cite{Iwaniec-Sarnak} proved an unconditional rate of non-vanishing of $50\%$ in a family of modular form $L$-functions, and in fact, they proved that any improvement of this rate would imply the non-existence of Landau-Siegel zeros for Dirichlet $L$-functions.

In some families, all $L$-functions are expected to be non-vanishing at the central point. One such family is the family of non-trivial Artin $L$-functions associated with quadratic extensions of $\mathbb{Q}$. This is closely related to families of quadratic Dirichlet $L$-functions. Over $\mathbb{Q}$, Özluk and Snyder \cite[Corollary 3]{Ozluk-Snyder} proved that under GRH, at least $93.75\%$ of such $L$-functions are non-vanishing when ordered by discriminant. After some additional optimisation, this lower bound can be improved to $\approx 94.27\%$, see e.g. \cite[Appendix A]{ILS}. The current best unconditional result over $\mathbb{Q}$ is due to Soundararajan \cite{S}, establishing a rate of non-vanishing of at least $87.5\%$, for a certain subfamily with specified ramification at the prime $2$.

Over the rational function field $\fq(T)$, with $q$ a prime power, the situation for the analogous quadratic family is more complicated. Indeed, Li \cite[Theorem 1.3]{Li} proved that an infinite number of quadratic Dirichlet $L$-functions vanish at the central point. On the other hand, non-vanishing results can be deduced from work of Rudnick \cite[Corollary 3]{Rudnick}, see also \cite[Corollary 2.1]{BF}, establishing a rate of non-vanishing analogous to the result over $\mathbb{Q}$. More recent work by Ellenberg, Li and Shusterman \cite[Theorem 1.2]{ELS} shows that if one considers $q=p^e$, then by picking $e$ large enough, the rate of non-vanishing can be brought arbitrarily close to $100\%$.

We study a certain family of Artin $L$-functions associated with $D_4$-quartic fields, ordered by conductor. More precisely, if $L$ is a quartic extension of $\fq(T)$, whose Galois closure has Galois group $D_4$, then $L$ contains a unique quadratic subfield $K$. Now, the group $D_4$ has a unique irreducible two-dimensional representation, with a corresponding Artin $L$-function equal to the quotient of Dedekind zeta functions $\zeta_L /\zeta_K$. Up to a scaling factor $q^{-4}$, its conductor $\mathrm{Cond}(L)$ equals the quotient $\abs{\Disc(L)}/\abs{\Disc(K)}$ of absolute values of absolute discriminants. Setting $C(L)=q^{4}\mathrm{Cond}(L)$, we prove the following theorem, establishing a positive proportion of non-vanishing in this family.
\begin{theorem}\label{introthmnonvanish}
    Let $q$ be a prime power coprime to $2$ and larger than a sufficiently large absolute constant. Define 
    \begin{equation*}
        \mathcal{F}(X) = \{(L,K): \text{ $L/\fq(T)$ is geometric and $D_4$-quartic with $C(L)=X$, $[L:K]=2$} \},
    \end{equation*}
    with one representative from each isomorphism class. Then, at least $77 \%$ of the Artin $L$-functions $P_{L/K}(s):=\zeta_L(s)/\zeta_K(s)$ associated with $(L,K)\in \mathcal{F}(X)$ are non-vanishing at the central point $s=1/2$. More precisely, 
    \begin{equation*}
        \liminf_{X\to \infty} \frac{1}{\#\mathcal{F}(X)}\sum_{(L,K)\in \mathcal{F}(X)}\mathbf{1}_{P_{L/K}(1/2) \neq 0}\geq 0.77. 
    \end{equation*}
\end{theorem}
Non-vanishing results for the corresponding family over $\mathbb{Q}$ were found by Durlanik \cite{Durlanik}, proving non-vanishing for a proportion $\gg X^{-1/4}$ of $L$-functions with conductor bounded by $X$. Conditional on the Generalised Lindelöf Hypothesis, this was improved to a proportion $\gg_\epsilon X^{-\epsilon}$ for every $\epsilon> 0$. The methods we use to prove Theorem \ref{introthmnonvanish} are essentially analytic in nature, and have natural analogues over number fields, whence we expect that by adapting the argument used here mutatis mutandis, one should obtain the same rate of non-vanishing over $\mathbb{Q}$ under relevant hypotheses.

The family $\mathcal{F}(X)$ is closely related to the quadratic families described above. Indeed, when ordered by conductor (or discriminant), essentially all quadratic extensions of quadratic extensions of $\fq(T)$ are $D_4$. Hence, finding non-vanishing results for $L$-functions parametrised by $\mathcal{F}(X)$ reduces to finding non-vanishing results for Artin $L$-functions corresponding to quadratic extensions $L/K$, with $K$ quadratic over $\fq(T)$, while keeping track of the $K$-dependence. In fact, by the same arguments as those used to obtain Theorem \ref{introthmnonvanish}, we are able to obtain the following non-vanishing result for Artin $L$-functions associated with quadratic extensions of suitable base fields.
\begin{prop}\label{quadprop}
    Let $K$ be an arbitrary fixed function field, with constant field $\fq$ of size coprime to $2$. For $X=q^{2n}$, let $\mathcal{G}(X) = \{L: [L:K]=2, \abs{\mathrm{Disc}(L/K)} = X\}$. Then, we have that at least $94\%$ of Artin $L$-functions associated with extensions in $\mathcal{G}(X)$ are non-vanishing at the central point. Specifically,
    \begin{equation*}
        \liminf_{X\to \infty} \frac{1}{\#\mathcal{G}(X)}\sum_{L\in \mathcal{G}(X)}\mathbf{1}_{P_{L/K}(1/2) \neq 0}\geq \frac{19-\cot(1/4)}{16}. 
    \end{equation*}
\end{prop}
For $K=\fq(T)$ our arguments and results essentially reduce to those of Rudnick \cite{Rudnick}, see also \cite{BF}. Because the field $K$ is fixed, we do not need to take $q$ large in contrast to Theorem \ref{introthmnonvanish}. Analogues of Proposition \ref{quadprop} were previously known for the finitely many imaginary quadratic number fields $K$ with class number $1$ by work of Gao and Zhao \cite{GZ1, GZ2}. Similarly to Theorem \ref{introthmnonvanish}, we expect that Proposition \ref{quadprop} generalises to the number-field case with only minor changes to the argument.

Our proofs of the two results above are based on studying the one-level density, which is a tool for studying the low-lying zeros of a family of $L$-functions. Specifically, if $\mathcal{H}(X)$ is a family of $L$-functions ordered by $X$, then for a Schwartz test function $\psi$, the one-level density is the average
\begin{equation}\label{abstractonelev}
    \frac{1}{\#\mathcal{H}(X)}\sum_{L(f,s)\in \mathcal{H}(X)}D_{f}(\psi),\,\,\,\,\text{ where }\,\,\,\,D_f(\psi) := \sum_{\rho_f: L(f,\rho_f)= 0} \psi\left(\frac{\gamma_f\log c_X}{2\pi}\right).
\end{equation}
Here $\rho_f = 1/2+i\gamma_f$ ranges over the nontrivial zeros of $L(s,f)$ and $c_X\asymp c_f$, the conductor of $L(s,f)$. Katz and Sarnak \cite{Katz-Sarnak} conjectured that as $X\to \infty$, the one-level density converges to the integral of $\phi(x)$ against one of five density functions, depending on the so-called symmetry type of $\mathcal{H}$. Aside from requiring that $\psi$ is Schwartz, one also requires that its Fourier transform is supported in $[-\sigma,\sigma]$ for some finite $\sigma >0$.

The Katz--Sarnak conjecture has been verified in several different families, as long as $\sigma$ is small enough. See e.g. \cite{ILS, Hughes-Rudnick, Young1}. The quadratic family is of symplectic type, and was studied in \cite{Ozluk-Snyder} and \cite{Rudnick, BF} over $\mathbb{Q}$ and $\fq(T)$ respectively, for $\sigma < 2$. In general, one can extract better non-vanishing results by proving the Katz--Sarnak conjecture for larger values of $\sigma$. We remark that obtaining results for support $\sigma > 2$ is a very hard problem, but has been accomplished when one performs additional averaging, see \cite{DPR} and \cite{BCL} for two recent examples.

In order to study the one-level density corresponding to $\mathcal{F}(X)$, the first step is to estimate the counting function $\#\mathcal{F}(X)$, i.e. the number of $D_4$-quartic field extensions $L/\fq(T)$, with $C(L)=X$. The analogous question over $\mathbb{Q}$ was first studied by \cite{ASVW}, who found a main term of size $\asymp X\log X$, with an error term of size $X\log\log X$. Their method utilises the so-called flipped field of a $D_4$-quartic field. This argument was later refined by Friedrichsen \cite{Friedrichsen}, who found an asymptotic formula for the counting function, consisting of the main term, a secondary term of size $\asymp X$, and an error term of size $\mathcal{O}_\epsilon(X^{11/12+\epsilon})$. 

We mention some results in different directions for the $D_4$-family over $\mathbb{Q}$. This family, when ordered by discriminant was first studied by Cohen, Diaz y Diaz and Olivier \cite{CDyDO}, and the current best bound for the error term is due to McGown and Tucker \cite{MT}. Moreover, Hansen and Zanoli \cite{HZ} studied this family when one instead orders it by so-called multi-invariants. Finally, we mention that the closely related family of $D_4$-octics was recently studied by Shankar and Varma \cite{SV}. 

Over $\fq(T)$, the number of $D_4$-quartic fields has previously been studied when ordered by discriminant by Keliher \cite{Keliher}. We apply Friedrichsen's methods to study $\#\mathcal{F}(X)$ and obtain the following theorem.
\begin{theorem}\label{thmfieldctall}
     For $q$ larger than a constant depending only on $\epsilon$, there is a constant $C_2=C_2(q)$, given in \eqref{fieldctthmeq}, such that, for $X=q^{2n}$, 
    \begin{equation*}
    \begin{split}
        \#\mathcal{F}(X) = 2X&(n-1)(1-q^{-2})^2\prod_{P}\left(1+\frac{1}{(\abs{P}+1)^2}\right) +C_2X + \mathcal{O}\left(X^{3/4+\epsilon}\right).
    \end{split}
    \end{equation*}
\end{theorem}
\begin{remark}
See \eqref{altEulerprod} for a different way of writing the Euler product in the main term, which is more similar to e.g. \cite[Theorem 1]{ASVW}.    
\end{remark}
 We remark that the reason we are able to improve the error term, compared to Friedrichsen's \cite[Theorem 1.1]{Friedrichsen}, is that the Generalised Riemann Hypothesis is a theorem over function fields, whence we are able to obtain better bounds for certain sums over Hecke characters arising in the computations.

One of the key insights for our study of the one-level density is the observation that we may split $\mathcal{F}(X)$ into subfamilies of positive proportion. Specifically, we have the following result.
\begin{prop}\label{reffieldct}
    Let $0\leq \alpha\leq \beta \leq 1$ be fixed, $X=q^{2n}$ and define
    \begin{equation*}
    \mathcal{F}_{\alpha, \beta}(X) = \{(L,K)\in \mathcal{F}(X): X^\alpha <\abs{\Disc(K)}\leq X^{\beta} \}.
    \end{equation*}
     Then, for fixed $q$ larger than an absolute constant, we have
    \begin{equation*}
        \frac{\#\mathcal{F}_{\alpha,\beta}(X)}{\#\mathcal{F}(X)} \to (\beta-\alpha),
    \end{equation*}
    as $X\to \infty$. Moreover, if $0 < \alpha <\beta <1/2$, or $1/2 < \alpha < \beta < 1$, then for $q$ larger than a constant depending on $\epsilon$, we have that
    \begin{equation*}
        \#\mathcal{F}_{\alpha,\beta}(X) = 2X(1-q^{-2})^2h(n,\alpha,\beta)\prod_{P}\left(1+\frac{1}{(\abs{P}+1)^2}\right)+ \mathcal{O}\big(X^{1/2+\max(\beta,1-\alpha)/2+\epsilon}\big),
    \end{equation*}
    where
    \begin{equation*}
        h(n,\alpha,\beta) = \begin{cases}
            \lceil n(1-\alpha)\rceil-\lceil n(1-\beta)\rceil, \text{ if $1/2 < \alpha < \beta<1$,}\\ 
            \lfloor n\beta\rfloor-\lfloor n\alpha\rfloor, \text{ if $0 <\alpha < \beta < 1/2$.}
        \end{cases}
    \end{equation*}
\end{prop}
Our methods do in fact allow us to obtain power-saving error terms also when $\alpha,\beta$ equal $0,1/2$ or $1$, and this is indeed what leads to Theorem \ref{thmfieldctall}. However, for simplicity we state our asymptotic formula for the refined counting function $\#\mathcal{F}_{\alpha,\beta}(X)$ only in the above range.

As the subfamilies described above have positive density, we may study the one-level density of each such subfamily in order to obtain non-vanishing results for the entire family. It turns out that we obtain better results for certain values of $\alpha$ and $\beta$.
\begin{theorem}\label{onelevtot}
    Let $\eta_0 >0$ be fixed, and let both $\alpha,\beta$ be at least a distance $\eta_0$ away from $0,1/2$ and $1$. Then, if $\psi$ is an even Schwartz function whose Fourier transform is supported in $(-\sigma,\sigma)$, with $\sigma < 2-3\min(1-\alpha,\beta)-o_q(1)$, we have for the one-level density that
    \begin{equation*}
        \lim_{X\to \infty }\frac{1}{\#\mathcal{F}_{\alpha,\beta}(X)}\sum_{(L,K)\in\mathcal{F}_{\alpha,\beta}(X)}D_{L/K}(\psi)= \widehat{\psi}(0)-\frac{1}{2}\int_{-1}^1\widehat{\psi}(u)du,
    \end{equation*}
    consistent with the symplectic Katz--Sarnak prediction. The expression $D_{L/K}(\psi)$ is defined in \eqref{dlkdef}, and is the analogue of $D_f$ in \eqref{abstractonelev}.
\end{theorem}
This improves results of Durlanik \cite[Theorem 3]{Durlanik}, who proved an analogue of the above theorem over $\mathbb{Q}$. Assuming the Generalised Lindelöf Hypothesis, he was able to prove the Katz--Sarnak conjecture when $\sigma < 1-\beta$. We remark that the inclusion of the parameter $\eta_0$ ensures that the convergence is uniform in $\alpha$ and $\beta$. It is possible to extend the results to allow for the possibility that $\alpha,\beta\in \{0,1/2,1\}$, but this is slightly more technical, and Theorem \ref{onelevtot} suffices for proving Theorem \ref{introthmnonvanish}. We remark that Proposition \ref{quadprop} essentially follows from the above theorem.

\subsection{Outline}
We begin in Section \ref{prelch} by recalling preliminaries concerning the arithmetic of function fields in general, and $D_4$-quartic fields in particular. We begin our computations in Section \ref{countfields}, where we count $D_4$-quartic function fields using the methods developed by Friedrichsen \cite{Friedrichsen}. Specifically, we study the subfamily $\mathcal{F}_{\alpha,\beta}(X)$, obtaining results leading to Theorem \ref{thmfieldctall} and Proposition \ref{reffieldct}.

In Section \ref{smallsubfldch}, we begin our study of the one-level density, obtaining Proposition \ref{smallalphabetaonelevdensthm}, proving Theorem \ref{onelevtot} for subfamilies of $\mathcal{F}(X)$ with a small enough quadratic subfield, as well as proving Proposition \ref{quadprop}. The main idea used to extend the support past the maximum of $\sigma < 1$ obtained by \cite{Durlanik} is to expand the Kronecker character $\chi_{L/K}(\mathcal{P})$ of a quadratic extension $L/K$ into a sum of Hecke characters on the ray class group modulo $\mathcal{P}$ and then apply the functional equation, and sum the resulting characters over $\mathcal{P}$. This is an extension of the approach used by Rudnick \cite{Rudnick}, who used quadratic reciprocity of Kronecker symbols.

As we are working with fields $K$ of genus larger than $0$, the reciprocal $1/\zeta_K(u)$ is no longer a polynomial, which complicates the computations somewhat. To handle this, we apply ideas used over $\mathbb{Q}$ by Gao \cite{Gao} to take care of sums of Möbius values. Moreover, compared to the computations over $\fq(T)$, our methods require analysis in the form of contour integration in order to obtain the uniformity, over the quadratic subfield $K$, required for our results, see e.g. \eqref{toshift} and \eqref{unifK}.

Next, in Section \ref{largesubfldsect}, we finish the proof of Theorem \ref{onelevtot} by proving Proposition \ref{bigsubfieldonelevthm}. The computations in this section are more delicate than those in the previous section. First, we apply precise results for the splitting of a rational prime $P\in \fq(T)$ in $L$ compared to its splitting in the flipped field $L'$, obtained in \cite[Table 1]{ASVW}, in order to turn sums over $\chi_{L/K}$ into sums over $\chi_{L'/K'}$. 

The splitting of a prime $P$ in $L/K$ is not completely analogous to its splitting in $L'/K'$, which requires us to bound a certain remainder term consisting of sums over "rogue" splitting types in Section \ref{roguech}. The rogue splitting types are quite rare, whence we only need to establish a very small degree of cancellation in these sums. In particular, we need to bound the contribution from character sums involving two primes simultaneously, see \eqref{doubleprime}. This leads to computations similar to those which would be used to study the $2$-level density of our family. Moreover, we also need to bound the difference of the number of fields where $P$ is inert, and where $P$ splits, weighted with certain arithmetic factors. 

Having estimated the contribution from the rogue splitting types, we move on to bounding character sums in the flipped field. To accomplish this, we use methods from \cite{Friedrichsen}, which by virtue of essentially being completely combinatorial, allow us to "count" the fields with a weight $\chi_{L'/K'}$, see \eqref{sumfg}. Bounding such character sums can be done with methods very similar to those used in Section \ref{smallsubfldch}. There are some difficulties arising from the additional summation and inclusion-exclusion, but this can be handled without too many changes to the approach from Section ~\ref{smallsubfldch}.

Finally, in Section \ref{nonvanishsect}, we apply Theorem \ref{onelevtot} in order to prove Theorem \ref{introthmnonvanish}, establishing non-vanishing for a significant proportion of the elements in $\mathcal{F}(X)$. The main method of converting one-level density results to non-vanishing results by choosing an appropriate test function remains the same as in e.g. \cite{Ozluk-Snyder}. However, the main insight we use to maximise the resulting proportion of non-vanishing is to split $\mathcal{F}(X)$ into subfamilies $\mathcal{F}_{\alpha,\beta}(X)$ in order to obtain as large of an admissible support $\sigma$ as possible, on average. Specifically, by splitting $\mathcal{F}(X)$ into smaller and smaller families, we obtain a bound for the total rate of non-vanishing in terms of a certain integral, see \eqref{nonvanishingeq}. Using the optimal test functions due to J. Vanderkam, and explicitly written down in \cite{Freeman}, we obtain Theorem \ref{introthmnonvanish}.
\subsection{Acknowledgements}
I want to thank Anders Södergren for introducing me to the problem of counting $D_4$-fields by conductor, and the problem of proving non-vanishing for families of $L$-functions at the central point. I also want to thank Christian Johansson for helpful discussions regarding the Galois theory of $D_4$-fields.

\subsection{Conventions}
Given some set $S$, and functions $f: S\to \mathbb{C}$ and $g: S\to \mathbb{R}_{\geq 0}$, we write $f = \mathcal{O}(g)$, if there is some constant $C$ such that $\abs{f(x)}\leq Cg(x)$ for all $x\in S$. Alternatively, we write this as $f\ll g$. Similarly, we write $f = h + \mathcal{O}(g)$ if $f-h\ll g$. Furthermore, if $f\ll g$ and $g\ll f$, then we write $f\asymp g$. The constant $C$ above is called the implied constant, and its dependence on a collection of variables $\mathcal{D}$ is indicated by writing $\mathcal{O}_\mathcal{D}$ or $\ll_\mathcal{D}$. We will allow all of the implied constants to depend on the variables $q$ and $\epsilon$, without indicating this.

We write $o(1)$ for a term that tends to $0$ as $X$, or equivalently $n$, tends to infinity. Furthermore, we write $o_q(1)$ for a term that tends to $0$ as $q\to \infty$. Sometimes, we shall make use of the non-standard notation $f\sim g$, meaning that $f=g+o(1)$.

For a Schwartz function $\psi$, its Fourier transform is the function
\begin{equation*}
    \widehat{\psi}(u) = \int_{\mathbb{R}}\psi(x)e^{-2\pi i x u }dx.
\end{equation*}
In particular, if $\psi$ is real and even, then its Fourier transform is real and even as well.

We will often study sums over divisors in some divisor group $D_K$. All divisors appearing in such sums are implicitly assumed to be effective, unless explicitly stated otherwise. Furthermore, if the field $K$ is understood from the context, we will simply write $\mathrm{Cl}$ instead of $\mathrm{Cl}(K)$.

\section{Preliminaries}\label{prelch}
In this section, we provide some background on function field extensions and their $L$-functions, as well as the Galois theory of $D_4$-quartic extensions. For a more thorough treatment of the theory of function fields, see \cite{Rosen}.

\subsection{Function fields and $L$-functions}
Let $K$ be a function field, by which we mean a finite separable extension of $\fq(T)$. A prime of $K$ is a discrete valuation $\nu$ such that $K$ is the fraction field of the corresponding ring $\mathcal{O}_\nu= \{x\in K: \nu(x)\geq 0\}$. The fractional ideal $\mathcal{P} := \{x\in K:\nu(x) \geq 1\}$ is also often referred to as a prime of $K$, and we sometimes write $\mathcal{O}_\mathcal{P}$ for $\mathcal{O}_\nu$.

Given a function field $K$, we may consider the group of divisors $D_K$ of $K$, which is the abelian free group generated by the primes of $K$. We will use multiplicative notation for the operation of this group. We have a map $K^*\to D_K$, given by mapping
\begin{equation*}
    x\mapsto \prod_{\mathcal{P} \text{ in $K$}} \mathcal{P}^{\nu_\mathcal{P}(x)},
\end{equation*}
where $\nu_P$ is the valuation corresponding to $P$. The kernel of this map is the constant field $\fq$, and its image is the group of principal divisors. The class group $\mathrm{Cl}(K)$ of $K$ is the quotient of $D_K$ with the principal divisors.

An element
\begin{equation*}
    D =\prod_{\mathcal{P}} \mathcal{P}^{i_{\mathcal{P}}}
\end{equation*}
in $D_K$ is called effective if all $i_{\mathcal{P}}\geq 0$. The degree of a prime $\mathcal{P}$ is defined to be the dimension of the residue field $\mathcal{O}_\mathcal{P}/\mathcal{P}$ over $\fq$. Moreover, the degree $\deg_K(D)$ of $D$ is defined to be the sum the (finitely many) nonzero $i_{\mathcal{P}}\deg(\mathcal{P})$. The norm, or absolute value (with respect to $K$) of a divisor $D$ is $\abs{D}_K := q^{\deg_K D}$.

The image of $K^*$ in $D_K$ is contained in the subgroup of divisors of degree $0$, by the function field analogue of the usual product formula. In particular, this lets us define the degree of an element in $\mathrm{Cl}(K)$. Unlike the number field case, the class group $\mathrm{Cl}(K)$ is never finite, however the group $\mathrm{Cl}^0(K)$ consisting of divisor classes of degree zero is finite.

If $L/K$ is a separable extension of function fields, then each prime $\mathfrak{P}$ in $L$ lies over some prime $\mathcal{P}$ in $K$, in the sense that $\nu_\mathfrak{P}$ resticts to a multiple of $\nu_\mathcal{P}$ in $K^*$, and in this case we write $\mathfrak{P}\mid \mathcal{P}$. More precisely, in $K^*$, $\nu_\mathcal{P}$ decomposes into a sum $e_1\nu_{\mathfrak{P}_1}+...+e_r\nu_{\mathfrak{P}_r}$. The number $e_i =: e(\mathfrak{P}_i/\mathcal{P})$ is called the ramification degree of $\mathfrak{P}_i\mid \mathcal{P}$. Moreover, we define the inertial degree $f(\mathfrak{P}/\mathcal{P})$ of $\mathfrak{P}\mid \mathcal{P}$ as the dimension of $\mathcal{O}_\mathfrak{P}/\mathfrak{P}$ over $\mathcal{O}_\mathcal{P}/\mathcal{P}$. This allows us to define a norm map
\begin{equation*}
    N_{L/K}: D_L\to D_K,
\end{equation*}
induced by $\mathfrak{P}\mapsto \mathcal{P}^{f(\mathfrak{P}/\mathcal{P})}$, where $\mathfrak{P}\mid \mathcal{P}.$

The integers $e(\mathfrak{P}/\mathcal{P})$, $f(\mathfrak{P}/\mathcal{P})$ and $r$ are related through the relation
\begin{equation*}
    \sum_{i=1}^re(\mathfrak{P}_i/\mathcal{P})f(\mathfrak{P}_i/\mathcal{P}) = [L:K].
\end{equation*}
In particular, if $L/K$ is quadratic, then there are three cases. The totally split case, when $r=2$. The ramified case when $r=1$, and $e(\mathfrak{P}/\mathcal{P}) = 2$ and finally the inert case when $r=1$, but $f(\mathfrak{P}/\mathcal{P})=2$. Note also that if one interprets $\mathcal{P}$ as a product of prime ideals in $L$, then $N_{L/K}(\mathcal{P}) = \mathcal{P}^{[L:K]}$.

The Dedekind zeta function associated with $K$ is the function
\begin{equation*}
    \zeta_K(s) = \sum_{D\geq 0}\frac{1}{\abs{D}_K^s} = \prod_{\mathcal{P}}\left(1-\abs{\mathcal{P}}^{-s}\right)^{-1}.
\end{equation*}
Often, one makes the substitution $u=q^{-s}$, whence we may write, with some abuse of notation,
\begin{equation*}
    \zeta_K(u) = \sum_{\mathcal{A}\geq 0}u^{\deg \mathcal{A}}=\sum_{k = 1}^\infty \left(\sum_{\deg_K \mathcal{A} = k}1\right)u^k = \prod_{\mathcal{P}}\left(1-u^{\deg \mathcal{P}}\right)^{-1}.
\end{equation*}
As over $\mathbb{Q}$, the sum and product above converges absolutely when $\Re(s) > 1$, or equivalently, when $\abs{u} < q^{-1}$. 

The primes of $K=\fq(T)$ have a particularly simple description. Indeed, they either correspond to irreducible (prime) polynomials, or they correspond to the so-called prime at infinity, given by the discrete valuation $v_\infty(f/g) = \deg(g)-\deg(f)$. Using this description of the primes, together with the fact that there are $q^k$ monic polynomials of degree $k$, we find that
\begin{equation*}
    \zeta_{\fq(T)}(u) = \frac{1}{(1-qu)(1-u)}.
\end{equation*}
In the sequel, we simply write $\zeta:=\zeta_{\fq(T)}$.

For a general $K$, we instead have the equality
\begin{equation}\label{zetaexpl}
    \zeta_K(u) = \frac{P_K(u)}{(1-qu)(1-u)},
\end{equation}
where $P_K(u)$ is a polynomial of degree $2g_K$, for an integer $g_K$ known as the genus of $K$. Moreover, all coefficients of $P_K(u)$ are integers. The case when $K$ is a quadratic extension of $\fq(T)$ will be of special interest to us. In this case, one sees by using the Euler product for $\zeta_K(u)$ that
\begin{equation}\label{Pkdef}
    P_K(u) = \sum_{\ell=0}^{2g_K}u^\ell \left(\sum_{\deg_K\mathcal{A} = \ell}\chi_K(\mathcal{A})\right),
\end{equation}
where $\chi_K$ is completely multiplicative on $D_K$ and defined on primes by
\begin{equation*}
    \chi_K(\mathcal{P}) = \begin{cases}
        1, &\text{ if $\mathcal{P}$ splits in $K$,}\\
        -1, &\text{ if $\mathcal{P}$ is inert in $K$,}\\
        0, &\text{ if $\mathcal{P}$ ramifies in $K$.}
    \end{cases}
\end{equation*}

The Riemann Hypothesis holds for a large class of $L$-functions over function fields. In particular, it holds for the polynomial $P_K(u)$ above, so that all roots of this polynomial have absolute value $q^{-1/2}$, see \cite[Theorem 5.10]{Rosen}. Moreover, these $L$-functions satisfy a functional equation, given by
\begin{equation*}
    \xi_K(u)=\xi_K(1/qu),
\end{equation*}
with 
\begin{equation*}
    \xi_K(u) = \zeta_K(u)u^{1-g_K}.
\end{equation*}

When $K/\fq(T)$ is geometric, meaning that the field of constants of $K$ is $\fq$, the genus is related to the discriminant $\mathrm{Disc}(K/\fq(T))$ of the extension $K/\fq(T)$ through the Riemann--Hurwitz formula, see \cite[Theorem 7.16]{Rosen},
\begin{equation*}
    2g_K-2 = (2g_{\fq(T)}-2)[K:\fq(T)]+\deg_K(\mathrm{Disc}(K/\fq(T))),
\end{equation*}
where $g_{\fq(T)} = 0$. The above formula also holds true when $K/\fq(T)$ is replaced by an arbitrary finite, separable, geometric extension of function fields. In particular, the Riemann--Hurwitz formula implies that the discriminant is always an even integer power of $q$.

We will not state the general definition of the discriminant here, instead we note that for our purposes, it suffices to know that for a quadratic extension $L/K$, 
\begin{equation*}
    \mathrm{Disc}(L/K) = \prod_{\mathfrak{P}} N_{L/K}\big(\mathfrak{P}\big)^{e(\mathfrak{P}/\mathcal{P})-1}\in D_K,
\end{equation*}
assuming that $2\nmid q$.

\subsection{The Galois theory of $D_4$-fields and Artin $L$-functions}
We now provide a very brief introduction to the Galois theory of $D_4$-fields, and their $L$-functions, see also \cite[Section 2]{ASVW} and \cite[Section 3]{Durlanik}. 

The group $D_4$ is generated by a rotation $\sigma$ of order $4$, and a reflection $\tau$ of order $2$. These are related through the relation $\tau^{-1}\sigma \tau = \sigma^3$. If $L$ is a $D_4$-quartic field, then its Galois closure $M$ is Galois, with Galois group $D_4$, with a subfield lattice corresponding to the subgroup lattice of $D_4$, see the figure below, \cite[Section 3]{Durlanik}.\newline
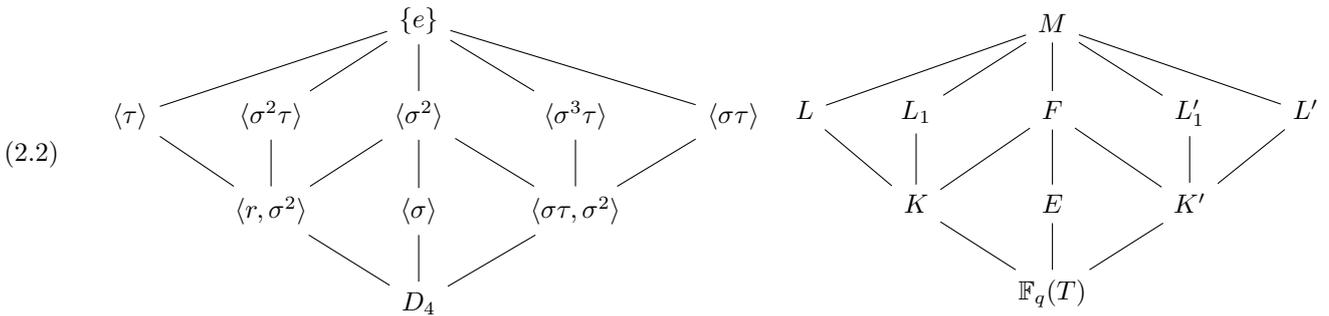
\begin{equation}
\begin{tikzcd}[every arrow/.append style={dash}]
&&\{e\} \ar[dll]\ar[dl]\ar[d]\ar[dr]\ar[drr]    \\
\langle \tau\rangle\ar[dr] & \langle \sigma^2\tau\rangle \ar[d] & \langle \sigma^2\rangle \ar[d]\ar[dl]\ar[dr] &\langle \sigma^3 \tau\rangle \ar[d] &\langle \sigma\tau \rangle\ar[dl]    \\
&\langle \tau,\sigma^2 \rangle\ar[dr]&\langle \sigma \rangle\ar[d]&\langle \sigma \tau,\sigma^2\rangle\ar[dl] \\
&& D_4 
\end{tikzcd}\,\,\,
\begin{tikzcd}[every arrow/.append style={dash}]
&& M \ar[dll]\ar[dl]\ar[d]\ar[dr]\ar[drr]    \\
L\ar[dr] & L_1 \ar[d] & F \ar[d]\ar[dl]\ar[dr] &L_1' \ar[d] & L'\ar[dl]    \\
&K\ar[dr]& E \ar[d]& K'\ar[dl] \\
&&\fq(T) 
\end{tikzcd}
\end{equation}

The fields $L$ and $L_1$ are isomorphic over $\fq(T)$. Similarly, the fields $L'_1$ and $L'$ are isomorphic, and also $D_4$-quartic. The fields $L$ and $L'$ are not isomorphic, however they are related in a different way. Indeed, there is a certain outer automorphism $\phi$ of $D_4$, mapping $\mathrm{Gal}(M/L)$ to $\mathrm{Gal}(M/L')$.

The group $D_4$ has a unique irreducible representation of dimension $2$, see \cite[Section 3]{Durlanik} for an explicit description of this representation. We are interested in the Artin $L$-function associated with this representation. By \cite[Proposition 3.1]{Durlanik}, this Artin $L$-function is in fact equal to $\zeta_L(u)/\zeta_K(u)$. For the general definition of an Artin $L$-function, see \cite[Chapter VII, §10]{Neukirch}. From this and the Riemann-Hurwitz formula, it follows, see \cite[Corollary 3.2]{Durlanik}, that up to a scaling factor $q^{-4}$, the so-called Artin conductor appearing in the functional equation of this $L$-function equals
\begin{equation*}
    \abs{\mathrm{Disc}(K)}\abs{N_{L/K}(\mathrm{Disc}(L/K))} =: C(L).
\end{equation*}

\section{Counting $D_4$-quartic fields}\label{countfields}
We begin with the task of estimating $\#\mathcal{F}(X)$ for $X=q^{2n}$. In fact we will need estimates for the refined counting function
\begin{equation*}
    \mathcal{F}_{\alpha,\beta}(X) := \{(L,K)\in \mathcal{F}(X): X^{\alpha} < \abs{\Disc(K_L)}\leq X^{\beta}\},
\end{equation*}
 for $0\leq \alpha\leq \beta\leq 1$. We remark that $\mathcal{F}(X) = \mathcal{F}_{0,1}(X) =\mathcal{F}_{0,1/2}(X)+\mathcal{F}_{1/2,1}(X) $.

The starting point for estimating $\#\mathcal{F}(X)$ is the observation that counting $D_4$-quartic fields of discriminant $X$ is essentially equivalent to counting fields $L$ which are quadratic extensions of a quadratic field $K$ over $\fq(T)$. To make this precise, we define the set $\Tilde{\mathcal{F}}_{\alpha,\beta}(X) = \{(L,K): [L:\fq(T)]=4, \,[L:K] = 2, \, X^{\alpha} <\abs{\Disc(K)}\leq X^\beta \}$. Then, we have the following lemma, analogous to \cite[Lemma 4.4]{ASVW}.
\begin{lemma}\label{quadonquadlemma}
    We have that
    \begin{equation*}
        \#\Tilde{\mathcal{F}}_{\alpha,\beta}(X)-\#\mathcal{F}_{\alpha,\beta}(X)\ll_\epsilon X^{(1+\beta)/2+\epsilon}.
    \end{equation*}
\end{lemma}
\begin{proof}
    The proof is identical to that of \cite[Lemma 4.4]{ASVW}, except that one uses \cite[Lemma 3.2]{Keliher} to bound the contribution from the Galois fields.
\end{proof}
When $\beta$ is sufficiently small, we may use the above lemma to estimate $\mathcal{F}_{\alpha,\beta}(X)$. However, when $\beta$ is close to $1$ we require some additional preparation before starting the counting process.

When counting all extensions in some fixed algebraic closure, we have the relation
\begin{equation}\label{tildeFormula}
    \#\Tilde{\mathcal{F}}_{\alpha,\beta}(X) = \sum_\sumstack{j\geq 1\\ X^{\alpha} < q^{2j}\leq X^\beta}\sum_\sumstack{[K:\fq(T)]=2\\ \abs{\Disc(K)}=q^{2j} }\sum_\sumstack{[L:K]=2\\\abs{\mathrm{Disc}(L/K)}=X/q^{2j}}1.
\end{equation}
If one counts each isomorphism class only once, multiplying the above by a $1/2$ in order to take the two conjugate $D_4$ fields into account yields an accurate formula up to the error from Lemma \ref{quadonquadlemma}. This motivates us to study the innermost sum above, counting the number of quadratic extensions of $K$.
\subsection{Counting the number of quadratic extensions}\label{quadsect}
In this section, we will let $M$ be a finite extension of $\fq(T)$. We will only be interested in the cases when $M$ is either a quadratic extension $K$ of $\fq(T)$, or $\fq(T)$ itself. We seek to estimate the sum
\begin{equation}\label{quadnoramcount}
    \sum_\sumstack{[L:M]=2\\\abs{\mathrm{Disc}(L/M)}=Y}1,
\end{equation}
for some $Y=q^{2r} > 1$. To later handle the case when $\beta$ is large, we will in fact need to study the slightly more refined counting function
\begin{equation}\label{quadramcount}
    \sum_\sumstack{[L:M]=2\\\abs{\mathrm{Disc}(L/M)}=Y\\ \mathcal{B}\mid \mathrm{Disc}(L/M)}1,
\end{equation}
where $\mathcal{B}$ is a product of distinct primes in $M$. We remark that $\eqref{quadnoramcount}$ is studied in \cite[Section 2]{Keliher}, and we will use similar methods for studying \eqref{quadramcount}, with some changes to accommodate the required ramification.

To begin, we let $V(M)$ denote the so-called virtual $2$-units in $M$, i.e. the elements in $M^{*}$ whose valuation at every prime is even. We then have the following lemma, cf. \cite[Lemma 3.3]{CDyDO} and the proof of \cite[Lemma 2.1]{Keliher}.
\begin{lemma}\label{quadbij}
    Let $M$ be a function field extension of $\fq(T)$ (with $2\nmid q$ as usual). Then the quadratic extensions of $M$ are in bijection with pairs $(\mathcal{A},\overline{u})$ with $\mathcal{A}$ being an effective, squarefree divisor and there existing $\alpha_0$ and $\mathcal{Q}$ such that $\mathcal{A}\mathcal{Q}^2 = (\alpha_0)$. Furthermore, $\overline{u}\in S(M) := V(M)/(M^*)^2$, and in the case that $\mathcal{A} = 0$, we require that $\overline{u} \neq \overline{\alpha_0^{-1}} \in S(M)$. In all cases, the discriminant of the extension corresponding to $(\mathcal{A},\overline{u})$ is $\mathcal{A}$.
\end{lemma}
\begin{proof}
    See the proof of \cite[Lemma 2.1]{Keliher}.
\end{proof}
Write $D_M$ for the divisor group of $M$ and $\Cl(M)$ for the associated divisor class group. Over function fields, $\Cl(M)$ is always infinite, however it contains the finite subgroup $\Cl^0(M)$, containing the divisor classes of degree zero. As every function field contains divisors of degree equal to $1$, every divisor $\mathcal{A}$ is a product of a divisor of degree zero, and one of degree $\deg_M(A)$. Moreover, if $\mathcal{A}$ is of even degree, which is the case that we are interested in, then the factor of degree zero and $\mathcal{A}$ are equal in $\Cl(M)/\Cl(M)^2$. 

Using that
\begin{equation*}
    \abs{G/G^2} = \abs{G[2]},
\end{equation*}
for an arbitrary finite abelian group $G$, we conclude that \eqref{quadramcount} equals
\begin{equation*}
    \abs{S(M)}\sum_\sumstack{ \mathcal{B}\mid \mathcal{A}\in D_M\\\abs{\mathcal{A}}=Y }\mu^2(\mathcal{A})\mathbf{1}_{\mathcal{A}\in \Cl^2_M} = \frac{\abs{S(M)}}{\abs{\Cl(M)/\Cl(M)^2}}\sum_{\chi \in \widehat{\Cl^0(M)}[2]}\sum_\sumstack{ \mathcal{B}\mid \mathcal{A}\in D_M\\\abs{\mathcal{A}}=Y }\mu^2(\mathcal{A})\chi(\mathcal{A}),
\end{equation*}
using orthogonality of characters. We recall that the sum is only over effective divisors $\mathcal{A}$, and this condition is implicit in all our sums. Factoring $\mathcal{B}$ from $\mathcal{A}$, we may rewrite the above as
\begin{equation}\label{quadcountsum}
    \frac{\abs{S(M)}}{\abs{\Cl^0(M)/\Cl^0(M)^2}}\sum_{\chi \in \widehat{\Cl^0(M)}[2]}\chi(\mathcal{B})\sum_\sumstack{ (\mathcal{B}, \mathcal{A})=1\\\abs{\mathcal{A}}=Y/\abs{\mathcal{B}} }\mu^2(\mathcal{A})\chi(\mathcal{A}).
\end{equation}

Recall that $\abs{Y}/\abs{\mathcal{B}} = q^{2r-\deg_M(\mathcal{B})}$. Hence, the innermost sum above is the $(2r-\deg_M(\mathcal{B})$)th coefficient of the power-series
\begin{equation*}
     \sum_\sumstack{ (\mathcal{B}, \mathcal{A})=1 }\mu^2(\mathcal{A})\chi(\mathcal{A})u^{\deg(\mathcal{A})} = \prod_{\mathcal{P}\nmid \mathcal{B}}\left(1+\chi(\mathcal{P})u^{\deg_M(\mathcal{P})}\right),
\end{equation*}
where we used multiplicativity to obtain the equality. As $\chi^2=1$, the Euler product in the right-hand side equals
\begin{equation*}
    \prod_{\mathcal{P}\mid \mathcal{B}}\left(1+\chi(\mathcal{P})u^{\deg_M(\mathcal{P})}\right)^{-1}\prod_{\mathcal{P}}\left(\frac{1-\chi(\mathcal{P})u^{\deg_M(\mathcal{P})}}{1-u^{2\deg_M(\mathcal{P})}}\right)^{-1} = \prod_{\mathcal{P}\mid \mathcal{B}}\left(1+\chi(\mathcal{P})u^{\deg_M(\mathcal{P})}\right)^{-1}\frac{L(u,\chi)}{\zeta_M(u^2)},
\end{equation*}
where
\begin{equation*}
    L(u,\chi) = \sum_{\mathcal{A}\in D_M}\chi(A)u^{\deg A}.
\end{equation*}
Both $L(u,\chi)$ and $\zeta_M(u)$ are absolutely convergent when $\abs{u} < q^{-1}$. As the Riemann hypothesis is known to hold, all zeros of $\zeta_M(u)$ lie on the circle $\abs{u}=q^{-1/2}$. The function $L(u,\chi)$ is a quadratic Hecke $L$-function, and the Riemann hypothesis is known for these functions as well. Assuming that $\chi$ is nontrivial, $L(u,\chi)$ is a polynomial in $u$, see \cite[Theorem 9.24]{Rosen}, and in the case when $\chi=\chi_0$ is trivial, we have that $L(u,\chi) = \zeta_M(u)$. 

We may extract coefficients from these $L$-functions using contour integration and Cauchy's formula. The reader more familiar with number fields should compare this to Perron's formula. Specifically, the innermost sum in \eqref{quadcountsum} equals
\begin{equation*}
    \frac{1}{2\pi i}\int_{\abs{u}=q^{-2}}\frac{1}{u^{2r-\deg_M(\mathcal{B})+1}}\prod_{\mathcal{P}\mid \mathcal{B}}\left(1+\chi(\mathcal{P})u^{\deg_M(\mathcal{P})}\right)^{-1}\frac{L(u,\chi)}{\zeta_M(u^2)}du. 
\end{equation*}
When $\chi$ is nontrivial, we may shift the contour to the circle $\abs{u}=q^{-1/2}$ without encountering any poles, using the Riemann hypothesis for function fields. To estimate the shifted integral, we use the fact that $L(u,\chi)$ is a polynomial, with constant coefficient equal to $1$, and \eqref{zetaexpl}. Writing 
\begin{equation*}
    P_M(u) = \prod_{j=1}^{2g_M}\left(1-uq^{1/2}e^{i\theta_{j,M}}\right)
\end{equation*}
and (for nontrivial $\chi$)
\begin{equation*}
    L(u,\chi) = \prod_{j=1}^{2g_M-2}\left(1-uq^{1/2}e^{i\theta_{j,\chi}}\right),
\end{equation*}
we may estimate the shifted integrand as being
\begin{equation*}
\begin{split}
    \ll q^{(2r-\deg_M(\mathcal{B})+1)/2}&\prod_{\mathcal{P}\mid \mathcal{B}}\left(1-q^{-\deg_M(\mathcal{P})/2}\right)^{-1}\frac{\prod_{j=1}^{2g_M-2}2}{\prod_{j=1}^{2g_M}\left(1-q^{-1/2}\right)} \\&\ll q^{(2r-\deg_M(\mathcal{B}))/2+\epsilon \deg_M(\mathcal{B})}C^{2g_M},
\end{split}
\end{equation*}
where $C $ is an absolute constant not depending on $q$.

Next, we estimate the integral when $\chi$ is principal. We once again shift the contour to the circle $\abs{u}= q^{-1/2}$. The shifted integral can be bounded as above. However, we see from \eqref{zetaexpl} that there is a pole at $u=q^{-1}$. Specifically, we pick up the residue
\begin{equation*}
    q^{2r-\deg_M(\mathcal{B})}\prod_{\mathcal{P}\mid \mathcal{B}}\left(1+\frac{1}{\abs{\mathcal{P}}}\right)^{-1}\frac{(1-q^{-2})P_M(q^{-1})}{P_M(q^{-2})}.
\end{equation*}

Now, a straightforward computation, cf. \cite[Proposition 5.2.5]{Cohen} shows that
\begin{equation*}
    \abs{S(M)} = 2\abs{\Cl^0(M)/\Cl^0(M)^2}.
\end{equation*}
When $M$ is a PID, the right-hand side above evidently equals $2$. When $M$ is a quadratic extension $\fq(T)(\sqrt{F})$ of $\fq(T)$, the right-hand side equals $2^{1+r_2}$, with $r_2$ bounded from above by the number of prime factors of $F$. For a geometric proof of an exact expression for $r_2$, see \cite[Theorem 1.4]{Cornelissen}.

We have proven the following.
\begin{lemma}\label{quadLemma}
    Let $M$ be an extension of degree $\leq 2$ of $\fq(T)$. Then, there is an absolute constant $C$ such that
    \begin{equation*}
        \sum_\sumstack{[L:M]=2\\\abs{\mathrm{Disc}(L/M)}=Y\\ \mathcal{B}\mid \mathrm{Disc}(L/M)}1 = \frac{2Y}{\abs{\mathcal{B}}}\prod_{\mathcal{P}\mid \mathcal{B}}\left(1+\frac{1}{\abs{\mathcal{P}}}\right)^{-1}\frac{(1-q^{-2})P_M(q^{-1})}{P_M(q^{-2})}+ \mathcal{O}\left(q^{\epsilon (g_M+\deg_M(\mathcal{B}))}\frac{Y^{1/2}}{\abs{\mathcal{B}}^{1/2}}C^{2g_M}\right).
    \end{equation*}
\end{lemma}
We remark that if one shifts the integral further away from the origin, then one obtains a smaller $Y$-dependence in the error term, at the cost of a factor involving a power of $\abs{\Disc(M)}$ and decreasing the exponent of $\abs{\mathcal{B}}$ in the denominator of the error term.

\subsection{Counting fields with \texorpdfstring{$\beta\leq 1/2$}{}}
When $\beta\leq 1/2$, we use Lemma \ref{quadonquadlemma} and study $\Tilde{F}_{\alpha,\beta}(X)$. Specifically, combining \eqref{tildeFormula} and Lemma \ref{quadLemma}, with $\mathcal{B} = 1$, shows that
\begin{equation*}
    \mathcal{F}_{\alpha,\beta}(X) = (1-q^{-2})\sum_\sumstack{j\geq 1\\ X^{\alpha}< q^{2j}\leq X^\beta}\sum_\sumstack{[K:\fq(T)]=2\\ \abs{\Disc(K)}=q^{2j} }\left(\frac{X}{\abs{\Disc(K)}}\frac{P_K(q^{-1})}{P_K(q^{-2})}+\mathcal{O}\left(\frac{X^{1/2+\epsilon}C^{2j}}{\abs{\Disc(K)}^{1/2}}\right)\right) + \mathcal{O}_\epsilon\left(X^{(1+\beta)/2+\epsilon}\right).
\end{equation*}
After summing the error term, we obtain
\begin{equation}\label{sumnoflip}
    (1-q^{-2})\sum_\sumstack{j\geq 1\\ X^{\alpha} < q^{2j}\leq X^\beta}\sum_\sumstack{[K:\fq(T)]=2\\ \abs{\Disc(K)}=q^{2j} }\left(\frac{X}{\abs{\Disc(K)}}\frac{P_K(q^{-1})}{P_K(q^{-2})}\right) + \mathcal{O}_\epsilon\left(X^{(1+\beta)/2+\epsilon}C^{n}\right).
\end{equation}
We postpone the evaluation of the main term, as we will obtain an asymptotic formula for the sum over $K$ as a special case of the computations in the forthcoming section.
\subsection{Counting fields with \texorpdfstring{$\beta > 1/2$}{}}
We now turn to the problem of estimating $\mathcal{F}_{\alpha,\beta}(X)$ when $\beta > 1/2$. Without loss of generality, we may assume that $\alpha \geq 1/2$. For this, we will adapt a method previously used over $\mathbb{Q}$ by \cite{Friedrichsen}.

Let $L$ be a $D_4$-field. Then, recall that there is a flipped field, say $L'$, associated with $L$. Moreover, this field contains a quadratic subfield $K'$. The conductor $C(L)$ equals $\abs{\Disc(K)\Disc(K')}$ up to a correction factor $J(L) = J(L')$, such that
\begin{equation*}
    C(L) = \abs{\Disc(K)\Disc(K')}J(L).
\end{equation*}
More precisely, $J(L)$ is the absolute value of the product $\Tilde{J}(L)$ of primes $\mathcal{P}$, which are unramified in $K$, but which, together with their Galois conjugate over $\fq(T)$, ramify in $L$. In particular, $J(L)$ is the squared absolute value of a product of primes in $\fq(T)$. See \cite[Section 3]{ASVW} for a more in-depth discussion.

We conclude that every pair $(L,K)$, with $L$ a $D_4$ field, is associated with a pair $(L',K')$, where
\begin{equation*}
    \abs{\Disc(K')} = \frac{X}{J(L')\abs{\Disc(K)}} < X^{1-\alpha} \leq X^{1/2}.
\end{equation*}
Hence, instead of counting pairs $(L,K)$ with $C(L)=X$ and $X^{\alpha} <\abs{\Disc(K)}\leq X^{\beta}$, we count pairs $(L',K')$, with $C(L')=X$ and $\abs{\Disc(K')} < X^{1-\alpha}$. However, we also have the additional condition that
\begin{equation*}
   X^{\beta } \geq \frac{X}{J(L')\abs{\Disc(K')}},\,\,\,\, \text{ and }\,\,\,\,  X^{\alpha } < \frac{X}{J(L')\abs{\Disc(K')}},
\end{equation*}
i.e.
\begin{equation*}
 \frac{X^{1-\beta}}{J(L')}\leq \abs{\Disc(K')} < \frac{X^{1-\alpha}}{J(L')} .
\end{equation*}
We conclude that
\begin{equation}\label{flipfeq}
\begin{split}
    \mathcal{F}_{\alpha,\beta}(X) &= \frac{1}{2}\sum_{f\in D_{\fq(T)}}\mu^2(f)\sum_\sumstack{j\geq 1\\  X^{1-\beta}/\abs{f}^2\leq q^{2j} < X^{1-\alpha}/\abs{f}^2}\sum_\sumstack{[K':\fq(T)]=2\\ \abs{\Disc(K')}=q^{2j}  \\ \text{$K'$ unramified at $f$}}\sum_\sumstack{[L':K']=2\\ \abs{\mathrm{Disc}(L'/K')}=X/q^{2j}\\ \Tilde{J}(L') = f^2}1 + \mathcal{O}\left(X^{1-\alpha/2+\epsilon}\right),
\end{split}
\end{equation}
where we also used Lemma \ref{quadonquadlemma}. For a general pair $(L',K')$, where $L'$ is not necessarily $D_4$, we define $\Tilde{J}(L') = \Tilde{J}(L',K')$ as the product over all primes in $K'$, which are unramified in $K'$, but which together with their Galois conjugates in $K'$ ramify in $L'$.

We continue studying the sum above, using inclusion-exclusion to handle the condition $\Tilde{J}(L') = f^2$. We can then rewrite the sum above as
\begin{equation}\label{flipfieldst}
     \frac{1}{2}\sum_{f\in D_{\fq(T)}}\mu^2(f)\sum_\sumstack{g\in D_{\fq(T)}\\ (g,f)=1}\mu(g)\sum_\sumstack{j\geq 1\\  X^{1-\beta}/\abs{f}^2\leq q^{2j} < X^{1-\alpha}/\abs{f}^2}\sum_\sumstack{[K':\fq(T)]=2\\ \abs{\Disc(K')}=q^{2j}  \\ \text{$K'$ unramified at $fg$}}\sum_\sumstack{[L:K']=2\\ \abs{\mathrm{Disc}(L'/K')}=X/q^{2j}\\ f^2g^2\mid N_{K'/\fq(T)}(\Disc(L'/K'))}1,
\end{equation}
where $N_{K'/\fq(T)}: D_{K'}\to D_{\fq(T)}$ is the norm map on divisors. We find an estimate for the innermost sum by applying Lemma \ref{quadramcount} with $\mathcal{B}$ being the product of all primes in $K'$ lying over $fg$. We then see that the above equals
\begin{equation}\label{Sumfieldflip}
\begin{split}
    X(1-q^{-2})\sum_{f\in D_{\fq(T)}}&\mu^2(f)\sum_\sumstack{g\in D_{\fq(T)}\\ (g,f)=1}\frac{\mu(g)}{\abs{fg}^2}\sum_\sumstack{j\geq 1\\  X^{1-\beta}/\abs{f}^2\leq q^{2j} < X^{1-\alpha}/\abs{f}^2}q^{-2j}\sum_\sumstack{[K':\fq(T)]=2\\ \abs{\Disc(K')}=q^{2j}  \\ \text{$K'$ unramified at $fg$}}\prod_{\mathcal{P}\mid fg}\left(1+\frac{1}{\abs{\mathcal{P}}}\right)^{-1}\frac{P_{K'}(q^{-1})}{P_{K'}(q^{-2})} \\&+ \mathcal{O}\left(X^{1-\alpha/2+\epsilon}C^{n}\right).
\end{split}
\end{equation}
\subsubsection{\texorpdfstring{Averaging over $K$}{}}\label{Kaversect}
We now evaluate the innermost sum in \eqref{Sumfieldflip}. Note that setting $f=g=1$ yields the innermost sum in \eqref{sumnoflip}. Let $\chi_{K'}$ denote the Kronecker character of $K'$. As $P_K(u)$ is a polynomial of degree $2g_{K'}$, we find using \eqref{Pkdef}, that 
\begin{equation*}
    P_{K'}(u) = \sum_{\ell= 0}^{2g_{K'}}u^{\ell}\sum_\sumstack{A\in D_{\fq(T)}\\ \deg(A)=\ell}\chi_{K'}(A).
\end{equation*}
Moreover, from the Euler product expansion, we have for $\abs{u} < q^{-1}$ that
\begin{equation*}
    \frac{1}{P_K(u)} = \prod_{P}\left(1-\chi_{K'}(P)u^{\deg(P)}\right),
\end{equation*}
so that
\begin{equation*}
    \prod_{\mathcal{P}\mid fg}\left(1+\frac{1}{\abs{\mathcal{P}}}\right)^{-1}P_{K'}(q^{-2})^{-1} = \prod_{P\nmid fg}\left(1-\frac{\chi_{K'}(P)}{\abs{P}^2}\right)\prod_{P\mid fg}\bigg(\left(1-\frac{\chi_{K'}(P)}{\abs{P}^2}\right)\prod_{\mathcal{P}\mid P}\left(1+\frac{1}{\abs{\mathcal{P}}}\right)^{-1}\bigg).
\end{equation*}
Splitting into cases depending on whether $P$ is inert or split in $K'$ shows that the above equals
\begin{equation*}
    \prod_{\mathcal{P}\mid fg}\left(1+\frac{1}{\abs{\mathcal{P}}}\right)^{-1}P_{K'}(q^{-2})^{-1} = \prod_{P\nmid fg}\left(1-\frac{\chi_{K'}(P)}{\abs{P}^2}\right)\prod_{P\mid fg}\left(1-\frac{\chi_{K'}(P)}{\abs{P}}\right)\prod_{P\mid fg}\left(1+\frac{1}{\abs{P}}\right)^{-1}.
\end{equation*}
Now, expanding into a Dirichlet series, we find that
\begin{equation}\label{sumtoavK}
    \prod_{\mathcal{P}\mid fg}\left(1+\frac{1}{\abs{\mathcal{P}}}\right)^{-1}\frac{P_{K'}(q^{-1})}{P_{K'}(q^{-2})} = \prod_{P\mid fg}\left(1+\frac{1}{\abs{P}}\right)^{-1}\sum_{\deg(A)\leq 2g_{K'}}\sum_{B\in D_{\fq(T)}}\frac{\chi_{K'}(AB)\mu(B)|(B,fg)|}{\abs{A}\abs{B}^2}.
\end{equation}
We wish to average this expression over $K'$ by interchanging the order of summation. For this, we need the following lemma.
\begin{lemma}\label{KavLemma}
    Let $C\in D_{\fq(T)}$. Then, if $C$ is a square, we have that
    \begin{equation*}
        \sum_\sumstack{[K':\fq(T)]=2\\ \abs{\Disc(K')}=q^{2j}  \\ \text{ $K'$ \textup{unramified at} $fg$ }}\chi_{K'}(C) = 2q^{2j}(1-q^{-2})\prod_{P\mid fgC}\left(1+\frac{1}{\abs{P}}\right)^{-1}+ \mathcal{O}_\epsilon\left(\abs{fgC}^{\epsilon}\right),
    \end{equation*}
    and if $C$ is a non-square, we have the bound
    \begin{equation*}
        \sum_\sumstack{[K':\fq(T)]=2\\ \abs{\Disc(K')}=q^{2j}  \\ \text{ $K'$ \textup{unramified at} $fg$ }}\chi_{K'}(C) \ll q^{j(1+\epsilon)}\abs{fgC}^{\epsilon}.
    \end{equation*}
\end{lemma}
\begin{proof}
    First, we use Lemma \ref{quadbij} to enumerate the quadratic $K'$ we are summing over. A computation shows that the Selmer group $S(\fq(T)) = \{1,\overline{\alpha_0}\}$, where $\alpha_0 \in \fq^{*}$ is a non-square. Moreover, the effective, squarefree divisors $D$ which are squares in the divisor class group are either divisors corresponding to squarefree polynomials of even degree, or they are $P_\infty$ multiplied with a squarefree polynomial of odd degree. This means that they are precisely the effective squarefree divisors of even degree. We conclude that
    \begin{equation}\label{charramsum}
        \sum_\sumstack{[K':\fq(T)]=2\\ \abs{\Disc(K')}=q^{2j}  \\ \text{ $K'$ \textup{unramified at} $fg$ }}\chi_{K'}(C)  = \sum_\sumstack{D\in D_{\fq(T)}\\(D,fgC)=1\\\deg(D) = 2j\\D \text{ squarefree}}\sum_{u\in \{1,\alpha_0\}}\chi_{D,u}(C).
    \end{equation}
    Here $(D,u)$ corresponds to the field $\fq(T)(\sqrt{uD'})$, where $D'=:D/(D,P_\infty)$ is interpreted as a monic element of $\fq[T]$.

    We claim that the innermost sum equals zero when $\deg(C)$ is odd. Indeed, it suffices to consider the case when $C$ is prime, and this case follows from the observation that $\alpha_0$ will be a non-square in $\mathbb{F}_{q^{\deg P}}$ as well. Therefore, we may assume that $\deg(C)$ is even, whence $\chi_{D,\alpha_0}(C)=\chi_{D,1}(C)=:\chi_{D}(C)$. We begin with the case when $C$ is not a square.

    Now, the splitting behaviour of primes over a quadratic field is governed by the quadratic residue symbol, see \cite[Chapter 3]{Rosen}. Specifically,
    \begin{equation*}
        \chi_{D}(C) = \left(\frac{D'}{C}\right).
    \end{equation*}
    If $P_\infty\mid C$, then $D' = D$, whence $\deg(D')$ is even. Hence, $P_\infty$ splits in the field $\fq(T)(\sqrt{D})$. Indeed, this follows by the Dedekind--Kummer theorem applied to the polynomial $X^2-1/D$ over $\fq[T^{-1}]$. If instead $P_\infty\nmid C$, then $C=C'$ so that in any case, the above equals
    \begin{equation*}
        \left(\frac{D'}{C'}\right) = \left(\frac{C'}{D'}\right) =\left(\frac{C}{D'}\right) = \left(\frac{C}{D}\right) = \chi_{C}(D),
    \end{equation*}
    where the second equality is quadratic reciprocity, using that at least one of $D'$ and $C'$ has even degree.

    As $D$ is coprime to $C$, $\chi_C(D) = \chi_{\mathrm{sqf}(C)}(D)$,
    where $\mathrm{sqf}(C)$ denotes the squarefree part of $C$, which is nontrivial as $C$ is not a square. Associated to this $\mathrm{sqf}(C)$, there is an associated quadratic field $K_C:=\fq(T)(\sqrt{\mathrm{sqf}(C)})$ and $\chi_{\mathrm{sqf}(C)}(D) = \chi_{K_C}(D)$. We are then interested in bounding the average
    \begin{equation*}
        \sum_\sumstack{D\in D_{\fq(T)}\\(D,fgC)=1\\\deg(D) = 2j}\chi_{K_C}(D).
    \end{equation*}
    This can be done by applying the methods of Section \ref{quadsect} to the $L$-function $P_{K_C}(u)$, and we find that the above sum is
    \begin{equation*}
        \ll_\epsilon q^{j(1+\epsilon)}\abs{fgC}^\epsilon.
    \end{equation*}

    We now turn to the case when $C$ is a square. Then, as long as $C$ is coprime to $D$, we have that $\chi_{D,u}(C) = 1$. Hence, in this case, \eqref{charramsum} equals
    \begin{equation*}
        2\sum_\sumstack{D\in D_{\fq(T)}\\(D,fgC)=1\\\deg(D) = 2j\\D \text{ squarefree}}1.
    \end{equation*}
    The sum is the $2j$th coefficient of
    \begin{equation*}
        \sum_\sumstack{D\in D_{\fq(T)}\\(D,fgC)=1\\D \text{ squarefree}}u^{\deg(D)} = \prod_{P\nmid fgC}\left(1+u^{\deg(P)}\right) = \frac{\zeta(u)}{\zeta(u^2)}\prod_{P\mid fgC}\left(1+u^{\deg(P)}\right)^{-1} = \frac{(1-qu^2)(1-u^2)}{(1-qu)(1-u)}\prod_{P\mid fgC}\left(1+u^{\deg(P)}\right)^{-1}.
    \end{equation*}
    Extracting this coefficient through contour integration by shifting the contour to the line $\abs{u}=1$, say, finishes the proof of the lemma.
\end{proof}

From the above lemma, we conclude that
\begin{equation*}
    \sum_\sumstack{[K':\fq(T)]=2\\ \abs{\Disc(K')}=q^{2j}  \\ \text{$K'$ unramified at $fg$}}\chi_{K'}(AB)-\mathbf{1}_{AB \text{ square}}\cdot2q^{2j}(1-q^{-2})\prod_{P\mid fgAB}\left(1+\frac{1}{\abs{P}}\right)^{-1} \ll_\epsilon q^{j(1+\epsilon)}\abs{fgAB}^{\epsilon}.
\end{equation*}
Summing the error term yields
\begin{equation*}
\begin{split}
    X(1-q^{-2})&\sum_{f\in D_{\fq(T)}}\mu^2(f)\sum_\sumstack{g\in D_{\fq(T)}\\ (g,f)=1}\frac{\mu(g)}{\abs{fg}^2}\prod_{P\mid fg}\left(1+\frac{1}{\abs{P}}\right)^{-1}\sum_\sumstack{j\geq 1\\  X^{1-\beta}/\abs{f}^2\leq q^{2j} < X^{1-\alpha}/\abs{f}^2}q^{-2j}
    \\& \times \sum_\sumstack{\deg(A)\leq 2j-2\\B\in D_{\fq(T)}}\frac{\mu(B)|(B,fg)|}{\abs{A}\abs{B}^2}\bigg(\sum_\sumstack{[K':\fq(T)]=2\\ \abs{\Disc(K')}=q^{2j}  \\ \text{$K'$ unramified at $fg$}}\chi_{K'}(AB)-\mathbf{1}_{AB \text{ square}}\cdot2q^{2j}(1-q^{-2})\prod_{P\mid fgAB}\left(1+\frac{1}{\abs{P}}\right)^{-1}\bigg).
\end{split}
\end{equation*}
If $\beta < 1$, the above is $\ll X^{(1+\beta)/2+\epsilon}$. Else, if $\beta=1$, then the lower bound in the summation over $j$ is $1$, and we may extend the upper bound to infinity, at the cost of an error term $\ll X^{(1+\alpha)/2+\epsilon}$.

We turn to the contribution of the main term from \eqref{sumtoavK} to \eqref{Sumfieldflip}. This contribution is
\begin{equation*}
\begin{split}
    2X(1-q^{-2})^2\sum_\sumstack{f\in D_{\fq(T)}\\ \abs{f} < X^{(1-\alpha)/2}}\mu^2(f)&\sum_\sumstack{g\in D_{\fq(T)}\\ (g,f)=1}\frac{\mu(g)}{\abs{fg}^2}\prod_{P\mid fg}\left(1+\frac{1}{\abs{P}}\right)^{-1}\sum_\sumstack{j\geq 1\\  X^{1-\beta}/\abs{f}^2\leq q^{2j} < X^{1-\alpha}/\abs{f}^2}1 \\&\times\sum_\sumstack{\deg(A)\leq 2j-2}\frac{\mu(\mathrm{sqf}(A))|(\mathrm{sqf}(A),fg)|}{\abs{A}\abs{\mathrm{sqf}(A)}^2}\prod_{P\mid fgA}\left(1+\frac{1}{\abs{P}}\right)^{-1}.
\end{split}
\end{equation*}
We now extend the summation over $A$ to infinity, at the cost of a tail sum
\begin{equation*}
\begin{split}
    -2X(1-q^{-2})^2\sum_\sumstack{f\in D_{\fq(T)}\\ \abs{f} < X^{(1-\alpha)/2}}\mu^2(f)&\sum_\sumstack{g\in D_{\fq(T)}\\ (g,f)=1}\frac{\mu(g)}{\abs{fg}^2}\prod_{P\mid fg}\left(1+\frac{1}{\abs{P}}\right)^{-1}\sum_\sumstack{j\geq 1\\  X^{1-\beta}/\abs{f}^2\leq q^{2j} < X^{1-\alpha}/\abs{f}^2}1 \\&\times\sum_\sumstack{\deg(A) > 2j-2}\frac{\mu(\mathrm{sqf}(A))|(\mathrm{sqf}(A),fg)|}{\abs{A}\abs{\mathrm{sqf}(A)}^2}\prod_{P\mid fgA}\left(1+\frac{1}{\abs{P}}\right)^{-1}.
\end{split}
\end{equation*}
The innermost sum is $\ll q^{-j}\abs{fg}^\epsilon$. Hence, we may extend the sum over $j$ to infinity at the cost of an error $\ll X^{(1+\alpha)/2 + \epsilon}$. We may extend the sum over $f$ to infinity with the same bound for the error term. If $\beta < 1$ the entire tail sum is in fact $\ll  X^{(1+\beta)/2 + \epsilon}$. When $\beta = 1$, the sum is $\ll X$.

We now consider the sum
\begin{equation*}
\begin{split}
    2X(1-q^{-2})^2\sum_\sumstack{f\in D_{\fq(T)}\\ \abs{f} < X^{(1-\alpha)/2}}\mu^2(f)&\sum_\sumstack{g\in D_{\fq(T)}\\ (g,f)=1}\frac{\mu(g)}{\abs{fg}^2}\prod_{P\mid fg}\left(1+\frac{1}{\abs{P}}\right)^{-1}\sum_\sumstack{j\geq 1\\  X^{1-\beta}/\abs{f}^2\leq q^{2j} < X^{1-\alpha}/\abs{f}^2}1 \\&\times\sum_\sumstack{A\in D_{\fq(T)}}\frac{\mu(\mathrm{sqf}(A))|(\mathrm{sqf}(A),fg)|}{\abs{A}\abs{\mathrm{sqf}(A)}^2}\prod_{P\mid fgA}\left(1+\frac{1}{\abs{P}}\right)^{-1}.
\end{split}
\end{equation*}
If $\beta < 1$, we restrict the summation over $f$ at $\abs{f} < X^{(1-\beta)/2}$ at the cost of an error $\ll X^{(1+\beta)/2+\epsilon}$. In this case, 
\begin{equation*}
    \sum_\sumstack{j\geq 1\\  X^{1-\beta}/\abs{f}^2\leq q^{2j} < X^{1-\alpha}/\abs{f}^2}1 = \lceil n(1-\alpha)\rceil-\lceil n(1-\beta)\rceil.
\end{equation*}
We replace the sum over $j$ with this expression, and then extend the sum over $f$ to infinity, at the cost of an error that can be absorbed in the existing error term. 

When instead $\beta = 1$, then
\begin{equation*}
    \sum_\sumstack{j\geq 1\\  q^0\leq q^{2j} < X^{1-\alpha}/\abs{f}^2}1 = \lceil n(1-\alpha)\rceil -\deg(f)-1.
\end{equation*}
In any case, the sum is $\sim n(\beta-\alpha)$.

Finally, we evaluate 
\begin{equation}\label{maintermsum}
\begin{split}
    \sum_\sumstack{f\in D_{\fq(T)}}\mu^2(f)&\sum_\sumstack{g\in D_{\fq(T)}\\ (g,f)=1}\frac{\mu(g)}{\abs{fg}^2}\prod_{P\mid fg}\left(1+\frac{1}{\abs{P}}\right)^{-1} \sum_\sumstack{A\in D_{\fq(T)}}\frac{\mu(\mathrm{sqf}(A))|(\mathrm{sqf}(A),fg)|}{\abs{A}\abs{\mathrm{sqf}(A)}^2}\prod_{P\mid fgA}\left(1+\frac{1}{\abs{P}}\right)^{-1}
    \\&=\sum_\sumstack{f\in D_{\fq(T)}}\mu^2(f)\sum_\sumstack{g\in D_{\fq(T)}\\ (g,f)=1}\frac{\mu(g)}{\abs{fg}^2}\prod_{P\mid fg}\left(1+\frac{1}{\abs{P}}\right)^{-2} \sum_\sumstack{A\in D_{\fq(T)}}\frac{\mu(\mathrm{sqf}(A))|(\mathrm{sqf}(A),fg)|}{\abs{A}\abs{\mathrm{sqf}(A)}^2}\prod_\sumstack{P\mid A\\ P\nmid fg}\left(1+\frac{1}{\abs{P}}\right)^{-1}.
\end{split}
\end{equation}
Now, by multiplicativity
\begin{equation*}
    \sum_\sumstack{A\in D_{\fq(T)}}\frac{\mu(\mathrm{sqf}(A))|(\mathrm{sqf}(A),fg)|}{\abs{A}\abs{\mathrm{sqf}(A)}^2}\prod_\sumstack{P\mid A\\ P\nmid fg}\left(1+\frac{1}{\abs{P}}\right)^{-1} = \prod_{P\nmid  fg}\left(1+\frac{1}{(\abs{P}+1)^2}\right).
\end{equation*}
Hence, \eqref{maintermsum} equals
\begin{equation*}
\begin{split}
    \prod_{P}\left(1+\frac{1}{(\abs{P}+1)^2}\right)&\sum_\sumstack{f\in D_{\fq(T)}}\mu^2(f)\prod_{P\mid f}\left(\frac{1}{1+(\abs{P}+1)^2}\right)\sum_\sumstack{g\in D_{\fq(T)}\\ (g,f)=1}\mu(g)\prod_{P\mid g}\left(\frac{1}{1+(\abs{P}+1)^2}\right)
    \\&=\prod_{P}\left(1+\frac{1}{(\abs{P}+1)^2}\right)\sum_\sumstack{f\in D_{\fq(T)}}\mu^2(f)\prod_{P\mid f}\left(\frac{1}{1+(\abs{P}+1)^2}\right)\prod_{P\nmid f}\left(\frac{(\abs{P}+1)^2}{1+(\abs{P}+1)^2}\right)
    \\&=\sum_\sumstack{f\in D_{\fq(T)}}\mu^2(f)\prod_{P\mid f}\frac{1}{(\abs{P}+1)^2} = \prod_{P}\left(1+\frac{1}{(\abs{P}+1)^2}\right).
\end{split}
\end{equation*}

By considering the case $f=g=1$ in the calculations above, one obtains a result for counting fields with $\beta \leq 1/2$. Here, one should also replace all instances of $1-\alpha$ with $\beta$ and all instances of $1-\beta$ with $\alpha$. Moreover, the upper bound in the summation over $j$ should be non-strict. This proves Proposition \ref{reffieldct}.

Finally, one obtains an asymptotic for all $D_4$-fields by adding together the results for $(\alpha,\beta)$ equal to $(0,1/2)$ and $(1/2,1)$. We obtain the following theorem.
\begin{theorem}\label{complicatedfieldctthm}
    Let $X=q^{2n}$. We have that
    \begin{equation}\label{fieldctthmeq}
    \begin{split}
        \#\mathcal{F}(X) = 2X&(n-1)(1-q^{-2})^2\prod_{P}\left(1+\frac{1}{(\abs{P}+1)^2}\right) \\&+X(1-q^{-2})\sum_\sumstack{0\leq f,g\in D_{\fq(T)}\\ (g,f)=1}\frac{\mu^2(f)\mu(g)}{\abs{fg}^2}\prod_{P\mid fg}\left(1+\frac{1}{\abs{P}}\right)^{-1}2^{\mathbf{1}_{fg=1}}\big(T_1(fg)+T_2(fg)+T_3(f,g)\big) + \mathcal{O}\left(X^{3/4+\epsilon}\right),
    \end{split}
    \end{equation}
    where
\begin{equation*}
\begin{split}
    T_1(fg)=& \sum_\sumstack{j\geq 1}q^{-2j}\sum_\sumstack{\deg(A)\leq 2j-2\\ B\in D_{\fq(T)}}\frac{\mu(B)|(B,fg)|}{\abs{A}\abs{B}^2}\bigg(\sum_\sumstack{[K':\fq(T)]=2\\ \abs{\Disc(K')}=q^{2j}   \\ \text{$K'$ unramified at $fg$}}\chi_{K'}(AB)-\mathbf{1}_{AB \text{ square}}\cdot2q^{2j}(1-q^{-2})\prod_{P\mid fgAB}\left(1+\frac{1}{\abs{P}}\right)^{-1}\bigg),
\end{split}
\end{equation*}
\begin{equation*}
\begin{split}
    T_2(fg)=-2(1-q^{-2}) \prod_{P\mid fg}\left(1+\frac{1}{\abs{P}}\right)^{-1}\cdot \sum_\sumstack{j\geq 1} \sum_\sumstack{\deg(A) > 2j-2}\frac{\mu(\mathrm{sqf}(A))|(\mathrm{sqf}(A),fg)|}{\abs{A}\abs{\mathrm{sqf}(A)}^2}\prod_{P\mid fgA}\left(1+\frac{1}{\abs{P}}\right)^{-1},
\end{split}
\end{equation*}
and 
\begin{equation*}
    T_3(f,g)=-2(1-q^{-2})\deg(f) \prod_{P\mid fg}\left(1+\frac{1}{\abs{P}}\right)^{-1} \sum_\sumstack{0\leq A\in D_{\fq(T)}}\frac{\mu(\mathrm{sqf}(A))|(\mathrm{sqf}(A),fg)|}{\abs{A}\abs{\mathrm{sqf}(A)}^2}\prod_{P\mid fgA}\left(1+\frac{1}{\abs{P}}\right)^{-1}.
\end{equation*}
\end{theorem}
We note that as $(1-q^{-2})(1-q^{-1}) =1/\zeta(2)$, the main term can be rewritten as
\begin{equation}\label{altEulerprod}
    2X(n-1)(1-q^{-1})^{-2}\prod_{P}\left(1-\frac{1}{\abs{P}^2}-\frac{2}{\abs{P}^3}+\frac{2}{\abs{P}^4}\right),
\end{equation}
cf. \cite[Theorem 1.1]{ASVW}.

\section{The one-level density: fields with small subfields}\label{smallsubfldch}
We now consider the one-level density of the subfamily
\begin{equation*}
    \mathcal{F}_{\alpha,\beta}(X)  = \{(L,K): [L:\fq(T)]=4, \,[L:K] = 2, \, X^{\alpha} <\abs{\Disc(K)}\leq X^\beta, L \text{ is a $D_4$-field} \},
\end{equation*}
with $0 <\alpha<\beta < 1/2$. In fact, we fix an $\eta_0 >0$ and require that $\eta_0 <\alpha<\beta <1/2-\eta_0$. We are interested in the $L$-functions 
\begin{equation}\label{PLKdef}
    \zeta_L(u)/\zeta_K(u) = P_{L/K}(u) = \prod_{\mathcal{P}}\left(1-\chi_{L/K}(\mathcal{P})u^{\deg_K(\mathcal{P})}\right)^{-1}.
\end{equation}
Here, $\mathcal{P}$ ranges over the primes in $K$, and $\chi_{L/K}(u)$ is the Kronecker character of $L$ over $K$. We remark that by the Riemann-Hurwitz theorem, $P_{L/K}$ is a polynomial of degree $2g_L-2g_K = 2n-4 =: N$. Moreover, using the Riemann Hypothesis, we see that the zeros have the form $q^{-1/2}e^{i\theta_j}$, where the $\theta_j$ lie in $N$ (not necessarily distinct) congruence classes in $\mathbb{R}/2\pi\mathbb{Z}$. 

We now fix a Schwartz function $\psi$ whose Fourier transform has support contained in $(-\sigma,\sigma)$. We then define
\begin{equation}\label{dlkdef}
    D_{L/K}(\psi) = \sum_{\theta_{L/K}}\psi\left(N\frac{\theta_{L/K}}{2\pi}\right),
\end{equation}
where the summation is over $\theta_{L/K}$ satisfying $P_{L/K}(q^{1/2}e^{i\theta_{L/K}})=0$, with multiplicity. The one-level density is then
\begin{equation*}
    \frac{1}{\#\mathcal{F}_{\alpha,\beta}(X)}\sum_{(L,K)\in\mathcal{F}_{\alpha,\beta}(X)}D_{L/K}(\psi).
\end{equation*}
We will be interested in the limit as $X\to \infty$ of the above expression, for fixed $\alpha,\beta$ and $q$. From the definition, and the fact that $\psi$ is Schwartz, we can immediately see that $D_{L/K}(\psi) \ll N$. Recall also that by Lemma \ref{quadonquadlemma}, the set $\tilde{\mathcal{F}}_{\alpha,\beta}(X)\setminus \mathcal{F}_{\alpha,\beta}(X)$ contains $\ll X^{1-\eta_0/2}$ elements, say. Hence,
\begin{equation*}
    \lim_{X\to \infty}\frac{1}{\#\mathcal{F}_{\alpha,\beta}(X)}\sum_{(L,K)\in\mathcal{F}_{\alpha,\beta}(X)}D_{L/K}(\psi) = \lim_{X\to\infty }\frac{1}{\#\tilde{\mathcal{F}}_{\alpha,\beta}(X)}\sum_{(L,K)\in\tilde{\mathcal{F}}_{\alpha,\beta}(X)}D_{L/K}(\psi).
\end{equation*}
We remark that we may count all fields, and not only count up to isomorphism. Analogous to $\tilde{\mathcal{F}}_{\alpha,\beta}(X)$, we use the notation $\tilde{\mathcal{F}}'_{\alpha,\beta}(X)$ for the collection of such fields where the isomorphism condition is dropped.

Next, we use a so-called explicit formula to rewrite the expression defining $D_{L/K}(\psi)$. First, if $\theta_1,...,\theta_{N}$ are the unique numbers with multiplicity, in $[0,2\pi)$, satisfying $P_{L/K}(q^{-1/2}e^{i\theta_j}) = 0$, then
\begin{equation*}
    D_{L/K}(\psi) = \sum_{j=1}^{N}\sum_{m\in \mathbb{Z}}\psi\left(N\frac{\theta_j+2\pi m}{2\pi}\right) = \frac{1}{N}\sum_{m\in \mathbb{Z}}\widehat{\psi}\left(\frac{m}{N}\right)\left(\sum_{j=1}^{N}e^{i\theta_j m}\right),
\end{equation*}
where the last equality is Poisson summation. This motivates us to define
\begin{equation}\label{cdef}
    c_{m}^{L/K} = \sum_{j=1}^{N}e^{i\theta_j m}.
\end{equation}
We remark that we have suppressed the dependence on $L$ and $K$ of the $\theta_j$. By the functional equation, we have $c_{m}^{L/K} = c_{-m}^{L/K}$, hence we may write
\begin{equation*}
    D_{L/K}(\psi)=\widehat{\psi}(0)+\frac{2}{N}\sum_{m=1}^\infty\widehat{\psi}\left(\frac{m}{N}\right)\left(\sum_{j=1}^{N}e^{i\theta_j m}\right).
\end{equation*}

We now find an alternative description of $c^{L/K}_m$ in terms of ramification data. First, as $P_{L/K}(u)$ is a polynomial of degree $N$, with constant coefficient $1$, we may factor it as
\begin{equation*}
    P_{L/K}(u) = \prod_{j=1}^N\left(1-q^{1/2}e^{i\theta_j}u\right).
\end{equation*}
Logarithmically differentiating this expression shows that
\begin{equation*}
    \frac{P'_{L/K}}{P_{L/K}}(u) = -\frac{1}{u}\sum_{m=1}^\infty q^{m/2}c^{L/K}_{m}  u^m.
\end{equation*}
On the other hand, by logarithmically differentiating the product in \eqref{PLKdef}, we observe that
\begin{equation*}
    \frac{P'_{L/K}}{P_{L/K}}(u) = \frac{1}{u}\sum_{\mathcal{P}}\sum_{r=1}^\infty \deg_K(\mathcal{P})\chi_{L/K}(\mathcal{P}^r)u^{r\deg_K(\mathcal{P})}.
\end{equation*}
Comparing coefficients, we conclude that
\begin{equation}\label{cequality}
    c^{L/K}_m = -q^{-m/2}\sum_{\deg_K\mathcal{P}\mid m}\deg_K(\mathcal{P})\chi_{L/K}(\mathcal{P}^{m/\deg_K(\mathcal{P})}).
\end{equation}

In conclusion, we have that
\begin{equation*}
\begin{split}
    \frac{1}{\#\tilde{\mathcal{F}}'_{\alpha,\beta}(X)}&\sum_{(L,K)\in\tilde{\mathcal{F}}'_{\alpha,\beta}(X)}D_{L/K}(\psi)  \\&=\widehat{\psi}(0)-\frac{2}{N\#\tilde{\mathcal{F}}'_{\alpha,\beta}(X)}\sum_{m=1}^\infty q^{-m/2}\widehat{\psi}\left(\frac{m}{N}\right)\sum_{(L,K)\in\tilde{\mathcal{F}}'_{\alpha,\beta}(X)}\sum_{\deg_K\mathcal{P}\mid m}\deg_K(\mathcal{P})\chi_{L/K}(\mathcal{P}^{m/\deg_K(\mathcal{P})}).
\end{split}
\end{equation*}
Now, from the Riemann Hypothesis over function fields, we obtain a prime 
number theorem \cite[Theorem 5.12]{Rosen}
\begin{equation}\label{primeidealct}
    \sum_{\deg_K(\mathcal{P})=m}1 = \frac{q^{m}}{m}+\mathcal{O}\left(g_Kq^{m/2}\right),
\end{equation}
where the error term can be improved to $\mathcal{O}\left(g_Kq^{m/3}\right)$ if $m$ is odd. Moreover, as a prime in $K$ of degree $m$ comes from a prime in $\fq(T)$ of degree equal to either $m$ or $m/2$, we have the uniform upper bound
\begin{equation}\label{primeunifbd}
    \sum_{\deg_K(\mathcal{P})=m}1 \ll \frac{q^m}{m},
\end{equation}
with no $K$-dependence. The bound \eqref{primeunifbd} already shows that
\begin{equation*}
    \frac{2}{N\#\tilde{\mathcal{F}}'_{\alpha,\beta}(X)}\sum_{m=1}^\infty q^{-m/2}\widehat{\psi}\left(\frac{m}{N}\right)\sum_{(L,K)\in\tilde{\mathcal{F}}'_{\alpha,\beta}(X)}\sum_\sumstack{\deg_K\mathcal{P}\mid m\\ \deg_K(\mathcal{P})\leq m/3}\deg_K(\mathcal{P})\chi_{L/K}(\mathcal{P}^{m/\deg_K(\mathcal{P})})\ll_\epsilon \frac{1}{N}\sum_{m=1}^\infty q^{m(-1/6+\epsilon)} \to 0,
\end{equation*}
as $X\to \infty$. Hence, we may restrict our attention to the two cases $\deg_K(\mathcal{P}) = m$ and $\deg_K(\mathcal{P}) = m/2$.

We first handle the case when $\deg_K(\mathcal{P}) = m/2$. Then, we note that $\chi_{L/K}(\mathcal{P}^{m/\deg_K(\mathcal{P})})=\chi_{L/K}(\mathcal{P}^{2}) = 1$, unless $\mathcal{P}$ ramified in $L$, in which case $\chi_{L/K}(\mathcal{P}^{2})$ equals zero. Now, we may bound the number of $\mathcal{P}$ which are ramified over $L$. Indeed, the primes which ramify are precisely those dividing the relative discriminant of $L/K$. The relative discriminant is of degree $\ll N$, whence the number of such primes is $o(N)$.

Using that $\chi_{L/K}(\mathcal{P}^{2}) = 1$ and adding back the ramified primes to the summation shows that, up to a term of size $o(1)$, the contribution from the primes with degree equal to half of $m$ is
\begin{equation*}
    -\frac{2}{N\#\tilde{\mathcal{F}}'_{\alpha,\beta}(X)}\sum_{m'=1}^\infty q^{-m'}\widehat{\psi}\left(\frac{2m'}{N}\right)\sum_{(L,K)\in\tilde{\mathcal{F}}'_{\alpha,\beta}(X)}\sum_{\deg_K\mathcal{P}= m'}\deg_K(\mathcal{P})\sim -\frac{2}{N}\sum_{m'=1}^\infty \widehat{\psi}\left(\frac{2m'}{N}\right),
\end{equation*}
where the asymptotic equality follows from \eqref{primeidealct} (and \eqref{primeunifbd} for small $m$). Applying Poisson summation once again shows that
\begin{equation*}
    -\frac{2}{N}\sum_{m'=1}^\infty \widehat{\psi}\left(\frac{2m'}{N}\right) \sim -\frac{\psi(0)}{2}.
\end{equation*}
We have arrived at
\begin{equation}\label{onelev}
\begin{split}
    \frac{1}{\#\tilde{\mathcal{F}}'_{\alpha,\beta}(X)}&\sum_{(L,K)\in\tilde{\mathcal{F}}'_{\alpha,\beta}(X)}D_{L/K}(\psi)  =\widehat{\psi}(0)-\frac{\psi(0)}{2}\\& - \frac{2}{N\#\tilde{\mathcal{F}}'_{\alpha,\beta}(X)}\sum_{m=1}^\infty q^{-m/2}m\widehat{\psi}\left(\frac{m}{N}\right)\sum_{(L,K)\in\tilde{\mathcal{F}}'_{\alpha,\beta}(X)}\sum_{\deg_K\mathcal{P}= m}\chi_{L/K}(\mathcal{P})+o(1).
\end{split}
\end{equation}

\subsection{Class groups and Hecke characters}\label{classgpsect}
Our goal is now to rewrite the sum in \eqref{onelev} in a way which makes the cancellation within the two innermost sums above clear. To begin, we expand
\begin{equation}\label{Heckestart}
    \sum_{(L,K)\in\tilde{\mathcal{F}}'_{\alpha,\beta}(X)}\sum_{\deg_K\mathcal{P}= m}\chi_{L/K}(\mathcal{P}) = \sum_\sumstack{j\geq 1\\ X^{\alpha}<q^{2j}\leq X^{\beta}} \sum_\sumstack{[K:\fq(T)]=2\\\abs{\Disc(K)}=q^{2j}}\sum_{\deg_K\mathcal{P}= m}\sum_\sumstack{[L:K]=2\\ \abs{\Disc(L/K)}=X/q^{2j}}\chi_{L/K}(\mathcal{P}).
\end{equation}
Now, from the parametrisation of quadratic extensions from Section \ref{countfields}, we know that the fields $L$ are in bijection with representatives $u$ of the elements in the Selmer group $S(K)$ and effective squarefree divisors $\mathcal{A}$ of even degree, with $\mathcal{A}$ being a square in the divisor class group. Moreover, $\abs{\mathcal{A}}= \abs{\Disc(L/K)}$. 

Writing $L(\mathcal{A},u)$ for the field corresponding to $\mathcal{A}$ and $u$ the innermost sum above equals
\begin{equation*}
    \sum_\sumstack{\mathcal{A} \text{ squarefree} \\ \mathcal{A}\in \mathrm{Cl}(K)^2 \\ \deg_K(\mathcal{A}) = 2n-2j }\sum_{u\in S(K)}\chi_{L(\mathcal{A},u)/K}(\mathcal{P}).
\end{equation*}
Now, we claim that if the degree of $\mathcal{P}$ is odd, then the innermost sum over $u$ is zero. Indeed, recall that $S(K)$ is a $2$-group and one of its components is $\fq^*/(\fq^*)^2$. Now, if $\mathcal{P}$ is odd, then a non-square in $\fq^*$ remains a non-square modulo $\mathcal{P}$. Hence, the innermost sum contains an equal amount of fields $L$ where $\mathcal{P}$ splits and where $\mathcal{P}$ is inert, whence it equals zero. Therefore, we may restrict the summation over $m$ to even $m$.

We now introduce some notation similar to \cite[Section 3]{CDyDO}. First, we let $S_\mathcal{P}(K)$ denote the subgroup of $S(K)$ consisting of classes $\overline{u}$ such that for some (all) lifts $u$ of $\overline{u}$, the polynomial $x^2-u$ splits over $\mathcal{O}_\mathcal{P}$. We let $\mathrm{Cl}_\mathcal{P}$ denote the ray class group modulo $\mathcal{P}$, i.e. the quotient of the group of divisors coprime to $\mathcal{P}$, with the set of principal divisors coprime to $\mathcal{P}$. 

We have the following lemma, connecting the splitting type of $\mathcal{P}$ in $L(\mathcal{A},u)$, with algebraic properties of $\mathcal{A}$ in the ray class group. The statement and proof are completely analogous to \cite[Lemma 3.5]{CDyDO}.
\begin{lemma}\label{splitlemma}
    Let $\mathcal{A}$ be an effective divisor which is a square in the divisor class group, such that there is some $\mathcal{Q}\in D_{K}$ with $\mathcal{A}\mathcal{Q}^2 = (\alpha)$. Let $\mathcal{M}$ be a squarefree divisor. Then the following are equivalent:

    \begin{enumerate}
        \item There exists an element $\overline{u}$ in the Selmer group, such that for any lift $u$ of $\overline{u}$, coprime to $\mathcal{P}$, the equation $x^2 \equiv \alpha u \nsmod{\mathcal{M}}$ has a solution.
        \item $\mathcal{A}$ is a square in $\mathrm{Cl}_\mathcal{M}(K)$.
    \end{enumerate}
\end{lemma}
Applying the lemma with $\mathcal{M}=\mathcal{P}$, we see that if any of the two conditions above holds, then there are exactly $\abs{\mathcal{S}_\mathcal{P}(K)}$ such $\overline{u}$. The size of this group can be discerned using the following lemma, see \cite[Lemma 3.7]{CDyDO}.
\begin{lemma}\label{seqlemma}
    There exists a natural exact sequence 
    \begin{equation*}
        1 \to S_\mathcal{P}(K)\to S(K) \to \frac{\mathcal{O}_\mathcal{P}^*}{(\mathcal{O}_\mathcal{P}^*)^2} \to \frac{\mathrm{Cl}^0_\mathcal{P}(K)}{(\mathrm{Cl}^0_\mathcal{P}(K))^2} \to \frac{\mathrm{Cl}^0(K)}{(\mathrm{Cl}^0(K))^2} \to 1.
    \end{equation*}
\end{lemma}
\begin{proof}
    The proof is essentially identical to the proof of \cite[Lemma 3.7]{CDyDO}, starting with the exact sequence
    \begin{equation*}
        \mathcal{O}_\mathcal{P}^*\to \mathrm{Cl}_\mathcal{P}^0 \to \mathrm{Cl}^0\to 1,
    \end{equation*}
    and then taking tensor products and applying the definitions of $S(K)$ and $S_\mathcal{P}(K)$.
\end{proof}
As $\mathcal{O}_\mathcal{P}^*/(\mathcal{O}_\mathcal{P}^*)^2$ has cardinality $2$, we have that $\abs{S(K)} = \abs{S_\mathcal{P}(K)}$, or $\abs{S(K)} = 2\abs{S_\mathcal{P}(K)}$. In the case that $\abs{S(K)} = 2\abs{S_\mathcal{P}(K)}$, $\chi_{L(\mathcal{A},u)/K}(\mathcal{P})$ has positive sign for half of the representatives $u$, and negative sign for the other half. Hence, we may restrict our attention to the case $\abs{S(K)} = \abs{S_\mathcal{P}(K)}$, in which case $\chi_{L(\mathcal{A},u)/K}(\mathcal{P}) = \chi_{L(\mathcal{A})/K}(\mathcal{P})$ only depends on $\mathcal{A}$. Specifically, if $\mathcal{P}\mid \mathcal{A}$, then the value is zero, the value is $1$ if $\mathcal{A}$ is a square in $\mathrm{Cl}_\mathcal{P}$, and $-1$ otherwise.

In conclusion, we have shown that
\begin{equation}\label{tobeprim}
    \sum_\sumstack{\mathcal{A} \text{ squarefree} \\ \mathcal{A}\in \mathrm{Cl}(K)^2 \\ \deg_K(\mathcal{A}) = 2n-2j }\sum_{u\in S(K)}\chi_{L(\mathcal{A},u)/K}(\mathcal{P}) = \abs{S(K)}\mathbf{1}_{\abs{S(K)}=\abs{S_\mathcal{P}(K)}}\sum_\sumstack{\mathcal{A} \text{ squarefree}  \\ \deg_K(\mathcal{A}) = 2n-2j\\ \mathcal{P}\nmid \mathcal{A} }\left(2\cdot \mathbf{1}_{\mathcal{A}\in \mathrm{Cl}_\mathcal{P}(K)^2}-\mathbf{1}_{\mathcal{A}\in \mathrm{Cl}(K)^2}\right).
\end{equation}
We remark that Lemma \ref{seqlemma} shows that the condition $\abs{S(K)}=\abs{S_\mathcal{P}(K)}$ is equivalent to $\abs{\mathrm{Cl}^0_\mathcal{P}(K)/(\mathrm{Cl}^0_\mathcal{P}(K))^2} = 2\abs{\mathrm{Cl}^0(K)/(\mathrm{Cl}^0(K))^2}$. We expand the indicator functions into a sum of characters
\begin{equation*}
    \abs{\mathrm{Cl}^0(K)/(\mathrm{Cl}^0(K))^2}\big(2\cdot \mathbf{1}_{\mathcal{A}\in \mathrm{Cl}_\mathcal{P}(K)^2}-\mathbf{1}_{\mathcal{A}\in \mathrm{Cl}(K)^2}\big) = \sum_{\chi \in \widehat{\mathrm{Cl}_\mathcal{P}^0(K)}[2]}\chi(\mathcal{A})-\sum_{\chi \in \widehat{\mathrm{Cl}^0(K)}[2]}\chi(\mathcal{A}) = \sum_\sumstack{\chi \in \widehat{\mathrm{Cl}_\mathcal{P}^0(K)}[2]\\ \chi \text{ primitive}}\chi(\mathcal{A}).
\end{equation*}
If $\mathcal{P}\mid \mathcal{A}$, we use the convention that $\chi(A)= 0$ so that we may include these $\mathcal{A}$ in the sum. Recall also that $\abs{\mathrm{Cl}^0(K)/(\mathrm{Cl}^0(K))^2} = \abs{S(K)}/2$. Hence, \eqref{tobeprim} equals
\begin{equation}\label{chiprimorigin}
    2\cdot \mathbf{1}_{\abs{S(K)}=\abs{S_\mathcal{P}(K)}}\sum_\sumstack{\chi \in \widehat{\mathrm{Cl}_\mathcal{P}^0(K)}[2]\\ \chi \text{ primitive}}\sum_\sumstack{\mathcal{A} \text{ squarefree}  \\ \deg_K(\mathcal{A}) = 2n-2j }\chi(\mathcal{A}).
\end{equation}

Now, similar to \cite[Section 3.1]{Rudnick}, we use inclusion-exclusion to extend the sum over $\mathcal{A}$ to non-squarefree divisors. Specifically, by inclusion-exclusion,
\begin{equation*}
    \mathbf{1}_{\mathcal{A} \text{ squarefree}} = \sum_{\mathcal{C}^2 \mid \mathcal{A}}\mu(C).
\end{equation*}
Hence, 
\begin{equation*}
    \sum_\sumstack{\mathcal{A} \text{ squarefree}  \\ \deg_K(\mathcal{A}) = 2n-2j }\chi(\mathcal{A}) = \sum_\sumstack{\deg_K(\mathcal{A}) = 2n-2j }\sum_{\mathcal{C}^2\mid \mathcal{A}}\mu(C)\chi(\mathcal{A}) = \sum_{2a+b=2n-2j}\sum_{\deg_K(\mathcal{B})=b}\chi(\mathcal{B})\sum_{\deg_K(\mathcal{C})=a}\mu(\mathcal{C})\chi(\mathcal{C}^2).
\end{equation*}
As $\chi(\mathcal{C}^2) = 1$, unless $\mathcal{P}\mid \mathcal{C}$, the above equals
\begin{equation*}
    \sum_{2a+b=2n-2j}\sum_{\deg_K(\mathcal{B})=b}\chi(\mathcal{B})\sum_\sumstack{\deg_K(\mathcal{C})=a \\ \mathcal{P} \nmid C}\mu(\mathcal{C}).
\end{equation*}

We can now conclude that 
\begin{equation}\label{expandcharsum}
\begin{split}
     &\sum_{(L,K)\in\tilde{\mathcal{F}}'_{\alpha,\beta}(X)}\sum_{\deg_K\mathcal{P}= m}\chi_{L/K}(\mathcal{P}) \\&= \sum_\sumstack{j\geq 1\\ X^{\alpha}<q^{2j}\leq X^{\beta}} \sum_\sumstack{[K:\fq(T)]=2\\\abs{\Disc(K)}=q^{2j}}\sum_{\deg_K(\mathcal{P})=m}2\cdot \mathbf{1}_{\abs{S(K)}=\abs{S_\mathcal{P}(K)}}\sum_\sumstack{\chi \in \widehat{\mathrm{Cl}_\mathcal{P}^0(K)}[2]\\ \chi \text{ primitive}}\sum_{2a+b=2n-2j}\sum_{\deg_K(\mathcal{B})=b}\chi(\mathcal{B})\sum_\sumstack{\deg_K(\mathcal{C})=a \\ \mathcal{P} \nmid C}\mu(\mathcal{C}).
\end{split}
\end{equation}
As in the proof of \cite[Lemma 4]{Rudnick}, we may write
\begin{equation}\label{pcoprimcomb}
    \sum_\sumstack{\deg_K(\mathcal{C})=a \\ \mathcal{P} \nmid C}\mu(\mathcal{C}) = \sum_\sumstack{\deg_K(\mathcal{C})=a}\mu(\mathcal{C})+ \sum_\sumstack{\deg_K(\mathcal{C})=a-\deg_K(\mathcal{P}) \\ \mathcal{P} \nmid C}\mu(\mathcal{C}) =\sum_\sumstack{\deg_K(\mathcal{C})=a}\mu(\mathcal{C})+ \sum_{\ell=1}^{\lfloor a/m\rfloor}\sum_\sumstack{\deg_K(\mathcal{C})=a-\ell m }\mu(\mathcal{C}).
\end{equation}

\subsection{Upper bounds on sums of divisors}
Before continuing our study of the one-level density, we formulate a few useful lemmas. All of the lemmas below are proven using contour integration, very similarly to the methods used in Section \ref{quadsect}. For the sake of brevity, we only provide brief proof sketches.

\begin{lemma}\label{polylemma}
    Let $H$ be a polynomial of degree $d \ll n$, with constant coefficient $1$, and with roots of absolute value $q^{-1/2}$. Then, the absolute value of the $\ell$th coefficient of $H$ is
    \begin{equation*}
        \ll q^{\ell/2}q^{o_q(n)}.
    \end{equation*}
\end{lemma}
\begin{proof}
    One shifts the integral 
    \begin{equation*}
    \frac{1}{2\pi i}\int_{\abs{u}= q^{-2}}\frac{H(u)}{u^{\ell+1}}du    
    \end{equation*}
    to the circle $\abs{u}=q^{-1/2}$ and uses a product expansion of $H(u)$.
\end{proof}

Next, we have a lemma bounding sums over $\mu(\mathcal{C})$. Recall that we write $o_q(1)$ for a term that tends to $0$ as $q\to \infty$. 
\begin{lemma}\label{mulemma}
    Let $K$ be a function field of genus $g_K \ll n$. Then, for $\ell\ll n$, we have that
    \begin{equation*}
        \sum_\sumstack{\deg_K(\mathcal{C})=\ell}\mu(\mathcal{C}) \ll q^{\ell/2}q^{o_q(n)}.
    \end{equation*}
\end{lemma}
\begin{proof}
    One studies the generating function $1/\zeta_K(u) = (1-u)(1-qu)/P_K(u)$ using contour integration, shifting the contour to $\abs{u} = q^{-1/2-1/\log q}$.
\end{proof}

Finally, we have a lemma allowing us to count the number of divisors of a certain degree.
\begin{lemma}\label{divisorctlemma}
    Let $K$ be a function field of genus $g_K\ll n$, and let $\ell \ll n$. Then,
    \begin{equation}\label{exactdivct}
        \sum_{\deg_K(\mathcal{A})=\ell}1 =\sum_{k=0}^{\min\{2g_K,\ell\}}a_{k,K}\left(\frac{q^{\ell-k+1}-1}{q-1}\right), 
    \end{equation}
    where $a_{k,K}$ is the $k$th coefficient of $P_K(u)$. Moreover, we have the upper bound
    \begin{equation*}
        \sum_{\deg_K(\mathcal{A})=\ell}1 \ll q^{\ell}q^{o_q(n)}.
    \end{equation*}
\end{lemma}
\begin{proof}
    In order to obtain the exact formula, one compares coefficients with the generating function $P_K(u)/\big((1-qu)(1-u)\big)$, using that
    \begin{equation*}
        \frac{1}{(1-qu)(1-u)} = \sum_{\ell' = 0}^\infty \frac{q^{\ell'+1}-1}{q-1}\cdot u^{\ell'}.
    \end{equation*}
    For the upper bound, one uses contour integration. 
\end{proof}

\subsection{Splitting the summation and summing over $\mathcal{P}$}
The sum
\begin{equation*}
    \sum_{\deg_K(\mathcal{B})=b}\chi(\mathcal{B})
\end{equation*}
is the $b$th coefficient of the Hecke $L$-function $L(u,\chi)$. Now, by \cite[Theorem 9.24A]{Rosen}, $L(u,\chi)$ is in fact a polynomial in $u$ of degree $2g_K-2+\deg_K(\mathcal{P})$. Moreover, it obeys a functional equation, implying that
\begin{equation}\label{fcnaleqHecke}
    \sum_{\deg_K(\mathcal{B})=b}\chi(\mathcal{B}) = q^{b-g_K+1-m/2}\sum_{\deg_K(\mathcal{B})=2g_K-2+m-b}\chi(\mathcal{B}).
\end{equation}
We will apply this functional equation at a later stage. For now, we use the observation that $L(u,\chi)$ is a polynomial, together with \eqref{pcoprimcomb}, Lemma \ref{polylemma} and Lemma \ref{mulemma} to conclude that the condition $\mathcal{P}\nmid \mathcal{C}$ can be removed from \eqref{expandcharsum}, at a total cost of an error $o_{X}(1)$ to \eqref{onelev}, as long as $q$ is larger than some absolute constant.

We now introduce the notation 
\begin{equation*}
    \sigma_K(a) = \sum_{\deg_K({\mathcal{C})= a}}\mu(C).
\end{equation*}
Furthermore, we write $2r:= 2n-2j$, and we fix a small $\delta > 0$. Next, we split the summation over $a$ and $b$ in \eqref{expandcharsum} (with the condition $\mathcal{P}\nmid \mathcal{C}$ removed) into two parts. One part $S$ where $a \leq \delta r$, and one part $T$ where instead $a >\delta r$.

For the sum $T$ we interchange the order of summation, obtaining 
\begin{equation*}
    T = 
2\sum_\sumstack{j\geq 1\\ X^{\alpha}<q^{2j}\leq X^{\beta}} \sum_\sumstack{[K:\fq(T)]=2\\\abs{\Disc(K)}=q^{2j}}\sum_\sumstack{2a+b=2r\\ a > \delta r}\sigma_K(a)\sum_{\deg_K(\mathcal{B})=b}\sum_{\deg_K(\mathcal{P})=m} \mathbf{1}_{\abs{S(K)}=\abs{S_\mathcal{P}(K)}}\sum_\sumstack{\chi \in \widehat{\mathrm{Cl}_\mathcal{P}^0(K)}[2]\\ \chi \text{ primitive}}\chi(\mathcal{B}).
\end{equation*}
If $\mathcal{P}\mid \mathcal{B}$, the innermost sum is zero. Else, it equals $\mathbf{1}_{\mathcal{B} \in \mathrm{Cl}(K)^2}\chi_{L(\mathrm{sqf}(\mathcal{B}))/K}(\mathcal{P})\abs{S(K)}/2$ by the same reasoning as that used for obtaining \eqref{chiprimorigin}. Hence,
\begin{equation*}
    T = \abs{S(K)}\sum_\sumstack{j\geq 1\\ X^{\alpha}<q^{2j}\leq X^{\beta}} \sum_\sumstack{[K:\fq(T)]=2\\\abs{\Disc(K)}=q^{2j}}\sum_\sumstack{2a+b=2r\\ a > \delta r}\sigma_K(a)\sum_\sumstack{\deg_K(\mathcal{B})=b\\ \mathcal{B}\in \mathrm{Cl}(K)^2}\sum_{\deg_K(\mathcal{P})=m} \mathbf{1}_{\abs{S(K)}=\abs{S_\mathcal{P}(K)}}\chi_{L(\mathrm{sqf}(\mathcal{B}))/K}(\mathcal{P}).
\end{equation*}
Recall that for a divisor $\mathcal{B}\in \mathrm{Cl}(K)^2$, $\chi_{L(\mathcal{B})/K}(\mathcal{P})$ is well-defined as $\abs{S(K)} = \abs{S_\mathcal{P}(K)}$. However, if this condition does not hold, we may still define this character by simply associating to $\mathcal{B}$, a choice of a field of relative discriminant $\mathrm{sqf}(\mathcal{B})$. Clearly, the value of the character will then depend on this choice, but this definition will still be useful to us.

From the definition of $S_{\mathcal{P}}(K)$, we have that $S_\mathcal{P}(K) = S(K)$ if and only if all elements $u\in S(K)$ are squares modulo $\mathcal{P}$. As the degree of $\mathcal{P}$ is even, this always holds for $u\in \fq^*$. The group $S(K)/\fq^*$ is a $2$-group of rank $r_K$, with generators $u_1,...,u_{r_K}$, say. Let $\psi_1$ be the nontrivial character of the group $\mathcal{O}_\mathcal{P}^*/(\mathcal{O}_\mathcal{P}^*)^2\simeq \mathbb{Z}/2\mathbb{Z}$. Then,
\begin{equation*}
    \mathbf{1}_{\abs{S(K)}=\abs{S_\mathcal{P}(K)}} = \frac{2}{\abs{S(K)}}\sum_{u\in S(K)/\fq^*}\psi_{1}(u).
\end{equation*}
Next, we note that 
\begin{equation*}
    \sum_{u\in S(K)/\fq^*}\psi_1(u)\chi_{L(\mathrm{sqf}(\mathcal{B}))/K}(\mathcal{P}) = \sum_{u\in S(K)/\fq^*}\chi_{L(\mathrm{sqf}(\mathcal{B}),u)/K}(\mathcal{P}),
\end{equation*}
so that
\begin{equation}\label{TsumtoPave}
    T = 2\sum_\sumstack{j\geq 1\\ X^{\alpha}<q^{2j}\leq X^{\beta}} \sum_\sumstack{[K:\fq(T)]=2\\\abs{\Disc(K)}=q^{2j}}\sum_{u\in S(K)/\fq^*}\sum_\sumstack{2a+b=2r\\ a > \delta r}\sigma_K(a)\sum_\sumstack{\deg_K(\mathcal{B})=b\\ \mathcal{B}\in \mathrm{Cl}(K)^2}\sum_{\deg_K(\mathcal{P})=m} \chi_{L(\mathrm{sqf}(\mathcal{B}),u)/K}(\mathcal{P}).
\end{equation}

The field $L(\mathrm{sqf}(\mathcal{B}),u)$ equals $K$ only if $\mathrm{sqf}(\mathcal{B})$ is the trivial divisor, i.e. when $\mathcal{B}$ is a square. Furthermore, in this case only one of the fields $L(\mathrm{sqf}(\mathcal{B}),u)$ as $u$ ranges over the elements in $S(K)/\fq^*$ equals $K$. The other fields are quadratic extensions of $K$, which are in fact also geometric (this is why we exclude $\fq^*$).

Next, we wish to average the character above over $\mathcal{P}$. For this, we use the following lemma.
\begin{lemma}
    Let $L$ be a geometric quadratic extension of $K$, with $g_K\ll n$, and $\chi_{L/K}$ the associated Kronecker character. Then,
    \begin{equation*}
        \sum_{\deg_K(\mathcal{P})=m}\chi_{L/K}(\mathcal{P}) \ll q^{m/2}q^{o_q(n)}.
    \end{equation*}
\end{lemma}
\begin{proof}
    We have that
    \begin{equation*}
    \begin{split}
        \sum_{\deg_K(\mathcal{P})=m}&\chi_{L/K}(\mathcal{P}) = \frac{1}{m}\sum_{\deg_K(\mathcal{P})=m}\deg_K(\mathcal{P})\chi_{L/K}(\mathcal{P}) = \frac{1}{m}\sum_{\deg_K(\mathcal{P})\mid m}\deg_K(\mathcal{P})\chi_{L/K}(\mathcal{P}^{m/\deg_K(\mathcal{P})}) \\&- \frac{1}{m}\sum_\sumstack{\deg_K(\mathcal{P})\mid m\\ \deg_K(\mathcal{P})\leq m/2}\deg_K(\mathcal{P})\chi_{L/K}(\mathcal{P}^{m/\deg_K(\mathcal{P})}) = -\frac{q^{m/2}c^{L/K}_m}{m}-\frac{1}{m}\sum_\sumstack{\deg_K(\mathcal{P})\mid m\\ \deg_K(\mathcal{P})\leq m/2}\deg_K(\mathcal{P}^{m/\deg_K(\mathcal{P})})\chi_{L/K}(\mathcal{P}),
    \end{split}
    \end{equation*}
    where the last step uses \eqref{cequality}. Now, the lemma follows from \eqref{cdef} and \eqref{primeidealct}.
\end{proof}
Using the lemma above, together with Lemma \ref{quadLemma}, Lemma \ref{mulemma} and Lemma \ref{divisorctlemma} we conclude that the contribution to \eqref{TsumtoPave} from nontrivial geometric field extensions of $K$ is
\begin{equation*}
    \ll \sum_{X^\alpha < q^{2j}\leq X^\beta}\sum_\sumstack{2a+b = 2r\\ a \geq \delta r}q^{o_q(n)}q^{m/2+2j+b+a/2} \ll \sum_{X^\alpha < q^{2j}\leq X^\beta}q^{o_q(n)}q^{m/2+2j+2n-2j-3\delta(n-j)/2}\ll \sum_{X^\alpha < q^{2j}\leq X^\beta}q^{o_q(n)}q^{m/2+2n-3\delta(n-j)/2}.
\end{equation*}
As $\beta < 1/2$, $n-j \geq n/2$, so that the bound above is $\ll q^{o_q(n)}q^{m/2+2n-3\delta n/4}$. Hence, also using Proposition \ref{reffieldct}, we see that the contribution to \eqref{onelev} is $\ll q^{o_q(n)-3\delta n/4}$, which is power-saving for large enough $q$.

What remains of $T$ is the trivial contribution coming from square $\mathcal{B}$, i.e. 
\begin{equation*}
    T' := 2\sum_\sumstack{j\geq 1\\ X^{\alpha}<q^{2j}\leq X^{\beta}} \sum_\sumstack{[K:\fq(T)]=2\\\abs{\Disc(K)}=q^{2j}}\sum_\sumstack{2a+b=2r\\ a > \delta r}\sigma_K(a)\sum_\sumstack{\deg_K(\mathcal{B})=b\\ \mathcal{B} \text{ square}}\sum_\sumstack{\deg_K(\mathcal{P})=m\\ \mathcal{P}\nmid \mathcal{B}}1.
\end{equation*}
The condition $\mathcal{P}\nmid \mathcal{B}$ can be removed at the cost of a power-saving error term. Using \eqref{primeidealct} the above becomes
\begin{equation*}
    2\sum_\sumstack{j\geq 1\\ X^{\alpha}<q^{2j}\leq X^{b}} \sum_\sumstack{[K:\fq(T)]=2\\\abs{\Disc(K)}=q^{2j}}\sum_\sumstack{2a+b=2r\\ a > \delta r}\sigma_K(a)\sum_\sumstack{\deg_K(\mathcal{B})=b\\ \mathcal{B} \text{ square}}\left(\frac{q^m}{m}+\mathcal{O}(q^{m/2}g_K)\right).
\end{equation*}
The contribution from the error term is negligible as well, as one sees similarly to how we bounded the contribution from the nontrivial characters.

In conclusion, we have shown that for large enough $q$, the contribution from the sum $T$ to \eqref{onelev} is, up to an error $o(1)$, equal to
\begin{equation*}
    T_1 := -\frac{4}{N\#\tilde{\mathcal{F}}_{\alpha,\beta}(X)}\sum_\sumstack{m\geq 1\\ m \text{ even}}q^{m/2}\widehat{\psi}\left(\frac{m}{N}\right)\sum_\sumstack{j\geq 1\\ X^{\alpha}<q^{2j}\leq X^{\beta}} \sum_\sumstack{[K:\fq(T)]=2\\\abs{\Disc(K)}=q^{2j}}\sum_\sumstack{2a+b=2r\\ a > \delta r}\sigma_K(a)\sum_\sumstack{\deg_K(\mathcal{B})=b/2}1.
\end{equation*}
We remark that the lower bound for $q$ depends so far only on $\delta > 0$. To conclude this section, we use that
\begin{equation*}
    \sum_\sumstack{2a+b=2r}\sigma_K(a)\sum_\sumstack{\deg_K(\mathcal{B})=b/2}1 = 0,
\end{equation*}
as this is the $r$th coefficient of the constant polynomial $1$. Hence, 
\begin{equation}\label{T1lastsec}
    T_1 = \frac{4}{N\#\tilde{\mathcal{F}}'_{\alpha,\beta}(X)}\sum_\sumstack{m\geq 1\\ m \text{ even}}q^{m/2}\widehat{\psi}\left(\frac{m}{N}\right)\sum_\sumstack{j\geq 1\\ X^{\alpha}<q^{2j}\leq X^{\beta}} \sum_\sumstack{[K:\fq(T)]=2\\\abs{\Disc(K)}=q^{2j}}\sum_\sumstack{2a+b=2r\\ a \leq \delta r}\sigma_K(a)\sum_\sumstack{\deg_K(\mathcal{B})=b/2}1.
\end{equation}

\subsection{Studying $S$}\label{studyOGS}
We now turn our attention to the complementary sum $S$ of $T$. By definition,
\begin{equation*}
    S := 
2\sum_\sumstack{j\geq 1\\ X^{\alpha}<q^{2j}\leq X^{\beta}} \sum_\sumstack{[K:\fq(T)]=2\\\abs{\Disc(K)}=q^{2j}}\sum_\sumstack{2a+b=2r\\ a \leq \delta r}\sigma_K(a)\sum_{\deg_K(\mathcal{P})=m} \mathbf{1}_{\abs{S(K)}=\abs{S_\mathcal{P}(K)}}\sum_\sumstack{\chi \in \widehat{\mathrm{Cl}_\mathcal{P}^0(K)}[2]\\ \chi \text{ primitive}}\sum_{\deg_K(\mathcal{B})=b}\chi(\mathcal{B}).
\end{equation*}
We apply the functional equation \eqref{fcnaleqHecke} to the innermost sum. Next, similar to how we arrived at \eqref{TsumtoPave}, we find that
\begin{equation*}
    S = 2\sum_\sumstack{j\geq 1\\ X^{\alpha}<q^{2j}\leq X^{\beta}} \sum_\sumstack{[K:\fq(T)]=2\\\abs{\Disc(K)}=q^{2j}}\sum_{u\in S(K)/\fq^*}\sum_\sumstack{2a+b=2r\\ a \leq  \delta r}\sigma_K(a)q^{b-g_K+1-m/2}\sum_\sumstack{\deg_K(\mathcal{B})=2g_K-2+m-b\\ \mathcal{B}\in \mathrm{Cl}(K)^2}\sum_{\deg_K(\mathcal{P})=m} \chi_{L(\mathrm{sqf}(\mathcal{B}),u)/K}(\mathcal{P}).
\end{equation*}
The contribution from nontrivial field extensions to this sum is 
\begin{equation*}
     \ll \sum_{X^\alpha < q^{2j}\leq X^\beta}\sum_\sumstack{2a+b = 2r\\ a \geq \delta r} q^{o_q(n)}q^{2j+\delta(n-j)/2+b-j}q^{2j+m-b} \ll q^mq^{3n\beta}q^{\delta n/2+o_q(n)}.
\end{equation*}
As $\widehat{\psi}$ is supported in $(-\sigma,\sigma)$ and $N = 2n+ \mathcal{O}(1)$, the total contribution from the nontrivial field extensions to \eqref{onelev} is
\begin{equation*}
    q^{n\sigma-2n+3n\beta+\delta n/2+o_q(n)} = q^{n(\sigma-2+3\beta + \delta/2+o_q(1))}.
\end{equation*}
As long as $\sigma < 2-3\beta$, we may pick some $\delta > 0$, depending only on $\sigma$, so that for large enough $q$, the above is power-saving as $n\to \infty$. The lower bound for $q$ depends only on how close $\sigma$ is to $2-3\beta$.

Similar to the previous section, the contribution from the trivial fields, which cannot directly be absorbed into the error term, is
\begin{equation}\label{sprimedef}
    S' = \frac{2q^m}{m}\sum_\sumstack{j\geq 1\\ X^{\alpha}<q^{2j}\leq X^{\beta}} \sum_\sumstack{[K:\fq(T)]=2\\\abs{\Disc(K)}=q^{2j}}\sum_\sumstack{2a+b=2r\\ a \leq  \delta r}\sigma_K(a)q^{b-g_K+1-m/2}\sum_\sumstack{\deg_K(\mathcal{B})=g_K-1+m/2-b/2}1.
\end{equation}
From the exact formula \eqref{exactdivct}, we know that
\begin{equation*}
    q^{b-g_K+1-m/2}\sum_\sumstack{\deg_K(\mathcal{B})=g_K-1+m/2-b/2}1 = \frac{q}{q-1}\sum_{k=0}^{2g_K}\mathbf{1}_{k\leq g_K-1+m/2-b/2}\cdot a_{k,K}\left(q^{b/2-k}-q^{b-g_K-m/2}\right).
\end{equation*}
Hence, 
\begin{equation*}
    S' = \frac{2q^{m+1}}{m(q-1)}\sum_\sumstack{j\geq 1\\ X^{\alpha}<q^{2j}\leq X^{\beta}} \sum_\sumstack{[K:\fq(T)]=2\\\abs{\Disc(K)}=q^{2j}}\sum_{k=0}^{2g_K}a_{k,K}\sum_\sumstack{2a+b=2r\\ b/2 >  r-\delta r\\ b/2\leq g_K-1+m/2-k}\sigma_K(a)\left(q^{b/2-k}-q^{b-g_K-m/2}\right),
\end{equation*}
so that $S'$ naturally splits into two sums, giving a contribution $S_1+S_2$ to \eqref{onelev}, corresponding to the first and second innermost terms in the expression above.

\subsection{Comparing $S_1$ and $T_1$}\label{firsts1t1sect}
In this section, we show that the sums $S_1$ and $T_1$ cancel each other, up to an acceptable error term. First, using the exact formula \eqref{exactdivct}, we may rewrite the innermost sum of \eqref{T1lastsec} according to
\begin{equation*}
    \sum_\sumstack{\deg_K(\mathcal{B})=b/2}1 = \frac{1}{q-1}\sum_{k=0}^{2g_K} \mathbf{1}_{k\leq b/2} \cdot a_{k,K}\left(q^{b/2+1-k}-1\right),
\end{equation*}
so that $T_1$ splits into two sub-sums corresponding to the two terms above. Using Lemma \ref{polylemma}, we find that
\begin{equation*}
    \sum_{k=0}^{2g_K} \abs{a_{k,K}} \ll q^{g_K}q^{o_q(n)},
\end{equation*}
whence a short computation shows that under the assumption $\sigma < 2-3\beta$, the second sub-sum gives only a contribution $o(1)$ to $T_1$, for large enough $q$.

We also use Lemma \ref{polylemma} to remove the indicator function from the first sub-sum. Indeed, 
\begin{equation*}
    \frac{1}{q-1}\sum_{k=0}^{2g_K} \mathbf{1}_{k\leq b/2} \cdot a_{k,K}q^{b/2+1-k} = \frac{1}{q-1}\sum_{k=0}^{2g_K} a_{k,K}q^{b/2+1-k}-\frac{1}{q-1}\sum_{b/2<k\leq 2g_K} a_{k,K}q^{b/2+1-k},
\end{equation*}
and the second sum is $\ll q^{g_K}q^{o_q(n)}$, similar to the sum estimated above, whence
\begin{equation*}
    T_1 \sim \frac{4q}{N\#\tilde{\mathcal{F}}'_{\alpha,\beta}(X)(q-1)}\sum_\sumstack{m\geq 1\\ m \text{ even}}q^{m/2}\widehat{\psi}\left(\frac{m}{N}\right)\sum_\sumstack{j\geq 1\\ X^{\alpha}<q^{2j}\leq X^{\beta}} \sum_\sumstack{[K:\fq(T)]=2\\\abs{\Disc(K)}=q^{2j}}\sum_\sumstack{2a+b=2r\\ a \leq \delta r}\sigma_K(a)\sum_{k=0}^{2g_K}a_{k,K}q^{b/2-k}.
\end{equation*}

On the other hand, 
\begin{equation*}
    S_1 = - \frac{4q}{N\#\tilde{\mathcal{F}}'_{\alpha,\beta}(X)(q-1)}\sum_\sumstack{m=1\\ \text{$m$ even}}^\infty q^{m/2}\widehat{\psi}\left(\frac{m}{N}\right)\sum_\sumstack{j\geq 1\\ X^{\alpha}<q^{2j}\leq X^{\beta}} \sum_\sumstack{[K:\fq(T)]=2\\\abs{\Disc(K)}=q^{2j}}\sum_\sumstack{2a+b=2r\\ b/2 > r-\delta r}\sigma_K(a)\sum_{k=0}^{2g_K}a_{k,K}q^{b/2-k}\mathbf{1}_{k\leq g_K-1+m/2-b/2},
\end{equation*}
so that, after also interchanging the order of summation, we have that
\begin{equation}\label{t1s1sumexpr}
\begin{split}
    T_1+S_1 \sim \frac{4q}{N\#\tilde{\mathcal{F}}'_{\alpha,\beta}(X)(q-1)}&\times\sum_\sumstack{m\geq 1\\ m \text{ even}}q^{m/2}\widehat{\psi}\left(\frac{m}{N}\right)\\&\sum_\sumstack{j\geq 1\\ X^{\alpha}<q^{2j}\leq X^{\beta}} \sum_\sumstack{[K:\fq(T)]=2\\\abs{\Disc(K)}=q^{2j}}\sum_{k=0}^{2g_K}a_{k,K}q^{-k}\sum_\sumstack{a+b/2=r\\ b/2 > r-\delta r\\ b/2 > g_K-1+m/2-k}\sigma_K(a)q^{b/2}.
\end{split}
\end{equation}
We may drop the second condition in the sum over $b$, for the contribution when the second condition is stronger than the third condition is power-saving. Indeed, in this case $m/2\leq r(1-\delta) + k-g_K+1$, and as $q^{-k}a_{k,K} \ll q^{-k/2}q^{o_q(n)}$, a computation shows that this gives a power-saving contribution.

It remains to bound
\begin{equation}\label{firstbd}
    \frac{4q}{N\#\tilde{\mathcal{F}}'_{\alpha,\beta}(X)(q-1)}\sum_\sumstack{m\geq 1\\ m \text{ even}}q^{m/2}\widehat{\psi}\left(\frac{m}{N}\right)\sum_\sumstack{j\geq 1\\ X^{\alpha}<q^{2j}\leq X^{\beta}} \sum_\sumstack{[K:\fq(T)]=2\\\abs{\Disc(K)}=q^{2j}}\sum_{k=0}^{2g_K}a_{k,K}q^{-k}\sum_\sumstack{a+b/2=r\\  b/2 \geq  g_K+m/2-k}\sigma_K(a)q^{b/2}.
\end{equation}
The innermost sum is the $r$th coefficient of
\begin{equation*}
    \frac{1}{\zeta_K(u)}\sum_{\ell=g_K+m/2-k}^\infty q^\ell u^\ell = \frac{(qu)^{g_K+m/2-k}(1-u)}{P_K(u)}, 
\end{equation*}
which we may write as
\begin{equation}\label{coeffint}
    \frac{1}{2\pi i}\int_{\abs{u}=q^{-3}}\frac{q^{g_K+m/2}(qu)^{-k}(1-u)}{u^{r+1-g_K-m/2}P_K(u)}du.
\end{equation}
Summing this over $k$ against $a_{k,K}q^{-k}$ yields
\begin{equation*}
    \frac{1}{2\pi i}\int_{\abs{u}=q^{-3}}\frac{q^{g_K+m/2}P_K(1/q^2u)(1-u)}{u^{r+1-g_K-m/2}P_K(u)}du.
\end{equation*}
By the functional equation for $P_K$, we have that
\begin{equation}\label{appliedfunceq}
 P_K(1/q^2u)= q^{-3g_K}u^{-2g_K}P_K(qu),   
\end{equation}
whence the above equals
\begin{equation*}
    \frac{1}{2\pi i}\int_{\abs{u}=q^{-3}}\frac{q^{m/2-2g_K}(1-u)}{u^{r+1+g_K-m/2}}\cdot \frac{P_K(qu)}{P_K(u)}du.
\end{equation*}
We now sum the above over $m$, writing $m=2m'$, to conclude that the expression \eqref{firstbd} equals
\begin{equation}\label{tosumm}
    \frac{4q}{N\#\tilde{\mathcal{F}}'_{\alpha,\beta}(X)(q-1)}\sum_\sumstack{j\geq 1\\ X^{\alpha}<q^{2j}\leq X^{\beta}} \sum_\sumstack{[K:\fq(T)]=2\\\abs{\Disc(K)}=q^{2j}}\frac{1}{2\pi i}\int_{\abs{u}=q^{-3}}\frac{q^{-2g_K}(1-u)}{u^{r+1+g_K}}\cdot \frac{P_K(qu)}{P_K(u)}\sum_{m'=1}^\infty (q^2u)^{m'}\widehat{\psi}\left(\frac{2m'}{N}\right) du.
\end{equation}

From the definition of the Fourier transform, we may write
\begin{equation*}
    \sum_{m'=1}^\infty (q^2u)^{m'}\widehat{\psi}\left(\frac{2m'}{N}\right) = \sum_{m'=1}^\infty (q^2u)^{m'}\int_{-\infty}^\infty \psi(x)e^{4\pi i x m'/N}dx = \int_{-\infty}^\infty \psi(x)\frac{q^2 u e^{4\pi i x/N}}{1-q^2ue^{4\pi i x/N}}dx,
\end{equation*}
as $\abs{u} = q^{-3}$, so that the sum converges absolutely. Hence, we may rewrite \eqref{tosumm} as
\begin{equation}\label{toshift}
    \frac{4q^3}{2\pi iN\#\tilde{\mathcal{F}}_{\alpha,\beta}(X)(q-1)}\int_{-\infty}^\infty \psi(x)\sum_\sumstack{j\geq 1\\ X^{\alpha}<q^{2j}\leq X^{\beta}} \sum_\sumstack{[K:\fq(T)]=2\\\abs{\Disc(K)}=q^{2j}}\int_{\abs{u}=q^{-3}}\frac{q^{-2g_K}(1-u)}{u^{r+g_K}}\cdot \frac{P_K(qu)}{P_K(u)}\cdot \frac{e^{4\pi i x/N}}{1-q^2ue^{4\pi i x/N}}dudx.
\end{equation}
We note that the $u$-integrand has a simple pole at one point on the circle $\abs{u}=q^{-2}$. We shift the integral to the line $\abs{u}=q^{-1}$. Then, the shifted integral is negligible. Indeed, from the product decomposition, we see that on this circle,
\begin{equation*}
    \frac{P_K(qu)}{P_K(u)}\ll q^{o_q(n)}q^{g_K},
\end{equation*}
so that the expression \eqref{toshift}, with the condition on $u$ replaced with $\abs{u}=q^{-1}$, is
\begin{equation*}
    \ll q^{o_q(n)}q^{-2n}\sum_\sumstack{j\geq 1\\ X^{\alpha}<q^{2j}\leq X^{\beta}} q^{r+2j} \ll q^{o_q(n)}\sum_\sumstack{j\geq 1\\ X^{\alpha}<q^{2j}\leq X^{\beta}} q^{j-n},
\end{equation*}
which is power-saving as $\beta < 1/2$.

When shifting, we also pick up a residue from the pole at some point $u_0$. Recall that $u_0$ depends on $x$ and that $\abs{u_0}=q^{-2}$. Hence, using Theorem \ref{thmfieldctall} and Proposition \ref{reffieldct}, we see that the contribution from the residue is 
\begin{equation*}
    \ll \frac{1}{N^2q^{2n}(\beta-\alpha)}\sum_\sumstack{j\geq 1\\ X^{\alpha}<q^{2j}\leq X^{\beta}} q^{2r} \max_{\abs{u_0} = q^{-2}}\bigg\lvert\sum_\sumstack{[K:\fq(T)]=2\\\abs{\Disc(K)}=q^{2j}}\frac{P_K(qu_0)}{P_K(u_0)}\bigg\rvert \ll \frac{1}{N}\sum_\sumstack{j\geq 1\\ X^{\alpha}<q^{2j}\leq X^{\beta}}  \max_{\abs{u_0} = q^{-2}}\bigg\lvert q^{-2j}\sum_\sumstack{[K:\fq(T)]=2\\\abs{\Disc(K)}=q^{2j}}\frac{P_K(qu_0)}{P_K(u_0)}\bigg\rvert.
\end{equation*}
We see that this is $o(1)$ if we can prove that
\begin{equation}\label{unifK}
    \max_{\abs{u_0} = q^{-2}}\bigg\lvert q^{-2j}\sum_\sumstack{[K:\fq(T)]=2\\\abs{\Disc(K)}=q^{2j}}\frac{P_K(qu_0)}{P_K(u_0)}\bigg\rvert \ll 1.
\end{equation}
This is studied using similar methods as in Section \ref{Kaversect}. We describe the process very briefly. Let us first define $u_0 =: q^{-2-ix_0}$. Using that $2g_K=2j-2$, one writes
\begin{equation*}
    \sum_\sumstack{[K:\fq(T)]=2\\\abs{\Disc(K)}=q^{2j}}\frac{P_K(qu_0)}{P_K(u_0)} = \sum_{\deg(A)\leq 2j-2}\sum_{B}\frac{\mu(B)}{\abs{A}^{1+ix_0}\abs{B}^{2+2ix_0}}\sum_\sumstack{[K:\fq(T)]=2\\\abs{\Disc(K)}=q^{2j}}\chi_K(AB),
\end{equation*}
and then uses Lemma \ref{KavLemma} to bound the innermost sum. As $\alpha \geq \eta_0>0$, the contribution from non-square $AB$ is now power-saving. The remaining $A,B$ provide a contribution that is $\ll 1$. This concludes our proof that $S_1+T_1 = o(1)$.
\subsection{Extracting the phase transition}
We now turn to the study of $S_2$. In particular, we will prove that it is this term that is responsible for the phase transition in the one-level density.

Recall that we defined
\begin{equation*}
    S_2 =  \frac{4q}{N\#\tilde{\mathcal{F}}'_{\alpha,\beta}(X)(q-1)}\sum_\sumstack{m=1\\ \text{$m$ even}}^\infty \widehat{\psi}\left(\frac{m}{N}\right)\sum_\sumstack{j\geq 1\\ X^{\alpha}<q^{2j}\leq X^{\beta}} \sum_\sumstack{[K:\fq(T)]=2\\\abs{\Disc(K)}=q^{2j}}\sum_\sumstack{2a+b=2r\\ b/2 \geq  r-\delta r}\sigma_K(a)\sum_{k=0}^{2g_K}a_{k,K}q^{b-g_K}\mathbf{1}_{k\leq g_K-1+m/2-b/2}.
\end{equation*}
We first study the two innermost sums using contour integration. Interchanging the order of summation yields
\begin{equation}\label{twoinnersum}
    \sum_{k=0}^{2g_K}a_{k,K}q^{-g_K}\sum_\sumstack{a+b/2=r\\ b/2 > r-\delta r\\ b/2\leq g_K-1+m/2-k}\sigma_K(a)q^{b}.
\end{equation}
Similar to before, the second condition in the sum above may be dropped, in which case the innermost sum above is the $r$th coefficient of 
\begin{equation*}
    \frac{1}{\zeta_K(u)}\sum_{\ell=0}^{g_K-1+m/2-k}q^{2\ell}u^\ell = \frac{(1-u)(1-qu)(1-q^{2g_K+m-2k}u^{g_K+m/2-k})}{P_K(u)(1-q^2u)}.
\end{equation*}
Hence, \eqref{twoinnersum} equals
\begin{equation}\label{S2innersum}
    \sum_{k=0}^{2g_K}a_{k,K}q^{-g_K}\cdot \frac{1}{2\pi i}\left(\int_{\abs{u}=q^{-3}}\frac{(1-u)(1-qu)(1-q^{2g_K+m-2k}u^{g_K+m/2-k})}{u^{r+1}P_K(u)(1-q^2u)}du\right).
\end{equation}
We shift the integral to the circle $\abs{u}=q^{-1}$. Note that there is no pole at $u=q^{-2}$ as the numerator has a zero at this point. We then split the integral into two parts
\begin{equation*}
    \sum_{k=0}^{2g_K}a_{k,K}q^{-g_K}\cdot \frac{1}{2\pi i}\left(\int_{\abs{u}=q^{-1}}\frac{(1-u)(1-qu)}{u^{r+1}P_K(u)(1-q^2u)}du-\int_{\abs{u}=q^{-1}}\frac{(1-u)(1-qu)q^{2g_K+m-2k}}{u^{r-g_K-m/2+k+1}P_K(u)(1-q^2u)}du\right).
\end{equation*}
Note that the first of the two terms above gives a power-saving contribution to $S_2$ and can be disregarded. 

To handle the second term above we split into two cases depending on whether $m\leq 2n$ or not. We may handle the case $m\leq 2n$ very similarly to how we bounded $S_1+T_1$. This provides only a term $o(1)$ to $S_2$ and we omit the details. What is more interesting is studying the contribution from $m > 2n$.

We first move the sum over $k$ inside the integral, obtaining
\begin{equation*}
    -\frac{1}{2\pi i}\int_{\abs{u}=q^{-1}}\frac{(1-u)(1-qu)q^{g_K+m}P_K(1/q^2u)}{u^{r-g_K-m/2+1}P_K(u)(1-q^2u)}du = -\frac{1}{2\pi i}\int_{\abs{u}=q^{-1}}\frac{(1-u)(1-qu)q^{-2g_K+m}P_K(qu)}{u^{r+g_K-m/2+1}P_K(u)(1-q^2u)}du.
\end{equation*}
Recall that $r+g_K = (n-j)+g_K = n-1 < n$. Hence, if $m > 2n$, the integrand does not have any pole at $u=0$. We may therefore shift the integral to $\abs{u} =q^{-3}$ and then use Cauchy's formula to conclude that the shifted integral equals zero. When shifting, we pick up a residue at $u=q^{-2}$ so that
\begin{equation*}
    -\frac{1}{2\pi i}\int_{\abs{u}=q^{-1}}\frac{(1-u)(1-qu)q^{-2g_K+m}P_K(qu)}{u^{r+g_K-m/2+1}P_K(u)(1-q^2u)}du = q^{2r}\cdot \frac{(1-q^{-2})(1-q^{-1})P_K(q^{-1})}{P_K(q^{-2})}.
\end{equation*}
Recalling that $2r=2n-2j$ and that $X=q^{2n}$, we have shown that
\begin{equation*}
    S_2 \sim \frac{4X(1-q^{-2})}{N\#\tilde{\mathcal{F}}'_{\alpha,\beta}(X)}\sum_\sumstack{m> 2n\\ \text{$m$ even}}^\infty \widehat{\psi}\left(\frac{m}{N}\right)\sum_\sumstack{j\geq 1\\ X^{\alpha}<q^{2j}\leq X^{\beta}} \sum_\sumstack{[K:\fq(T)]=2\\\abs{\Disc(K)}=q^{2j}}q^{-2j}\frac{P_K(q^{-1})}{P_K(q^{-2})}.
\end{equation*}
Recalling that $\tilde{\mathcal{F}}'_{\alpha,\beta}(X)\sim 2\tilde{\mathcal{F}}_{\alpha,\beta}(X)$, and comparing the above to the expression \eqref{sumnoflip} obtained when estimating $\tilde{\mathcal{F}}_{\alpha,\beta}(X)$, we see that
\begin{equation}\label{phasetrans}
    S_2\sim \frac{2}{N}\sum_\sumstack{m> 2n\\ \text{$m$ even}}^\infty \widehat{\psi}\left(\frac{m}{N}\right) = \frac{2}{N}\sum_\sumstack{m > 0\\ \text{$m$ even}}^\infty \widehat{\psi}\left(\frac{m}{N}\right)-\frac{2}{N}\sum_\sumstack{1\leq m\leq  2n\\ \text{$m$ even}} \widehat{\psi}\left(\frac{m}{N}\right) \sim \frac{\psi(0)}{2}-\frac{2}{N}\sum_\sumstack{1\leq m\leq  2n\\ \text{$m$ even}}\widehat{\psi}\left(\frac{m}{N}\right)\sim \frac{\psi(0)}{2}-\frac{1}{2}\int_{-1}^1 \widehat{\psi}(u)du,
\end{equation}
where the last step was Riemann summation and the fact that $2n/N \to 1$ as $n$ (and thus also $N$) tends to infinity. We have proven the following theorem.
\begin{prop}\label{smallalphabetaonelevdensthm}
    Let $\eta_0 >0$ be fixed, and let $\eta_0 < \alpha < \beta <1/2-\eta_0$. Then, if $\psi$ is an even Schwartz function whose Fourier transform is supported in $(-\sigma,\sigma)$, with $\sigma < 2-3\beta-o_q(1)$, we have that the one-level density
    \begin{equation*}
        \frac{1}{\#\mathcal{F}_{\alpha,\beta}(X)}\sum_{(L,K)\in\mathcal{F}_{\alpha,\beta}(X)}D_{L/K}(\psi)\sim \widehat{\psi}(0)-\frac{1}{2}\int_{-1}^1\widehat{\psi}(u)du,
    \end{equation*}
    consistent with the symplectic Katz--Sarnak prediction. The expression $D_{L/K}(\psi)$ is defined in \eqref{dlkdef}.
\end{prop}
A very slight extension of our results in this chapter, to allow for $\alpha =0$ proves Proposition \ref{quadprop}, see also Section \ref{nonvanishsect}.

\section{The one-level density: fields with large subfields}\label{largesubfldsect}
The goal of this section is to prove an analogue of Proposition \ref{smallalphabetaonelevdensthm}, where $\alpha,\beta > 1/2$. Specifically, we prove the following result.
\begin{prop}\label{bigsubfieldonelevthm}
    Let $\eta_0 >0$ be fixed, and let $1/2+\eta_0 < \alpha < \beta <1-\eta_0$. Then, if $\psi$ is an even Schwartz function whose Fourier transform is supported in $(-\sigma,\sigma)$, with $\sigma < 2-3(1-\alpha)-o_q(1)$, we have that the one-level density
    \begin{equation*}
        \frac{1}{\#\mathcal{F}_{\alpha,\beta}(X)}\sum_{(L,K)\in\mathcal{F}_{\alpha,\beta}(X)}D_{L/K}(\psi)\sim \widehat{\psi}(0)-\frac{1}{2}\int_{-1}^1\widehat{\psi}(u)du,
    \end{equation*}
    consistent with the symplectic Katz--Sarnak prediction. The expression $D_{L/K}(\psi)$ is defined in \eqref{dlkdef}.
\end{prop}
Proving an analogue of Proposition \ref{smallalphabetaonelevdensthm} with $\sigma$ bounded by $2-3\beta$ for $\alpha,\beta > 1/2$ is much simpler. The above theorem improves the support to $2-3(1-\alpha)$, which is significantly larger if $\alpha$ and $\beta$ are larger than $1/2$. 

The proof of Proposition \ref{bigsubfieldonelevthm} relies on utilising the flipped field of a $D_4$-pair $L/K$, similar to how we studied $\#\mathcal{F}_{\alpha,\beta}(X)$ for $\alpha, \beta \geq 1/2$. The starting point is a slightly modified version of \eqref{onelev} proven in the same way. Specifically, we have that the one-level density equals
\begin{equation}\label{flipstart}
    \widehat{\psi}(0)-\frac{\psi(0)}{2} - \frac{2}{N\#\tilde{\mathcal{F}}_{\alpha,\beta}(X)}\sum_{m=1}^\infty q^{-m/2}m\widehat{\psi}\left(\frac{m}{N}\right)\sum_{(L,K)\in \mathcal{F}_{\alpha,\beta}(X)}\sum_{\deg_K\mathcal{P}= m}\chi_{L/K}(\mathcal{P}) + o(1),
\end{equation}
whence our goal is to study the sum above. Similarly to before, we may also restrict the summation over $m$ to even positive integers.

\subsection{Splitting behaviour in the flipped field}
We would like to replace the summation over $(L,K)$ in \eqref{flipstart}, with the summation over their flipped fields $(L',K')$. In order to accomplish this, we need to relate splitting behaviours of primes over $L$ to splitting behaviour over $L'$. First, we rewrite the two innermost sums in \eqref{flipstart} according to
\begin{equation*}
    \sum_{(L,K)\in \mathcal{F}_{\alpha,\beta}(X)}\sum_{\deg_K\mathcal{P}= m}\chi_{L/K}(\mathcal{P}) = \sum_{\deg P \in \{m,m/2\}}\sum_{(L,K)\in \mathcal{F}_{\alpha,\beta}(X)}\sum_\sumstack{\mathcal{P}\mid P\\\deg_K\mathcal{P}= m}\chi_{L/K}(\mathcal{P}),
\end{equation*}
where $P$ ranges over the primes of $\fq(T)$. We are now interested in comparing the behaviour of
\begin{equation*}
    \sum_\sumstack{\mathcal{P}\mid P\\\deg_K\mathcal{P}= m}\chi_{L/K}(\mathcal{P})
\end{equation*}
to that of
\begin{equation*}
    \sum_\sumstack{\mathcal{P}\mid P\\\deg_{K'}\mathcal{P}= m}\chi_{L'/K'}(\mathcal{P}).
\end{equation*}
Here $\mathcal{P}$ ranges over primes in $K$ lying over $P$ in the first sum, and over primes in $K'$ lying over $P$ in the second sum. As before, the contribution from primes which are ramified in $L$ (or $L'$) is negligible. Moreover, the contribution from primes $P$ which are ramified in $K$ or $K'$ is negligible as well. 

Having excluded the contribution from such primes, we may use the first five rows of \cite[Table 1]{ASVW} to conclude that
\begin{equation}\label{flipeq}
\begin{split}
    \sum_{\deg P \in \{m,m/2\}}\sum_\sumstack{\mathcal{P}\mid P\\\deg_K\mathcal{P}= m}\chi_{L/K}(\mathcal{P}) &= \sum_{\deg P \in \{m,m/2\}}\sum_\sumstack{\mathcal{P}\mid P\\\deg_{K'}\mathcal{P}= m}\chi_{L'/K'} (\mathcal{P})  \\&+\sum_{\deg P = m/2}\bigg(\mathbf{1}_{L'/K' \text{ of type }(22,2)\text{ at $P$}}-\mathbf{1}_{L/K \text{ of type }(22,2) \text{ at $P$}}\bigg),
\end{split}
\end{equation}
where the equality is up to a negligible error term. Here, $L/K$ being of type $(22,2)$ at $P$ means that $P$ is inert in $K$ and that the prime lying above $P$ in $K$ splits in $L$. Now, again using \cite[Table 1]{ASVW}, we have that
\begin{equation*}
    \sum_{\deg P =m/2}\bigg(\mathbf{1}_{L'/K' \text{ of type }(22,2)\text{ at $P$}}-\mathbf{1}_{L/K \text{ of type }(22,2) \text{ at $P$}}\bigg) = \sum_{\deg P = m/2 } \left(\mathbf{1}_{L/K \text{ of type }(211,11)\text{ at $P$}}-\mathbf{1}_{L/K \text{ of type }(22,2) \text{ at $P$}}\right).
\end{equation*}
\subsection{Bounding the contribution from rogue splitting types}\label{roguech}
We now bound the contribution to \eqref{flipstart} from the last sum in \eqref{flipeq}, i.e. we study
\begin{equation}\label{fliprogue}
    \frac{2}{N\#\tilde{\mathcal{F}}'_{\alpha,\beta}(X)}\sum_\sumstack{m=1\\ m \text{ even}}^\infty q^{-m/2}m\widehat{\psi}\left(\frac{m}{N}\right)\sum_{\deg P = m/2}\sum_\sumstack{(L,K)\in \tilde{\mathcal{F}}'_{\alpha,\beta}(X)}\left(\mathbf{1}_{L/K \text{ of type }(22,2) \text{ at $P$}}-\mathbf{1}_{L/K \text{ of type }(211,11)\text{ at $P$}}\right).
\end{equation}
When $P$ lying over $\mathcal{P}$ is inert, then excluding the negligible contribution from ramified primes, we see that
\begin{equation}\label{easychar}
    \mathbf{1}_{L/K \text{ of type }(22,2) \text{ at $P$}} = \frac{1}{2}+\frac{1}{2}\chi_{L/K}(\mathcal{P}).
\end{equation}
Similarly, when $P = \mathcal{P}_1\mathcal{P}_2$ is split, we instead have that
\begin{equation}\label{diffchar}
    \mathbf{1}_{L/K \text{ of type }(211,11)\text{ at $P$}} = \frac{1}{2}-\frac{1}{2}\chi_{L/K}(\mathcal{P}_1\mathcal{P}_2).
\end{equation}
\subsubsection{Splitting types in $K$}
We first handle the contribution to \eqref{fliprogue} from the constant terms in \eqref{easychar} and \eqref{diffchar}. For this purpose, we study
\begin{equation*}
    \frac{1}{2}\sum_\sumstack{j\geq 1\\ X^{\alpha} < q^{2j}\leq X^\beta}\sum_\sumstack{[K:\fq(T)]=2\\ \abs{\Disc(K)}=q^{2j} \\ P \text{ inert in $K$}}\sum_\sumstack{[L:K]=2\\\abs{\mathrm{Disc}(L/K)}=X/q^{2j}}1-\frac{1}{2}\sum_\sumstack{j\geq 1\\ X^{\alpha} < q^{2j}\leq X^\beta}\sum_\sumstack{[K:\fq(T)]=2\\ \abs{\Disc(K)}=q^{2j} \\ P \text{ splits in $K$}}\sum_\sumstack{[L:K]=2\\\abs{\mathrm{Disc}(L/K)}=X/q^{2j}}1.
\end{equation*}
We may write this as 
\begin{equation}\label{constrogue}
    \frac{1}{2}\sum_\sumstack{j\geq 1\\ X^{\alpha} < q^{2j}\leq X^\beta}\sum_\sumstack{[K:\fq(T)]=2\\ \abs{\Disc(K)}=q^{2j} }\sum_\sumstack{[L:K]=2\\\abs{\mathrm{Disc}(L/K)}=X/q^{2j}}\chi_K(P).
\end{equation}
When $m$ is large, say $m/2 \geq g_K/10$, we sum the character $\chi_K(P)$ over $\deg P = m/2$. Specifically, we have that
\begin{equation*}
    \sum_{\deg P =m/2}\chi_K(P) \ll q^{m/4}q^{o_q(n)}.
\end{equation*}
As $\alpha > 1/2$ so that $g_K \geq n/2$, this provides a power-saving contribution to \eqref{fliprogue}.

When $m/2 < g_K/10$, we instead use Lemma \ref{quadLemma}, with $\mathcal{B} = 1$, to study \eqref{constrogue}. Up to a power-saving error that contributes $o(1)$ to \eqref{fliprogue}, \eqref{constrogue} is then bounded as
\begin{equation*}
    \ll q^{2n}\bigg \lvert  \sum_\sumstack{j\geq 1\\ X^{\alpha} < q^{2j}\leq X^\beta}q^{-2j}\sum_\sumstack{[K:\fq(T)]=2\\ \abs{\Disc(K)}=q^{2j} }    \frac{P_K(q^{-1})}{P_K(q^{-2})}\chi_K(P)  \bigg \rvert.
\end{equation*}
Now,
\begin{equation*}
    q^{-2j}\sum_\sumstack{[K:\fq(T)]=2\\ \abs{\Disc(K)}=q^{2j} }    \frac{P_K(q^{-1})}{P_K(q^{-2})}\chi_K(P) = q^{-2j}\sum_\sumstack{\deg(A)\leq 2j-2}\sum_{B}\frac{\mu(B)}{\abs{A}\abs{B}^2}\sum_\sumstack{[K:\fq(T)]=2\\ \abs{\Disc(K)}=q^{2j} } \chi_K(ABP).
\end{equation*}
As $\alpha > 0$, the contribution from $ABP$ non-square is power-saving. Hence, using Lemma \ref{KavLemma}, we need only consider
\begin{equation*}
    \sum_\sumstack{\deg(A)\leq 2j-2 \\ABP \text{ square}\\B}\frac{\mu(B)}{\abs{A}\abs{B}^2}\prod_{Q\mid ABP}\left(1+\frac{1}{\abs{Q}}\right)^{-1}.
\end{equation*}
We bound the tail sum by noting that
\begin{equation*}
    \sum_\sumstack{\deg(A) > 2j-2 \\ABP \text{ square}\\B}\frac{\mu^2(B)}{\abs{A}\abs{B}^2}\prod_{Q\mid AP}\left(1+\frac{1}{\abs{Q}}\right)^{-1} \ll \abs{P}^2\sum_\sumstack{\deg(A) > 2j-2 \\AB \text{ square}\\B}\frac{\mu^2(B)}{\abs{A}\abs{B}^2} \ll q^{m-j},
\end{equation*}
which is power-saving.

Finally, we compute that
\begin{equation*}
    \sum_\sumstack{A,B \\ABP \text{ square}}\frac{\mu(B)}{\abs{A}\abs{B}^2}\prod_{Q\mid ABP}\left(1+\frac{1}{\abs{Q}}\right)^{-1} = \left(1+\frac{1}{\abs{P}}\right)^{-1}\sum_\sumstack{A,B \\ABP \text{ square}}\frac{\mu(B)}{\abs{A}\abs{B}^2}\prod_\sumstack{Q\mid AB\\ Q\nmid P}\left(1+\frac{1}{\abs{Q}}\right)^{-1}.
\end{equation*}
Using multiplicativity, this equals
\begin{equation*}
\begin{split}
    \left(1+\frac{1}{\abs{P}}\right)^{-1}&\left(\sum_{\ell=0}^\infty \frac{1}{\abs{P}^{2\ell+1}}-\sum_{\ell=0}^\infty \frac{1}{\abs{P}^{2\ell}\abs{P}^2}\right)\prod_{Q\nmid P}\left(1+\left(1+\frac{1}{\abs{Q}}\right)^{-1}\left(-\sum_{\ell=0}^\infty \frac{1}{\abs{Q}^{2\ell+1}\abs{Q}^2}+\sum_{\ell=1}^\infty \frac{1}{\abs{Q}^{2\ell}}\right)\right)
    \\& \ll \abs{P}^{-1},
\end{split}
\end{equation*}
which in total gives a contribution $\mathcal{O}\left(N^{-1}\right)$ to \eqref{constrogue}.

\subsubsection{Averages of characters}

We now study the contribution to \eqref{fliprogue} from the characters in \eqref{easychar} and \eqref{diffchar}. The character $\chi_{L/K}(\mathcal{P})$ is easiest to handle. Indeed, we can proceed as in Section \ref{classgpsect} to write, up to a negligible error,
\begin{equation*}
    \sum_\sumstack{(L,K)\in \tilde{\mathcal{F}}'_{\alpha,\beta}(X)\\ P \text{ inert in $K$}}\chi_{L/K}(\mathcal{P}) = \frac{1}{2}\sum_\sumstack{[K:\fq(T)]=2\\P \text{ inert in $K$} \\ \abs{\Disc(K)}=q^{2j}\\X^\alpha \leq q^{2j}< X^{\beta}}\abs{S(K)}\mathbf{1}_{S(K) = S_\mathcal{P}(K)}\sum_\sumstack{\chi \in \widehat{\mathrm{Cl}_\mathcal{P}^0(K)}[2]\\ \chi \text{ primitive}}\sum_\sumstack{\mathcal{A} \text{ squarefree} \\ \mathcal{A}\in \mathrm{Cl}(K)^2 \\ \deg_K(\mathcal{A}) = 2n-2j }\chi(\mathcal{A}).
\end{equation*}
Here, we may use contour integration similar to the methods used in Section \ref{countfields} to bound the innermost sum by $\ll q^{n-j}q^{o_q(n)}$. Together with \eqref{primeidealct}, this proves that the contribution from the sum over the character $\chi_{L/K}(\mathcal{P})$ to \eqref{fliprogue} is $o(1)$.

Next, we wish to bound
\begin{equation}\label{doubleprime}
    \sum_\sumstack{(L,K)\in \tilde{\mathcal{F}}'_{\alpha,\beta}(X)\\ P \text{ splits in $K$}}\chi_{L/K}(\mathcal{P}_1\mathcal{P}_2).
\end{equation}
Up to a negligible error term, this is
\begin{equation*}
    \frac{1}{2}\sum_\sumstack{[K:\fq(T)]=2\\P \text{ splits in $K$} \\ X^\alpha \leq q^{2j}< X^{\beta}\\ \abs{\Disc(K)}=q^{2j}}\sum_\sumstack{\mathcal{A} \text{ squarefree} \\ \mathcal{A}\in \mathrm{Cl}(K)^2 \\ \deg_K(\mathcal{A}) = 2n-2j }\sum_{u\in S(K)}\chi_{L(\mathcal{A},u)/K}(\mathcal{P}_1\mathcal{P}_2).
\end{equation*}
The innermost sum equals zero, unless $S_{\mathcal{P}_1}(K) = S_{\mathcal{P}_2}(K)$, so we may assume that this is the case. The above, then equals
\begin{equation}
    \frac{\abs{S(K)}}{2}\sum_\sumstack{[K:\fq(T)]=2\\P \text{ splits in $K$} \\ X^\alpha \leq q^{2j}< X^{\beta}\\ \abs{\Disc(K)}=q^{2j}}\mathbf{1}_{S_{\mathcal{P}_1}(K) = S_{\mathcal{P}_2}(K)}\sum_\sumstack{\mathcal{A} \text{ squarefree} \\ \mathcal{A}\in \mathrm{Cl}(K)^2 \\ \deg_K(\mathcal{A}) = 2n-2j }\chi_{L(\mathcal{A})/K}(\mathcal{P}_1\mathcal{P}_2).
\end{equation}

Recall that $\chi_{L(\mathcal{A})/K}(\mathcal{P}_1\mathcal{P}_2)$ equals $1$ if both $\mathcal{P}_1$ and $\mathcal{P}_2$ have the same splitting type in $L$ (excluding the ramified cases as usual), and else equals $-1$. Now, using Lemma \ref{splitlemma}, we see that $\mathcal{A}$ being a square in $\mathrm{Cl}_{\mathcal{P}_1\mathcal{P}_2}$ is equivalent to there existing an extension of discriminant $\mathcal{A}$, such that both $\mathcal{P}_i$ splits here, whence $\chi_{L(\mathcal{A})/K}(\mathcal{P}_1\mathcal{P}_2) = 1$.

If we instead assume that $\mathcal{A}$ is not a square in $\mathrm{Cl}_{\mathcal{P}_1\mathcal{P}_2}$, then there is no extension where both primes split at the same time, and we then have three subcases. The first of these subcases is when $\mathcal{A}$ is a square in both $\mathrm{Cl}_{\mathcal{P}_i}$. Then $\mathcal{P}_1$ splits in half the extensions associated to $\mathcal{A}$, and $\mathcal{P}_2$ splits in the other half, so that $\chi_{L(\mathcal{A})/K}(\mathcal{P}_1\mathcal{P}_2) = -1$. Moreover, this also implies that $\abs{S_{\mathcal{P}_1}(K)}=\abs{S_{\mathcal{P}_2}(K)} = \abs{S(K)}/2$.

The second subcase is when $\mathcal{A}$ is not a square in either of the $\mathrm{Cl}_{\mathcal{P}_i}$. Then, both $\mathcal{P}_i$ are always inert, and $\chi_{L(\mathcal{A})/K}(\mathcal{P}_1\mathcal{P}_2) = 1$. Hence, we also have that $S_{\mathcal{P}_1}(K) = S_{\mathcal{P}_2}(K) = S(K)$. Finally, the last subcase is when $\mathcal{A}$ is a square in precisely one of the $\mathrm{Cl}_{\mathcal{P}_i}$. As above, we must have that $S_{\mathcal{P}_1}(K) = S_{\mathcal{P}_2}(K) = S(K)$, but now $\chi_{L(\mathcal{A})/K}(\mathcal{P}_1\mathcal{P}_2) = -1$, as one of the primes always splits, while the other one is always inert.

We now claim that
\begin{equation}\label{eqclaim}
    \mathbf{1}_{\mathcal{A}\in (\mathrm{Cl}^0(K))^2}\cdot\chi_{L(\mathcal{A})/K}(\mathcal{P}_1\mathcal{P}_2) = \frac{1}{\abs{\mathrm{Cl}^0(K)/(\mathrm{Cl}^0(K))^2}}\sum_\sumstack{\chi \in \widehat{\mathrm{Cl}_{\mathcal{P}_1\mathcal{P}_2}^0}(K)[2]\\ \chi \text{ primitive}}\chi(\mathcal{A}).
\end{equation}
Indeed, as every character is induced from a unique modulus, this sum equals
\begin{equation}\label{primsumcand}
    \frac{1}{\abs{\mathrm{Cl}^0(K)/(\mathrm{Cl}^0(K))^2}}\bigg(\sum_\sumstack{\chi \in \widehat{\mathrm{Cl}_{\mathcal{P}_1\mathcal{P}_2}^0}(K)[2]}\chi(\mathcal{A})-\sum_\sumstack{\chi \in \widehat{\mathrm{Cl}_{\mathcal{P}_1}^0}(K)[2]}\chi(\mathcal{A})-\sum_\sumstack{\chi \in \widehat{\mathrm{Cl}_{\mathcal{P}_2}^0}(K)[2]}\chi(\mathcal{A})+\sum_\sumstack{\chi \in \widehat{\mathrm{Cl}^0}(K)[2]}\chi(\mathcal{A})\bigg).
\end{equation}
Each of the four sums above is just a constant multiplied with the indicator function of $\mathcal{A}$ being a square in each class group. To prove the claim \eqref{eqclaim} we need some additional setup.

Recall that $\abs{\mathrm{Cl}^0(K)/(\mathrm{Cl}^0(K))^2} = \eta(K,\mathcal{P}_i)\abs{\mathrm{Cl}_{\mathcal{P}_i}^0(K)/(\mathrm{Cl}_{\mathcal{P}_i}^0(K))^2}$, where $\eta(K,\mathcal{P}_i)\in \{1,1/2\}$ is determined by $\eta(K,\mathcal{P}_1) = \abs{S(K)}/(2\abs{S_{\mathcal{P}_1}(K)}) = \eta(K,\mathcal{P}_2)$. Furthermore, similar to Lemma \ref{seqlemma}, we have an exact sequence
\begin{equation*}
        1 \to S_{\mathcal{P}_1}(K)\cap S_{\mathcal{P}_2}(K) \to S(K) \to \frac{\mathcal{O}_{\mathcal{P}_1}^*}{(\mathcal{O}_{\mathcal{P}_1}^*)^2}\times \frac{\mathcal{O}_{\mathcal{P}_2}^*}{(\mathcal{O}_{\mathcal{P}_2}^*)^2} \to \frac{\mathrm{Cl}^0_{\mathcal{P}_1\mathcal{P}_2}(K)}{(\mathrm{Cl}^0_{\mathcal{P}_1\mathcal{P}_2}(K))^2} \to \frac{\mathrm{Cl}^0(K)}{(\mathrm{Cl}^0(K))^2} \to 1.
    \end{equation*}
    As $S_{\mathcal{P}_1}(K) = S_{\mathcal{P}_2}(K)$, the intersection is equal to $S_{\mathcal{P}_1}(K)$, say. A diagram chase shows that
    \begin{equation*}
        \abs{\mathrm{Cl}^0(K)/(\mathrm{Cl}^0(K))^2} = \abs{\mathrm{Cl}_{\mathcal{P}_1\mathcal{P}_2}^0(K)/(\mathrm{Cl}_{\mathcal{P}_1\mathcal{P}_2}^0(K))^2}\frac{\abs{S(K)}}{4\abs{S_{\mathcal{P}_1}(K)}} = \abs{\mathrm{Cl}_{\mathcal{P}_1\mathcal{P}_2}^0(K)/(\mathrm{Cl}_{\mathcal{P}_1\mathcal{P}_2}^0(K))^2}\frac{\eta(K,\mathcal{P}_1)}{2}.
    \end{equation*}
    We then see that \eqref{primsumcand} equals
    \begin{equation}\label{explprimsum}
        \frac{2}{\eta(K,\mathcal{P}_1)}\mathbf{1}_{\mathcal{A}\in (\mathrm{Cl}_{\mathcal{P}_1\mathcal{P}_2}^0(K))^2}-\frac{1}{\eta(K,\mathcal{P}_1)}\mathbf{1}_{\mathcal{A}\in (\mathrm{Cl}_{\mathcal{P}_1}^0(K))^2}-\frac{1}{\eta(K,\mathcal{P}_1)}\mathbf{1}_{\mathcal{A}\in (\mathrm{Cl}_{\mathcal{P}_2}^0(K))^2}+\mathbf{1}_{\mathcal{A}\in (\mathrm{Cl}^0(K))^2}.
    \end{equation}
    Using this, we may now compare the expression above to the left-hand side of \eqref{eqclaim}. First, if $\mathcal{A}$ is not a square in the divisor class-group, then both sides are zero and we are done, so let us assume that $\mathcal{A}$ is a square. Next, we use the breakdown into cases and subcases described earlier in the section.

    First, if $\mathcal{A}$ is a square in $\mathrm{Cl}_{\mathcal{P}_1\mathcal{P}_2}$, then we saw earlier that $\chi_{L(\mathcal{A})/K}(\mathcal{P}_1\mathcal{P}_2) =1$. Moreover, evaluating \eqref{explprimsum} also yields $1$. We turn to the three subcases when $\mathcal{A}$ is not a square in $\mathrm{Cl}_{\mathcal{P}_1\mathcal{P}_2}$. In the first subcase when $\mathcal{A}$ was a square in the ray class groups modulo both $\mathcal{P}_i$, we had that $\chi_{L(\mathcal{A})/K}(\mathcal{P}_1\mathcal{P}_2) =-1$ and that $\eta(K,\mathcal{P}_i) = 1$. In this case, \eqref{explprimsum} evaluates to $-1$, as desired.

    In the second subcase, when $\mathcal{A}$ is not a square in either of the $\mathrm{Cl}_{\mathcal{P}_i}$, we have the equalities $\chi_{L(\mathcal{A})/K}(\mathcal{P}_1\mathcal{P}_2) =1$ and $\eta(K,\mathcal{P}_i)=1/2$. We see that \eqref{explprimsum} evaluates to $1$ as well. In the final case when $\mathcal{A}$ is a square in precisely one of the $\mathrm{Cl}_{\mathcal{P}_{i_0}}$, we also had that $\eta(K,\mathcal{P}_i)=1/2$ and that $\chi_{L(\mathcal{A})/K}(\mathcal{P}_1\mathcal{P}_2) =-1$. Lastly, \eqref{explprimsum} evaluates to $-1$.
    
We have shown that
\begin{equation*}
    \sum_\sumstack{(L,K)\in \tilde{\mathcal{F}'}_{\alpha,\beta}(X)\\ P \text{ splits in $K$}}\chi_{L/K}(\mathcal{P}_1\mathcal{P}_2) = \sum_\sumstack{[K:\fq(T)]=2\\P \text{ splits in $K$} \\ X^\alpha \leq q^{2j}< X^{\beta}\\ \abs{\Disc(K)}=q^{2j}}\mathbf{1}_{S_{\mathcal{P}_1}(K) = S_{\mathcal{P}_2}(K)}\sum_\sumstack{\chi \in \widehat{\mathrm{Cl}_{\mathcal{P}_1\mathcal{P}_2}^0}(K)[2]\\ \chi \text{ primitive}}\sum_\sumstack{\mathcal{A} \text{ squarefree} \\ \mathcal{A}\in \mathrm{Cl}(K)^2 \\ \deg_K(\mathcal{A}) = 2n-2j }\chi(\mathcal{A}).
\end{equation*}
Using the Riemann hypothesis and contour integration, the innermost sum can be bounded by $\ll q^{n-j+o_q(n)}$, giving a power-saving contribution to \eqref{fliprogue}. This concludes the bounding of the contribution from the rogue splitting types.
\subsection{Character sums in the flipped field}
We now study the contribution from the character sum in the flipped field. Specifically, we study
\begin{equation}\label{lastonelevtot}
    - \frac{2}{N\#\tilde{\mathcal{F}}_{\alpha,\beta}(X)}\sum_{m=1}^\infty q^{-m/2}m\widehat{\psi}\left(\frac{m}{N}\right)\sum_{(L,K)\in \mathcal{F}_{\alpha,\beta}(X)}\sum_{\deg_{K'}\mathcal{P}= m}\chi_{L'/K'}(\mathcal{P}) .
\end{equation}
Similar to how we arrived at \eqref{flipfieldst}, we see that up to a small error coming from non $D_4$-fields, the two innermost sums above equal
\begin{equation}\label{sumfg}
    \frac{1}{2}\sum_{f\in D_{\fq(T)}}\mu^2(f)\sum_\sumstack{g\in D_{\fq(T)}\\ (g,f)=1}\mu(g)\sum_\sumstack{j\geq 1\\  X^{1-\beta}/\abs{f}^2\leq q^{2j} < X^{1-\alpha}/\abs{f}^2}\sum_\sumstack{[K':\fq(T)]=2\\ \abs{\Disc(K')}=q^{2j}  \\ \text{$K'$ unramified at $fg$,}}\sum_{\deg_{K'}\mathcal{P}= m}\sum_\sumstack{[L':K']=2\\ \abs{\mathrm{Disc}(L'/K')}=X/q^{2j}\\ f^2g^2\mid N_{K'/\fq(T)}(\Disc(L'/K'))}\chi_{L'/K'}(\mathcal{P}).
\end{equation}
As before, the innermost sum equals
\begin{equation*}
    \sum_\sumstack{\mathcal{A} \text{ squarefree} \\ \mathcal{A}\in \mathrm{Cl}(K')^2 \\ \deg_{K'}(\mathcal{A}) = 2n-2j\\ fg\mid \mathcal{A} }\sum_{u\in S(K')}\chi_{L(\mathcal{A},u)/K'}(\mathcal{P}),
\end{equation*}
where $fg\mid \mathcal{A}$ means that all primes in $K$ dividing $fg$ also divides $\mathcal{A}$. As in Section \ref{classgpsect}, we need only consider the case when $S_\mathcal{P}(K) = S(K)$, and $m$ even, so that the splitting behaviour does not depend on $u$. We arrive at
\begin{equation*}
    \sum_\sumstack{\mathcal{A} \text{ squarefree} \\ \mathcal{A}\in \mathrm{Cl}(K')^2 \\ \deg_{K'}(\mathcal{A}) = 2n-2j\\ fg\mid \mathcal{A} }\sum_{u\in S(K')}\chi_{L(\mathcal{A},u)/K'}(\mathcal{P}) = 2\cdot \mathbf{1}_{\abs{S(K')}=\abs{S_\mathcal{P}(K')}}\sum_\sumstack{\chi \in \widehat{\mathrm{Cl}_\mathcal{P}^0(K')}[2]\\ \chi \text{ primitive}}\sum_\sumstack{\mathcal{A} \text{ squarefree}  \\ \deg_{K'}(\mathcal{A}) = 2n-2j\\ fg\mid \mathcal{A} }\chi(\mathcal{A}).
\end{equation*}
Separating $fg$, we obtain
\begin{equation}\label{toshift2}
    2\cdot \mathbf{1}_{\abs{S(K')}=\abs{S_\mathcal{P}(K')}}\sum_\sumstack{\chi \in \widehat{\mathrm{Cl}_\mathcal{P}^0(K')}[2]\\ \chi \text{ primitive}}\chi(fg)\sum_\sumstack{\mathcal{A} \text{ squarefree}  \\ \deg_{K'}(\mathcal{A}) = 2n-2j-2\deg(fg)\\ (fg,\mathcal{A}) =1  }\chi(\mathcal{A}),
\end{equation}
where we used that $\deg_{K'}(fg) = 2\deg(fg)$. Using multiplicativity, we can see that the innermost sum (multiplied with $\chi(fg)$ is the $2n-2j-2\deg(fg)$th coefficient of 
\begin{equation}\label{lfcncoeff}
    \chi(fg)\big(1-u^{\deg(\mathcal{P})}\big)^{-1}\prod_{\mathcal{Q}\mid fg}\big(1+\chi(Q)u^{\deg{Q}}\big)^{-1}\frac{L(u,\chi)}{\zeta_K(u^2)}.
\end{equation}

For some of the remaining arguments, it will be convenient to be able to assume that $m$ is $\gg n$. We may indeed assume that this is the case, for if $m\leq cn$, for some small constant $c$, then we may bound the innermost sum of \eqref{toshift2} by contour integration, shifting to e.g. $\abs{u}=q^{-3/4}$. The contribution from such $m$ to \eqref{lastonelevtot} is power-saving. Moreover, this also shows that we may restrict $f,g$ to $\abs{f},\abs{g} < X^{\delta}$, for any small $\delta >0$.

Now, for the remaining $m$, we may in fact ignore the first factor of \eqref{lfcncoeff}, as it only modifies the coefficients by a factor $1+\mathcal{O}\left(\abs{\mathcal{P}}^{-1}\right)$. Therefore, we instead study the coefficients of 
\begin{equation}\label{lfcnnoP}
    \chi(fg)\prod_{\mathcal{Q}\mid fg}\big(1+\chi(Q)u^{\deg{Q}}\big)^{-1}\frac{L(u,\chi)}{\zeta_K(u^2)}.
\end{equation}
The $2n-2j-2\deg(fg)$th coefficient of the expression above is (with $r=n-j$), up to a negligible error, equal to
\begin{equation*}
    \chi(fg)\sum_{v=0}^{2r-2\deg(fg)}\sum_\sumstack{\deg_{K'}(\mathcal{C})=v\\ \mathcal{C}\mid (fg)^\infty}\mu(\mathrm{sqf}(\mathcal{C}))\chi(\mathcal{C})\sum_{2a+b+v = 2r-2\deg(fg)}\sigma_{K'}(a)\sum_{\deg_{K'}(\mathcal{B})=b}\chi(\mathcal{B}).
\end{equation*}
At this point, we may add back the $m\leq cn$ to the sum over $m$. Similar to when we studied the one-level density for fields with small subfields, we split the summation depending on whether $\alpha \leq \delta r$ or not.
\subsubsection{Splitting the summation}
We now split the four innermost sums from \eqref{sumfg} into two parts $S$ and $T$. Let us write, cf \eqref{TsumtoPave},
\begin{equation}
\begin{split}
    T = 2&\sum_\sumstack{j\geq 1\\  X^{1-\beta}/\abs{f}^2\leq q^{2j} < X^{1-\alpha}/\abs{f}^2}\sum_\sumstack{[K':\fq(T)]=2\\ \abs{\Disc(K')}=q^{2j}  \\ \text{$K'$ unramified at $fg$,}}\\&\times\sum_{u\in S(K')/\fq^*}\sum_{v=0}^{2r-2\deg(fg)}\sum_\sumstack{\deg_{K'}(\mathcal{C})=v\\ \mathcal{C}\mid (fg)^\infty}\mu(\mathrm{sqf}(\mathcal{C}))\sum_\sumstack{2a+b+v=2r-2\deg(fg)\\ a > \delta r}\sigma_{K'}(a)\sum_\sumstack{\deg_{K'}(\mathcal{B})=b\\ \mathcal{B}\in \mathrm{Cl}(K')^2}\sum_{\deg_{K'}(\mathcal{P})=m} \chi_{L(\mathrm{sqf}(\mathcal{BC}fg),u)/K}(\mathcal{P}).
\end{split}
\end{equation}
When the innermost character is nontrivial, we can bound the sum over $\mathcal{P}$ by $\ll q^{m/2}q^{o_q(n)}$. The condition $a > \delta r$ ensures that the total contribution to \eqref{lastonelevtot} is power-saving. The character is trivial when $\mathcal{BC}fg$ is a square, and then only one of the $u$ corresponds to a trivial extension. For fixed $\mathcal{C},f,g$ this means that $\mathcal{B}$ is a square multiplied with the squarefree part of $\mathcal{C}fg \mid (fg)^\infty$.

The trivial contribution to $T$ is, using that $\mathcal{B}$ must then be a square, up to a negligible error term equal to
\begin{equation*}
\begin{split}
    T':= 2q^{m}m^{-1}&\sum_\sumstack{j\geq 1\\  X^{1-\beta}/\abs{f}^2\leq q^{2j} < X^{1-\alpha}/\abs{f}^2}\sum_\sumstack{[K':\fq(T)]=2\\ \abs{\Disc(K')}=q^{2j}  \\ \text{$K'$ unramified at $fg$,}}\\&\times\sum_{v=0}^{2r-2\deg(fg)}\sum_\sumstack{\deg_{K'}(\mathcal{C})=v\\ \mathcal{C}\mid (fg)^\infty}\mu(\mathrm{sqf}(\mathcal{C}))\sum_\sumstack{2a+b+v=2r-2\deg(fg)\\ a > \delta r}\sigma_{K'}(a)\sum_\sumstack{\deg_{K'}(\mathcal{B})=(b-2\deg(fg)+\deg(\mathrm{sqf}(\mathcal{C})))/2}1.
\end{split}
\end{equation*}
To simplify the notation, we write $c_0 := 2\deg(fg)-\deg(\mathrm{sqf}(\mathcal{C}))\geq 0$. Next, we note that if one removes the condition $a > \delta r$, then the three innermost sums become
\begin{equation*}
    \sum_\sumstack{\deg_{K'}(\mathcal{C})=v\\ \mathcal{C}\mid (fg)^\infty}\mu(\mathrm{sqf}(\mathcal{C}))\sum_\sumstack{2a+b=2r-v}\sigma_{K'}(a)\sum_\sumstack{\deg_{K'}(\mathcal{B})=(b-c_0)/2}1 =\sum_\sumstack{\deg_{K'}(\mathcal{C})=v\\ \mathcal{C}\mid (fg)^\infty}\mu(\mathrm{sqf}(\mathcal{C}))\sum_\sumstack{2a=0}^{2r-2\deg(fg)-v}\sigma_{K'}(a)\sum_\sumstack{\deg_{K'}(\mathcal{B})=(2r-v-2a-c_0)/2}1. 
\end{equation*}
We now simply note that if $2a > 2r-v-c_0$, then the innermost sum above is empty. Assuming that this is not the case, the two innermost sums above equal the $r-(v+c_0)/2$th coefficient of the constant polynomial $1$, i.e. $0$, unless $v$ is large, but the case when $v$ is large contributes only a negligible error. We then see that up to a small error $T'$ equals
\begin{equation*}
\begin{split}
    T'':=-2q^{m}m^{-1}&\sum_\sumstack{j\geq 1\\  X^{1-\beta}/\abs{f}^2\leq q^{2j} < X^{1-\alpha}/\abs{f}^2}\sum_\sumstack{[K':\fq(T)]=2\\ \abs{\Disc(K')}=q^{2j}  \\ \text{$K'$ unramified at $fg$,}}\\&\times\sum_{v=0}^{2r-2\deg(fg)}\sum_\sumstack{\deg_{K'}(\mathcal{C})=v\\ \mathcal{C}\mid (fg)^\infty}\mu(\mathrm{sqf}(\mathcal{C}))\sum_\sumstack{2a+b+v=2r-2\deg(fg)\\ a \leq \delta r}\sigma_{K'}(a)\sum_\sumstack{\deg_{K'}(\mathcal{B})=(b-c_0)/2}1.
\end{split}
\end{equation*}

If we apply \eqref{exactdivct}, we find that
\begin{equation}\label{firstdivneglscnd}
    \sum_\sumstack{\deg_{K'}(\mathcal{B})=(b-c_0)/2}1 = \frac{q}{q-1}\sum_{k=0}^{2g_{K'}}a_{k,K}\left(q^{(b-c_0)/2-k}-q^{-1}\right)\mathbf{1}_{\{k\leq (b-c_0)/2\}}.
\end{equation}
As $\sigma < 2-3(1-\alpha)$, the contribution of the second term above to \eqref{lastonelevtot} is power-saving. The remaining contribution of $T''$ to \eqref{lastonelevtot} is
\begin{equation*}
\begin{split}
    T_1 :=  &\frac{2q}{(q-1)N\#\tilde{\mathcal{F}}_{\alpha,\beta}(X)}\sum_\sumstack{m=1\\ \text{$m$ even}}^\infty q^{m/2}\widehat{\psi}\left(\frac{m}{N}\right)\sum_\sumstack{f\in D_{\fq(T)}\\ \abs{f} < X^{\delta}}\mu^2(f)\sum_\sumstack{g\in D_{\fq(T)}\\ (g,f)=1\\ \abs{g} < X^\delta}\mu(g)\sum_\sumstack{j\geq 1\\  X^{1-\beta}/\abs{f}^2\leq q^{2j} < X^{1-\alpha}/\abs{f}^2}\sum_\sumstack{[K':\fq(T)]=2\\ \abs{\Disc(K')}=q^{2j}  \\ \text{$K'$ unramified at $fg$,}}
    \\& \times\sum_{v=0}^{2r-2\deg(fg)}\sum_\sumstack{\deg_{K'}(\mathcal{C})=v\\ \mathcal{C}\mid (fg)^\infty}\mu(\mathrm{sqf}(\mathcal{C}))\sum_\sumstack{2a+b+v=2r-2\deg(fg)\\ a \leq \delta r}\sigma_{K'}(a)\sum_{k=0}^{2g_{K'}}a_{k,K} q^{(b-c_0)/2-k}\mathbf{1}_{\{k\leq (b-c_0)/2\}}.
\end{split}
\end{equation*}
We note that we may remove the indicator function above at the cost of an acceptable error. The error is bounded similarly to the contribution from the secondary term in \eqref{firstdivneglscnd}. 
\subsubsection{Studying $S$}
We now turn to the sum
\begin{equation*}
\begin{split}
    S: = 2&\sum_\sumstack{j\geq 1\\  X^{1-\beta}/\abs{f}^2\leq q^{2j} < X^{1-\alpha}/\abs{f}^2}\sum_\sumstack{[K':\fq(T)]=2\\ \abs{\Disc(K')}=q^{2j}  \\ \text{$K'$ unramified at $fg$,}}\\&\times\sum_{u\in S(K')/\fq^*}\sum_{v=0}^{2r-2\deg(fg)}\sum_\sumstack{\deg_{K'}(\mathcal{C})=v\\ \mathcal{C}\mid (fg)^\infty}\mu(\mathrm{sqf}(\mathcal{C}))\sum_\sumstack{2a+b+v=2r\\ a \leq \delta r}\sigma_{K'}(a)\sum_{\deg_{K'}(\mathcal{P})=m} \sum_\sumstack{\deg_{K'}(\mathcal{B})=b\\ \mathcal{B}\in \mathrm{Cl}(K')^2}\chi_{L(\mathrm{sqf}(\mathcal{BC}fg),u)/K'}(\mathcal{P}).
\end{split}
\end{equation*}
Similar to Section \ref{studyOGS}, we expand into Hecke characters and use the functional equation to rewrite the above into
\begin{equation*}
\begin{split}
    S: = 2&\sum_\sumstack{j\geq 1\\  X^{1-\beta}/\abs{f}^2\leq q^{2j} < X^{1-\alpha}/\abs{f}^2}\sum_\sumstack{[K':\fq(T)]=2\\ \abs{\Disc(K')}=q^{2j}  \\ \text{$K'$ unramified at $fg$,}}\sum_{u\in S(K')/\fq^*}\\&  \times \sum_{v=0}^{2r-2\deg(fg)}\sum_\sumstack{\deg_{K'}(\mathcal{C})=v\\ \mathcal{C}\mid (fg)^\infty}\mu(\mathrm{sqf}(\mathcal{C}))\sum_\sumstack{2a+b+v=2r\\ a \leq \delta r}\sigma_{K'}(a)q^{b-g_{K'}+1-m/2}\sum_\sumstack{\deg_{K'}(\mathcal{B})=2g_{K'}-2+m-b\\ \mathcal{B}\in \mathrm{Cl}(K)^2}\sum_{\deg_{K'}(\mathcal{P})=m} \chi_{L(\mathrm{sqf}(\mathcal{BC}fg),u)/K'}(\mathcal{P}).
\end{split}
\end{equation*}
The contribution from nontrivial $L(\mathrm{sqf}(\mathcal{BC}fg),u)$ to the above is
\begin{equation*}
    \ll \sum_\sumstack{j\geq 1\\  X^{1-\beta}/\abs{f}^2\leq q^{2j} < X^{1-\alpha}/\abs{f}^2}q^{m/2}q^{j+m/2}q^{2j}q^{o_q(n)}\ll q^{m+o_q(n)}\frac{X^{3(1-\alpha)/2}}{\abs{f}^3},
\end{equation*}
which provides a contribution to \eqref{lastonelevtot} of size 
\begin{equation*}
    \ll X^{\delta}X^{\sigma/2-1}X^{3(1-\alpha)/2}X^{o_q(n)},
\end{equation*}
using that $\abs{f},\abs{g}\leq X^{\delta}$. As $\delta > 0$ was arbitrary, the above is power-saving if $\sigma < 2-3(1-\alpha)$, for large enough $q$.

The relevant contribution from trivial fields to $S$ is, cf. \eqref{sprimedef}, equal to
\begin{equation*}
\begin{split}
    S': = 2q^{m}m^{-1}&\sum_\sumstack{j\geq 1\\  X^{1-\beta}/\abs{f}^2\leq q^{2j} < X^{1-\alpha}/\abs{f}^2}\sum_\sumstack{[K':\fq(T)]=2\\ \abs{\Disc(K')}=q^{2j}  \\ \text{$K'$ unramified at $fg$,}}\sum_{u\in S(K')/\fq^*}\\&  \times \sum_{v=0}^{2r-2\deg(fg)}\sum_\sumstack{\deg_{K'}(\mathcal{C})=v\\ \mathcal{C}\mid (fg)^\infty}\mu(\mathrm{sqf}(\mathcal{C}))\sum_\sumstack{2a+b+v=2r\\ a \leq \delta r}\sigma_{K'}(a)q^{b-g_{K'}+1-m/2}\sum_\sumstack{\deg_{K'}(\mathcal{B})=g_{K'}-1+m/2-(b+c_0)/2}1.
\end{split}
\end{equation*}
Now, using \eqref{exactdivct},
\begin{equation*}
    q^{b-g_{K'}+1-m/2}\sum_\sumstack{\deg_{K'}(\mathcal{B})=g_{K'}-1+m/2-(b+c_0)/2}1 = \frac{q}{q-1}\sum_{k=0}^{2g_{K'}}a_{k,K}\left(q^{(b-c_0)/2}-q^{b-g_{K'}-m/2}\right)\mathbf{1}_{k\leq g_{K'}-1+m/2-(b+c_0)/2},
\end{equation*}
so that the contribution of $S'$ to \eqref{lastonelevtot} naturally splits into a sum $S_1+S_2$.
\subsubsection{Comparing $S_1$ and $T_1$}
We now prove that the sum $S_1+T_1$ is negligible. First, we have an analogue of \eqref{t1s1sumexpr}, specifically
\begin{equation}\label{flipanalogue}
\begin{split}
    S_1+T_1 \sim &\frac{2q}{(q-1)N\#\tilde{\mathcal{F}}_{\alpha,\beta}(X)}\sum_\sumstack{m=1\\ \text{$m$ even}}^\infty q^{m/2}\widehat{\psi}\left(\frac{m}{N}\right)\sum_\sumstack{f\in D_{\fq(T)}\\ \abs{f} < X^{\delta}}\mu^2(f)\sum_\sumstack{g\in D_{\fq(T)}\\ (g,f)=1\\ \abs{g} < X^\delta}\mu(g)\sum_\sumstack{j\geq 1\\  X^{1-\beta}/\abs{f}^2\leq q^{2j} < X^{1-\alpha}/\abs{f}^2}\sum_\sumstack{[K':\fq(T)]=2\\ \abs{\Disc(K')}=q^{2j}  \\ \text{$K'$ unramified at $fg$,}}
    \\& \times\sum_{v=0}^{2r-2\deg(fg)}\sum_\sumstack{\deg_{K'}(\mathcal{C})=v\\ \mathcal{C}\mid (fg)^\infty}\mu(\mathrm{sqf}(\mathcal{C}))\sum_{k=0}^{2g_{K'}}a_{k,K}q^{-k}\sum_\sumstack{2a+b+v=2r-2\deg(fg)\\ b/2\geq 2r-2\deg(fg)-v- \delta r \\ (b-c_0)/2 > g_{K'}-1+m/2-k-c_0}\sigma_{K'}(a) q^{(b-c_0)/2}.
\end{split}
\end{equation}
As before, we may disregard the second condition in the innermost sum above, as the contribution from the ranges where the second condition is stricter than the third is negligible. Before removing the second condition one can make the change of variables $b'=b-c_0$, and note that the second condition implies that both $b$ and $b'$ are large, so that we may sum over $b'$ instead. Next, one removes the second condition in the sum.

Recalling that the contribution from large $v$ to the sum above was negligible, we may write $\mathcal{C}=\mathcal{C}_1\mathcal{C}_2^2$, and sum over all $\mathcal{C}_1\mid fg$, $\mathcal{C}_2\mid (fg)^\infty$, in place of the sum over $\mathcal{C}$. Note that $\mathrm{sqf}(\mathcal{C})=\mathcal{C}_1$. 

The innermost sum is then
\begin{equation*}
    \sum_\sumstack{2a+b'=2r-4\deg(fg)-2\deg(\mathcal{C}_2) \\ b'/2 > g_{K'}-1+m/2-k-c_0}\sigma_{K'}(a) q^{b'/2} = \frac{q^{-c_0}}{2\pi i}\int_{\abs{u}=q^{-3}}\frac{q^{g_{K'}+m/2}(qu)^{-k}(1-u)}{u^{r-2\deg(fg)-\deg(\mathcal{C}_2)+1-g_{K'}-m/2+c_0}P_{K'}(u)}du,
\end{equation*}
cf. \eqref{coeffint}. Summing this over $k$ and $m$, against the appropriate factors from \eqref{flipanalogue} yields 
\begin{equation*}
    \frac{q^{2-c_0}}{2\pi i}\int_{-\infty}^\infty \psi(x)\int_{\abs{u}=q^{-3}}\frac{q^{-2g_{K'}}(1-u)P_{K'}(qu)}{u^{r-2\deg(fg)-\deg(\mathcal{C}_2)+g_{K'}+c_0}P_{K'}(u)}\cdot \frac{e^{4\pi i x/N}}{1-q^2ue^{4\pi i x/N}}dudx.
\end{equation*}
We may shift the integral over $u$ to the circle $\abs{u}=q^{-1}$. The shifted integral is negligible, so that we can focus on the contribution from the residue at $u=q^{-2}e^{-4\pi i x/N}$. The total contribution from the residue to \eqref{flipanalogue} is
\begin{equation}\label{finals1t1contrib}
\begin{split}
    \ll \frac{1}{Nq^{2n}}&\sum_{f\in D_{\fq(T)}}\mu^2(f)\sum_\sumstack{g\in D_{\fq(T)}\\ (g,f)=1}\mu^2(g)\frac{1}{\abs{fg}^2}\sum_\sumstack{j\geq 1\\  X^{1-\beta}/\abs{f}^2\leq q^{2j} < X^{1-\alpha}/\abs{f}^2}q^{2r}
    \\&\times\max_{\abs{u_0}=q^{-2}}\bigg\lvert q^{-2j}\sum_\sumstack{[K':\fq(T)]=2\\ \abs{\Disc(K')}=q^{2j}  \\ \text{$K'$ unramified at $fg$}} \frac{P_{K'}(qu_0)}{P_{K'}(u_0)}\sum_{\mathcal{C}_1\mid fg}\mu(\mathcal{C}_1)q^{\deg_{K'}(\mathcal{C}_1)}u_0^{\deg_{K'}(\mathcal{C}_1)}\sum_{\mathcal{C}_2\mid (fg)^\infty} u_0^{\deg(\mathcal{C}_2)} \bigg\rvert.
\end{split}
\end{equation}
Now, by multiplicativity
\begin{equation*}
    \sum_{\mathcal{C}_1\mid fg}q^{\deg_{K'}(\mathcal{C}_1)}u_0^{\deg_{K'}(\mathcal{C}_1)}\sum_{\mathcal{C}_2\mid (fg)^\infty} u_0^{\deg(\mathcal{C}_2)} = \prod_{\mathcal{P}\mid fg}\frac{\left(1-\abs{\mathcal{P}}^{-1-ix_0}_K\right)}{\left(1-\abs{\mathcal{P}}^{-2-ix_0}_K\right)}=\prod_{P\mid fg}\frac{\left(1-\abs{P}^{-1-ix_0}\right)}{\left(1-\abs{P}^{-2-ix_0}\right)}\cdot \frac{\left(1-\chi_{K'}(P)\abs{P}^{-1-ix_0}\right)}{\left(1-\chi_{K'}(P)\abs{P}^{-2-ix_0}\right)}.
\end{equation*}
Expanding this, together with $P_{K'}(qu_0)/P_{K'}(u_0)$ into a Dirichlet series allows us to estimate the maximum above by averaging $\chi_{K'}$ over $K'$. The condition that $K'$ is unramified at $fg$ is handled using Lemma \ref{KavLemma}. The computations are somewhat simplified by noting that
\begin{equation*}
    \frac{1}{P_K(u_0)}\prod_{P\mid fg}\left(1-\chi_{K'}(P)\abs{P}^{-2-ix_0}\right)^{-1} = \prod_{P\nmid fg}\left(1-\chi_{K'}(P)\abs{P}^{-2-ix_0}\right).
\end{equation*}
Performing the averaging shows that \eqref{finals1t1contrib} is $\ll 1/N$, cf. the computations at the end of Section \ref{firsts1t1sect}. 

\subsubsection{Extracting the phase transition}
We now turn to the sum $S_2$, defined by
\begin{equation*}
\begin{split}
    S_2:= &\frac{2q}{(q-1)N\#\tilde{\mathcal{F}}_{\alpha,\beta}(X)}\sum_\sumstack{m=1\\ \text{$m$ even}}^\infty \widehat{\psi}\left(\frac{m}{N}\right)\sum_{f\in D_{\fq(T)}}\mu^2(f)\sum_\sumstack{g\in D_{\fq(T)}\\ (g,f)=1}\mu(g)\sum_\sumstack{j\geq 1\\  X^{1-\beta}/\abs{f}^2\leq q^{2j} < X^{1-\alpha}/\abs{f}^2}\sum_\sumstack{[K':\fq(T)]=2\\ \abs{\Disc(K')}=q^{2j}  \\ \text{$K'$ unramified at $fg$,}}
    \\& \times\sum_{v=0}^{2r-2\deg(fg)}\sum_\sumstack{\deg_{K'}(\mathcal{C})=v\\ \mathcal{C}\mid (fg)^\infty}\mu(\mathrm{sqf}(\mathcal{C}))\sum_{k=0}^{2g_{K'}}a_{k,K}\sum_\sumstack{2a+b+v=2r-2\deg(fg)\\ b/2\geq 2r-2\deg(fg)-v- \delta r \\ (b-c_0)/2 \leq g_{K'}-1+m/2-k-c_0}\sigma_{K'}(a) q^{b-g_{K'}}.
\end{split}
\end{equation*}
As in the previous section, we may remove the second condition in the innermost sum and make the change of variables $b' = b-c_0$. Our goal is now to extract the phase transition from this sum.

Writing $\mathcal{C}=\mathcal{C}_1\mathcal{C}_2^2$ as before, the innermost sum above becomes
\begin{equation*}
    q^{c_0-g_{K'}}\sum_\sumstack{2a+b'=2r-4\deg(fg)-\deg_{K'}(\mathrm{sq}(\mathcal{C}))\\ b'/2 \leq g_{K'}-1+m/2-k-c_0}\sigma_{K'}(a) q^{b-c_0} = \frac{q^{c_0-g_{K'}}}{2\pi i}\int_{\abs{u}=q^{-3}}\frac{(1-u)(1-qu)(1-(q^2u)^{g_{K'}+m/2-k-c_0})}{u^{r-2\deg(fg)-\deg_{K'}(\mathcal{C}_2)+1}P_{K'}(u)(1-q^2u)}du,
\end{equation*}
cf. \eqref{S2innersum}. As before, we shift the integral to the circle $\abs{u}=q^{-1}$. We may then write the above as
\begin{equation*}
\begin{split}
    \frac{q^{c_0-g_{K'}}}{2\pi i}&\int_{\abs{u}=q^{-1}}\frac{(1-u)(1-qu)}{u^{r-2\deg(fg)-\deg_{K'}(\mathcal{C}_2)+1}P_{K'}(u)(1-q^2u)}du\\&-\frac{1}{2\pi i}\int_{\abs{u}=q^{-1}}\frac{(1-u)(1-qu)q^{g_{K'}+m-2k-c_0}u^{-k}}{u^{r-g_{K'}-m/2-2\deg(fg)-\deg_{K'}(\mathcal{C}_2)+c_0+1}P_{K'}(u)(1-q^2u)}du.
\end{split}
\end{equation*}
The first integral above provides only a power-saving contribution to $S_2$ and can be disregarded. Hence, we may focus our efforts on the second integral. The contribution from this integral when $m\leq 2n$ is $o(1)$, as can be seen using a similar argument as when we bounded $S_1+T_1$ in the previous section. We omit the details.

We turn to the case when $m>2n$. Summing the second integral above against $a_{k,K'}$ over $k$ yields
\begin{equation*}
\begin{split}
    -\frac{1}{2\pi i}\int_{\abs{u}=q^{-1}}\frac{(1-u)(1-qu)q^{-2g_{K'}+m-c_0}P_{K'}(qu)}{u^{r+g_{K'}-m/2-2\deg(fg)-\deg_{K'}(\mathcal{C}_2)+c_0+1}P_{K'}(u)(1-q^2u)}du.
\end{split}
\end{equation*}
Noting that $c_0-2\deg(fg) \leq 0$, we see that when $m> 2n$, the exponent of $u$ is positive, which means that the integrand does not have a pole at $u=0$. Hence, if we shift the integral to $\abs{u}=q^{-3}$, then the only contribution is from the residue at $u=q^{-2}$ so that
\begin{equation*}
\begin{split}
    -\frac{1}{2\pi i}\int_{\abs{u}=q^{-1}}\frac{(1-u)(1-qu)q^{-2g_{K'}+m-c_0}P_{K'}(qu)}{u^{r+g_{K'}-m/2-2\deg(fg)-\deg_{K'}(\mathcal{C}_2)+c_0+1}P_{K'}(u)(1-q^2u)}du = \frac{q^{2r}}{\abs{\mathcal{C}_1}\abs{\mathcal{C}_2}^2\abs{fg}^2}\cdot \frac{(1-q^{-2})(1-q^{-1})P_{K'}(q^{-1})}{P_{K'}(q^{-2})}.
\end{split}
\end{equation*}
Next, we note that
\begin{equation*}
    \sum_{\mathcal{C}_1\mid fg}\sum_{\mathcal{C}_2\mid (fg)^\infty}\frac{\mu(\mathcal{C}_1)}{\abs{\mathcal{C}_1}\abs{\mathcal{C}_2}^2} = \prod_{\mathcal{P}\mid fg}\frac{\left(1-\abs{\mathcal{P}}^{-1}\right)}{\left(1-\abs{\mathcal{P}^{-2}}\right)}=\prod_{\mathcal{P}\mid fg}\left(1+\frac{1}{\abs{\mathcal{P}}}\right)^{-1}.
\end{equation*}

We have shown that 
\begin{equation*}
\begin{split}
    S_2\sim  &\frac{2X(1-q^{-2})}{N\#\tilde{\mathcal{F}}_{\alpha,\beta}(X)}\sum_\sumstack{m>2n\\ \text{$m$ even}} \widehat{\psi}\left(\frac{m}{N}\right)\sum_{f\in D_{\fq(T)}}\mu^2(f)\sum_\sumstack{g\in D_{\fq(T)}\\ (g,f)=1}\mu(g)\sum_\sumstack{j\geq 1\\  X^{1-\beta}/\abs{f}^2\leq q^{2j} < X^{1-\alpha}/\abs{f}^2}q^{-2j}\\&\times \sum_\sumstack{[K':\fq(T)]=2\\ \abs{\Disc(K')}=q^{2j}  \\ \text{$K'$ unramified at $fg$}}\frac{P_{K'}(q^{-1})}{P_{K'}(q^{-2})}\prod_{\mathcal{P}\mid fg}\left(1+\frac{1}{\abs{\mathcal{P}}}\right)^{-1}.
\end{split}
\end{equation*}
Comparing with \eqref{Sumfieldflip}, we see that
\begin{equation*}
    S_2\sim \frac{2}{N}\sum_\sumstack{m> 2n\\ \text{$m$ even}} \widehat{\psi}\left(\frac{m}{N}\right) \sim \frac{\psi(0)}{2}-\frac{1}{2}\int_{-1}^1\widehat{\psi}(u)du,
\end{equation*}
cf. \eqref{phasetrans}. This concludes the proof of Proposition \ref{bigsubfieldonelevthm}.

\section{Non-vanishing of $P_{L/K}$}\label{nonvanishsect}
We now apply Proposition \ref{smallalphabetaonelevdensthm} and Proposition \ref{bigsubfieldonelevthm} to study the non-vanishing of the $L$-functions 
\begin{equation*}
    P_{L/K}(u) = \frac{\zeta_L(u)}{\zeta_K(u)},
\end{equation*}
with $L$ a $D_4$-field and $K$ its quadratic subfield. The starting point for our computations is an observation from \cite[Corollary 3]{Ozluk-Snyder}, namely that if $\psi$ is real, even and nonnegative with $\psi(0) \neq 0$, then 
\begin{equation*}
    \frac{1}{\#\mathcal{F}_{\alpha,\beta}(X)}\sum_{(L,K)\in \mathcal{F}_{\alpha,\beta}(X)}\mathbf{1}_{P_{L/K}(q^{-1/2})=0} \leq \frac{1}{2\psi(0)\#\mathcal{F}_{\alpha,\beta}(X)}\sum_{(L,K)\in \mathcal{F}_{\alpha,\beta}(X)}D_{L/K}(\psi),
\end{equation*}
with $D_{L/K}(\psi)$ defined in \eqref{dlkdef}. Indeed, 
\begin{equation*}
    D_{L/K}(\psi) = \sum_{\theta_{L/K}}\psi\left(N\frac{\theta_{L/K}}{2\pi}\right) \geq 2\cdot \mathbf{1}_{P_{L/K}(q^{-1/2})=0},
\end{equation*}
as every central zero has even multiplicity by self-duality. Now, if $\mathrm{supp}(\psi)\subseteq [-\sigma,\sigma]$, with $\sigma < \max\{2-3\beta,2-3(1-\alpha)\}$, then by Propositions \ref{smallalphabetaonelevdensthm} and \ref{bigsubfieldonelevthm}, we have for large enough $q$, that
\begin{equation}\label{nonvanisheq}
    \lim_{X\to \infty}\frac{1}{2\psi(0)\#\mathcal{F}_{\alpha,\beta}(X)}\sum_{(L,K)\in \mathcal{F}_{\alpha,\beta}(X)}D_{L/K}(\psi) = \frac{1}{2\psi(0)}\left(\widehat{\psi}(0)-\frac{1}{2}\int_{-1}^1\widehat{\psi}(u)du\right). 
\end{equation}
Let us write $\psi = \psi_\sigma$ for a test function satisfying the conditions above, including that $\mathrm{supp}(\psi) \subseteq [-\sigma,\sigma]$. Then, for every $\sigma< 2$, say, we wish to minimise the right-hand side in \eqref{nonvanisheq}, by finding an optimal $\psi_\sigma$.

Let $\epsilon > 0$ be arbitrary and write $\sigma(\alpha,\beta,\epsilon) = \max\{2-3\beta,2-3(1-\alpha)\}-\epsilon$. Recall also that we have fixed an arbitrary $\eta_0 >0$ earlier. Now, we pick numbers $0<\eta_0 =\alpha_1<\beta_1=\alpha_2<\beta_2=\alpha_3 <.....<\beta_m\leq 1-\eta_0< 1$. Then, recalling that $\mathcal{F}_{\alpha,\beta}(X)$ makes up a proportion $\beta-\alpha$ of $\mathcal{F}(X)$, we have for large enough $q$, depending only on $\eta_0$, that
\begin{equation*}
    \limsup_{X\to \infty}\frac{1}{\#\mathcal{F}(X)}\sum_{(L,K)\in \mathcal{F}_{\alpha,\beta}(X)}\mathbf{1}_{P_{L/K}(q^{-1/2})=0} \leq \sum_{i=1}^m \frac{1}{(\beta_i-\alpha_i)}\cdot \frac{1}{2\psi_{\sigma(\alpha_i,\beta_i,\epsilon)}(0)}\left(\widehat{\psi}_{\sigma(\alpha_i,\beta_i,\epsilon)}(0)-\frac{1}{2}\int_{-1}^1\widehat{\psi}_{\sigma(\alpha_i,\beta_i,\epsilon)}(u)du\right)+2\eta_0,
\end{equation*}
where the term $2\eta_0$ comes from the trivial bound that at most every $L$-function in $\mathcal{F}_{0,\eta_0}$ and $\mathcal{F}_{1-\eta_0,1}$ vanishes at the central point.

We now increase $m$ and pick the $\alpha_i$ and $\beta_i$ so that $\max_{i=1,...,m}\{\beta_i-\alpha_i\}\to 0$ as $m\to \infty$. As $\eta_0$ and $\epsilon>0$ is arbitrary, this proves that
\begin{equation}\label{nonvanishingeq}
    \limsup_{q\to \infty}\limsup_{X\to \infty}\frac{1}{\#\mathcal{F}(X)}\sum_{(L,K)\in \mathcal{F}_{\alpha,\beta}(X)}\mathbf{1}_{P_{L/K}(q^{-1/2})=0} \leq \int_{0}^1 \frac{1}{2\psi_{\sigma(x)}(0)}\left(\widehat{\psi}_{\sigma(x)}(0)-\frac{1}{2}\int_{-1}^1\widehat{\psi}_{\sigma(x)}(u)du\right)dx,
\end{equation}
with $\sigma(x) = \max\{2-3x,2-3(1-x)\}$.

It remains to optimise $\psi_\sigma$. A density argument allows us to extend the class of admissible $\psi$ to functions that are $L^1$, say, but not necessarily Schwartz. A close to optimal choice, used in \cite{Ozluk-Snyder}, is
\begin{equation}\label{nonoptchoice}
    \psi_\sigma(y) = \left(\frac{\sin \pi \sigma y}{\pi \sigma y}\right)^2,
\end{equation}
with Fourier transform 
\begin{equation*}
    \widehat{\psi_\sigma}(u) = \sigma^{-2}\max\{\sigma-\abs{u},0\}.
\end{equation*}
With this choice of test function, a computation shows that the right-hand side in \eqref{nonvanishingeq} becomes
\begin{equation*}
    \frac{\log(2)}{3}\leq  0.2311,
\end{equation*}
so that at least $76.89\%$ of the $P_{L/K}$ are non-vanishing at the central point.

As we mentioned above, the choice \eqref{nonoptchoice} is not fully optimal, at least for $\sigma > 1$. Optimal test functions were originally found by J. Vanderkam, see also \cite[Appendix A]{ILS}. These optimal test functions are also described very explicitly in \cite[Section 7]{Freeman}. 

For our purposes, it suffices to know the following facts from \cite{Freeman}, applicable when $1<\sigma \leq2$. First, the optimal choice for the expression 
\begin{equation*}
    \frac{1}{\psi_\sigma(0)}\left(\widehat{\psi}_{\sigma}(0)-\frac{1}{2}\int_{-1}^1\widehat{\psi}_{\sigma}(u)du\right)dx = \frac{1}{\langle 1,g_{\sigma'}\rangle},
\end{equation*}
for a certain $g_{\sigma'}$ with support in $[-\sigma',\sigma']$, with $\sigma' = \sigma/2$. In fact, equations \cite[(7.23), (7.29)]{Freeman} imply (after correcting minor errors in their statement) that
\begin{equation*}
    \frac{1}{2}\langle 1,g_{\sigma'}\rangle = g_{\sigma'}(0)-1,
\end{equation*}
with $g_\sigma$ explicitly given by
\begin{equation*}
    g_{\sigma'}(x) :=\frac{1}{\gamma_{\sigma'}}\begin{cases}
        \cos\left(\frac{1-\sigma'}{2}+\frac{\pi -1}{4}\right), \abs{x}\leq 1-\sigma',
        \\ \cos\left(\frac{\abs{x}}{2}+\frac{\pi -1}{4}\right), 1-\sigma'\leq \abs{x}\leq \sigma',
        \\ 0, \text{ else}.
    \end{cases}
\end{equation*}
Here,
\begin{equation*}
    \gamma_{\sigma'}  = \sigma' \cos\left(\frac{1-\sigma'}{2}+\frac{\pi-1}{4}\right)+2\sqrt{2}\sin\left(\frac{1-2\sigma'}{4}\right).
\end{equation*}
Using these formulae to numerically compute the integral in \eqref{nonvanisheq} yields an upper bound $\leq 0.2296$, so that at least $77.04\%$ of the $P_{L/K}$ are non-vanishing at the central point. We have proven the following theorem.
\begin{theorem}
    Let $q$ be a prime power larger than an absolute constant, coprime to $2$. Then, at least $77 \%$ of the relative Dedekind zeta functions $P_{L/K}(q^{-s})$ associated to $(L,K)\in \mathcal{F}(X)$ are non-vanishing at the central point $s=1/2$. More precisely, 
    \begin{equation*}
        \liminf_{X\to \infty} \frac{1}{\#\mathcal{F}_{\alpha,\beta}(X)}\sum_{(L,K)\in \mathcal{F}_{\alpha,\beta}(X)}\mathbf{1}_{P_{L/K}(q^{-1/2})\neq 0}\geq 0.77. 
    \end{equation*}
\end{theorem}


\addcontentsline{toc}{chapter}{Bibliography}
\begin{thebibliography}{}
\bibitem[ASVW]{ASVW} S. A. Altug, A. Shankar, I. Varma, and K. H. Wilson, "The number of $D_4$-fields ordered by conductor", \textit{J. Eur. Math. Soc. (JEMS)}, vol. 23, no. 8, pp. 2733–2785, 2021.
\bibitem[B]{Brumer} A. Brumer, "The average rank of elliptic curves I." \textit{Invent. Math.}, vol. 109, no. 3, pp. 445–472, 1992.
\bibitem[BCL]{BCL} S. Baluyot, V. Chandee and X. Li, "Low-lying zeros of a large orthogonal family of automorphic $L$-functions", Preprint 2023, arXiv:2310.07606.
\bibitem[BF]{BF}H. M. Bui, and A. Florea. “Zeros of quadratic Dirichlet $L$-functions in the hyperelliptic ensemble.” \textit{Trans. Amer. Math. Soc.}, vol. 370, no. 11, pp. 8013–8045, 2018.
\bibitem[BS]{BS} M. Bhargava and A. Shankar, "The average size of the 5-Selmer group of elliptic curves is 6, and the average rank is less than 1", Preprint, 2013, arXiv:1312.7859.
\bibitem[C]{Cohen} H. Cohen, Advanced topics in computational number theory., \textit{Grad. Texts in Math.}, vol. 193, New York, NY: Springer, 2000.
\bibitem[CDO]{CDyDO} H. Cohen, F. Diaz y Diaz, and M. Olivier, "Enumerating Quartic Dihedral Extensions of $\mathbb{Q}$", \textit{Compos. Math.}, vol. 133, no. 1, pp. 65-93, 2002.
\bibitem[Co]{Cornelissen}G. Cornelissen, "Two-torsion in the Jacobian of hyperelliptic curves over finite fields". \textit{Arch. Math.} vol. 77, pp. 241–246, 2001.
\bibitem[DPR]{DPR} S. Drappeau, K. Pratt, M. Radziwill, "One-level density estimates for Dirichlet $L$-functions with extended support", \textit{Algebra Number Theory}, vol. 17, no. 4, pp. 805-830, 2023.
\bibitem[D]{Durlanik} M. E. Durlanık, \textit{Non-vanishing and $1$-level density for Artin $L$-functions of $D_4$-fields}. Ph.D. dissertation, University of Toronto, 2023.
\bibitem[ELS]{ELS} J. S. Ellenberg, W. Li and M. Shusterman, "Nonvanishing of hyperelliptic zeta functions over finite fields", \textit{Algebra Number Theory}, vol. 14, no. 7, pp. 1895-1909, 2020.
\bibitem[Fm]{Freeman} J. Freeman, "Fredholm Theory and Optimal Test Functions for Detecting Central Point Vanishing Over Families of L-functions", Preprint, 2017, arXiv:1708.01588. 
\bibitem[Fd]{Friedrichsen}M. Friedrichsen, "A Secondary Term for $D_4$ Quartic Fields Ordered By Conductor", Preprint, 2021, arXiv: 2111.03982.
\bibitem[G]{Gao}P. Gao, "$n$-Level Density of the Low-lying Zeros of Quadratic Dirichlet L-Functions", \textit{Int. Math. Res. Not. IMRN}, vol. 2014, no. 6, pp. 1699–1728, 2014.
\bibitem[GZ1]{GZ1} P. Gao and L. Zhao, "One level density of low-lying zeros of quadratic and quartic Hecke $L$-functions", \textit{Canad. J. Math.}, vol. 72, no. 2, pp. 427–454, 2020.
\bibitem[GZ2]{GZ2} P. Gao and L. Zhao, "One level density of low-lying zeros of quadratic Hecke $L$-functions of imaginary quadratic number fields", \textit{J. Aust. Math. Soc.}, vol. 112, no. 2, pp. 170-192, 2022.
\bibitem[HZ]{HZ} W. Hansen, and A. Zanoli. "Counting $D_4$-field extensions by multi-invariants." Preprint, 2025, arXiv:2507.12342.
\bibitem[H]{Heath-Brown}D. R. Heath-Brown, "The average analytic rank of elliptic curves", \textit{Duke Math. J.} vol. 122, no. 3, 2004.
\bibitem[HR]{Hughes-Rudnick}C. P. Hughes and Z. Rudnick, “Linear statistics of low-lying zeros of L-functions,” \textit{Q. J. Math}., vol. 54, no. 3, pp. 309–333, 2003.
\bibitem[ILS]{ILS}H. Iwaniec, W. Luo, and P. Sarnak, “Low lying zeros of families of L-functions,”\textit{ Publ. Math. Inst. Hautes Etudes Sci}., vol. 91, pp. 55–131, 2000.
\bibitem[IS]{Iwaniec-Sarnak}H. Iwaniec, and P. Sarnak, "The non-vanishing of central values of automorphic L-functions and Landau-Siegel zeros", \textit{Israel J. Math}, vol. 120, pp. 155-177, 2000.
\bibitem[KS]{Katz-Sarnak}N. M. Katz and P. Sarnak, “Zeroes of zeta functions and symmetry,” \textit{Bull. Amer. Math. Soc.}, vol. 36, no. 1, pp. 1–26, 1999.
\bibitem[K]{Keliher} D. Keliher, "Enumerating $D_4$ quartics and a Galois group bias over function fields", \textit{J. Théor. Nombres Bordeaux}, vol. 34, no. 2, pp. 371-391, 2022.
\bibitem[L]{Li} W. Li, "Vanishing of Hyperelliptic L-Functions at the Central Point", \textit{J. Number Theory} vol. 191, pp. 85-103, 2018.
\bibitem[MT]{MT}K. J. McGown, and A. Tucker. "An improved error term for counting $D_4$‐quartic fields." \textit{Bull. Lond. Math. Soc.}, vol. 56, no. 9, pp. 2874-2885, 2024.
\bibitem[N]{Neukirch}J. Neukirch, Algebraic number theory., \textit{Grundlehren Math. Wiss.}, vol. 322, Berlin, Germany: Springer-Verlag, 1999.
\bibitem[ÖS]{Ozluk-Snyder}A. E. Özlük, and C. Snyder. "On the distribution of the nontrivial zeros of quadratic L-functions close to the real axis." \textit{Acta Arith.}, vol. 91, no. 3, pp. 209-228, 1999.
\bibitem[Ro]{Rosen}M. Rosen, Number theory in function fields., \textit{Grad. Texts in Math.}, vol. 210, New York, NY: Springer, 2002. 
\bibitem[Ru]{Rudnick}Z. Rudnick, “Traces of high powers of the Frobenius class in the hyperelliptic ensemble,” \textit{Acta Arith.}, vol. 143, no. 1, pp. 81-99, 2010.
\bibitem[SV]{SV} A. Shankar, and I. Varma, "Malle's Conjecture for Galois octic fields over $\mathbb{Q}$." Preprint, 2025, arXiv:2505.23690.
\bibitem[S]{S}K. Soundararajan, "Nonvanishing of quadratic
Dirichlet $L$-functions at $s = 1/2$.", \textit{Ann. of Math. (2)}, vol. 152, no. 2, pp. 447–488, 2000. 
\bibitem[Y1]{Young1}M. P. Young, “Low-lying zeros of families of elliptic curves”, \textit{J. Amer. Math. Soc}., vol. 19, no. 1, pp. 205–250, 2006.
\bibitem[Y2]{Young} M. P. Young, "On the Non-Vanishing of Elliptic Curve L-Functions at the Central Point."\textit{Proc. Lond. Math. Soc. (3)}, vol. 93, no. 1, pp. 1-42, 2006.

\end{thebibliography}
\end{document}